\documentclass[11pt]{amsart}
\usepackage{texmac,semmac,semtkz,cleveref,bm}
\input{Magma/by.inc} 

\def\MILLER{\begin{tikzpicture}[scale=1.2]
  \def\TreeEdgeSignDist{0.14}
  \def\vertlabeldist{0.3}
  \def\clusterdslabelscale{0.6}
  \def\graphdslabelscale{0.6}
  \def\SnakeWiggle{6pt}
  \def\SnakeAmplitude{0.6pt}

  \InfVertices
  \Vertex[x=2.2,y=3.7,L=\relax]{0}
  \Vertex[x=1.5,y=1.0,L=$-$]{6A}
  \BlueVertices
  \VertexLN[x=1.5,y=4.0,L=$2$]{1}{$\bm{w_1}$}
  \VertexLW[x=0.0,y=3.0,L=$\>\>$]{2}{$\bm{u_1}$}
  \VertexLNW[x=1.5,y=3.0,L=$\>\>$]{3}{$\bm{x}$}
  \VertexLE[x=3.0,y=3.0,L=$\>\>$]{4}{$\bm{u_2}$}
  \VertexLE[x=1.5,y=2.0,L=$\>\>$]{5}{$\bm{w_2}$}
  \VertexLW[x=0.0,y=0.5,L=$\>\>$]{7}{$\bm{u_3}$}
  \VertexLE[x=3.0,y=0.5,L=$\>\>$]{8}{$\bm{u_4}$}
  \YellowVertices
  \VertexLNW[x=1.5,y=1.0,L=$ $]{6}{$\bm{w_3}\>\>\>$}
  \VertexLNE[x=1.5,y=1.0,L=\relax]{6S}{$-$}

  \BlueEdges
  \Edge(1)(3)
  \YellowEdges
  \Edge(3)(0)
  \Edge(2)(3)
  \Edge(3)(4)
  \Edge(3)(5)
  \Edge(5)(6)
  \Edge(8)(6)
  \Edge(6)(7)

  \TreeEdgeSignW(1)(3){0.5}{2}
  \def\TreeEdgeSignDist{0.2}
  \TreeEdgeSignW(3)(0){0.8}{+}
  \def\TreeEdgeSignDist{0.14}
  \TreeEdgeSignN(2)(3){0.5}{+}
  \TreeEdgeSignS(2)(3){0.5}{5}
  \TreeEdgeSignN(3)(4){0.5}{-}
  \TreeEdgeSignS(3)(4){0.5}{5}
  \TreeEdgeSignW(3)(5){0.6}{2}
  \TreeEdgeSignE(3)(5){0.6}{+}
  \TreeEdgeSignW(5)(6){0.5}{2}
  \TreeEdgeSignS(6)(7){0.5}{6}
  \TreeEdgeSignS(6)(8){0.5}{6}

  \ESwapPerc3234{in=-120,out=-60}{0.35}
\end{tikzpicture}}

\def\WINDMILL{\begin{tikzpicture}[scale=1.3]
  \def\TreeEdgeSignDist{0.14}
  \def\GraphEdgeSignDist{0.1}
  \def\clusterdslabelscale{0.6}
  \def\graphdslabelscale{0.6}
  \def\SnakeWiggle{6pt}
  \def\SnakeAmplitude{0.6pt}

  \InfVertices
  \Vertex[x=2.2,y=3.6,L=\relax]{0+}
  \Vertex[x=2.0,y=3.8,L=\relax]{0-}
  \BlueVertices
  \def\vertlabeldist{0.22}
  \VertexLN[x=1.5,y=4.0,L=$2$]{1}{$\bm{v_1}$}
  \VertexLE[x=1.5,y=2.0,L=$\>\>$]{5}{$\bm{v_2}$}
  \def\vertlabeldist{0.3}
  \VertexLW[x=1.5,y=3.0,L=$\>\>$]{3}{$\bm{v_x}$}
  \VertexLNW[x=1.25,y=1,L=$\>\>$]{7+}{$\bm{v_3^+}$}
  \VertexLNE[x=1.75,y=0.5,L=$\>\>$]{7-}{$\bm{v_3^-}$}

  \BlueEdges
  \Edge(0+)(3)
  \Edge(0-)(3)
  \Edge(1)(3)
  \Edge(5)(7+)
  \Edge(5)(7-)
  \LoopW(3)
  \LoopW(3)
  \LoopE(3)
  \EdgeW(7)
  \EdgeE(7)
  \BendedEdges
  \Edge(3)(5)
  \Edge(5)(3)

  \VSwap{7+}{7-}{}2
  \VArr{$(3)-(0.4,-0.18)$}{$(3)+(0.5,0.2)$}{in=120}{1.5}
  \VArr{$(3)+(0.5,0.2)$}{$(3)-(0.4,0.2)$}{out=-90,in=-90}{1.5}
  
  \GraphEdgeSignE(0-)(3){0.1}{\>\>e_\infty^+}
  \GraphEdgeSignE(0+)(3){0.1}{\>\>e_\infty^-}
  \GraphEdgeSignW(1)(3){0.4}{e_1}
  \GraphEdgeSignE(1)(3){0.39}{1\>}
  \GraphEdgeSignW(5)(7+){0.48}{e_3^+}
  \GraphEdgeSignE(5)(7+){0.48}{1\>}
  \GraphEdgeSignW(5)(7-){0.5}{\>\,1}
  \GraphEdgeSignE(5)(7-){0.5}{\>\>e_3^-}
  \TreeSignAt(7+)(1.7,-0.17){6}
  \TreeSignAt(7+)(1.9,-0.17){\ell_4}
  \TreeSignAt(7-)(-1.7,0.17){6}
  \TreeSignAt(7-)(-1.9,0.17){\ell_3}
  \TreeSignAt(3)(-1.4,0){\ell_1}
  \TreeSignAt(3)(-1.2,0.01){5}
  \TreeSignAt(3)(1.4,0){\ell_2}
  \TreeSignAt(3)(1.2,0.01){5}
  \TreeSignAt(5)(-0.27,0.4){e_2^+}
  \TreeSignAt(5)(-0.07,0.4){1}
  \TreeSignAt(5)(0.08,0.5){1}
  \TreeSignAt(5)(0.3,0.5){e_2^-}
\end{tikzpicture}
}

\def\GRAINthree{\begingroup
\def\clusterdslabelscale{0.9}
\def\clustersep{3pt}
\def\clnodedist{0pt}
\clusterpicture[1.1]
  \Root(2.50,2)(r6);
  \Root(3.00,2)(r13);
  \Root(3.25,2)(r14);
  \ClusterDSName(c1){(r13)(r14)}{5\!/\!2}{+}{\t_1};
  \Root(3.91,2)(r15);
  \Root(4.16,2)(r16);
  \ClusterDSName(c2){(r15)(r16)}{5\!/\!2}{-}{\t_2};
  \Root(4.82,2)(r1);
  \Root(5.07,2)(r2);
  \Root(5.32,2)(r3);
  \Root(5.57,2)(r4);
  \Root(5.82,2)(r5);
  \ClusterDSName(c3){(r1)(r2)(r3)(r4)(r5)}{2}{}{\s_1};
  \Root(6.48,2)(r11);
  \Root(6.73,2)(r12);
  \Root(7.23,2)(r7);
  \Root(7.48,2)(r8);
  \ClusterDSName(c4){(r7)(r8)}{3}{-}{\t_3};
  \Root(8.14,2)(r9);
  \Root(8.39,2)(r10);
  \ClusterDSName(c5){(r9)(r10)}{3}{-}{\t_4};
  \ClusterDSName(c6){(c4)(c4n)(c4name)(c5)(c5n)(c5name)}{1}{-}{\s_3};
  \ClusterDSName(c7){(r11)(r12)(c6)(c6n)(c6name)}{1}{+}{\s_2};
  \ClusterDSName(c8){(r6)(c1)(c1n)(c1name)(c2)(c2n)(c2name)(c3)(c3n)(c3name)(c7)(c7n)(c7name)}{}{+}{X};
  \frob(c1)(c2);
\endclusterpicture\endgroup}
\begin{document}

\title{Semistable types of hyperelliptic curves}
\author{Tim and Vladimir Dokchitser, C\'eline Maistret, Adam Morgan}

\thanks{This research is supported by EPSRC grants EP/M016838/1 and EP/M016846/1}
\thanks{The second author is supported by a Royal Society University Research Fellowship.}




\begin{abstract}
In this paper, we explore three combinatorial descriptions of 
semistable types
of hyperelliptic curves over local fields: dual graphs, their quotient trees by
the hyperelliptic involution, and configurations of the roots of the defining equation
(`cluster pictures'). We construct explicit combinatorial one-to-one
correspondences between the three, which furthermore respect automorphisms and allow to keep track of the monodromy pairing and the Tamagawa group of the 
Jacobian. We introduce a classification scheme and a naming convention 
for semistable types of hyperelliptic curves and types with a Frobenius action. This is the higher genus analogue of the distinction between good, split and non-split multiplicative reduction for elliptic curves.
Our motivation is to understand $L$-factors, Galois representations, conductors, Tamagawa numbers and other local invariants of hyperelliptic curves and their Jacobians.
\end{abstract}
\maketitle
\clearpage
\tableofcontents
\clearpage

\section{Introduction}\label{s:intro}

\def\GraphScale{0.4}
\def\SnakeWiggle{3pt}
\def\clustersep{1.3pt}
\rootsize{0.5}

Suppose $K$ is a field with a discrete valuation, say of odd residue characteristic,
and $C/K$ is a hyperelliptic curve of genus $g$,
$$
  C: y^2 = f(x). 
$$
Our motivation is the study of the arithmetic of $C$ and its Jacobian, including its minimal model, Tamagawa number, $L$-factor, conductor and other invariants related to the Birch--Swinnerton-Dyer conjecture. It would be desirable to have a classification of reduction types in the fashion of Kodaira types for elliptic curves, which moreover would take into consideration non-algebraically closed residue fields. 
In order to do so, for semistable curves, this paper develops a correspondence between three
natural combinatorial objects attached to $C$ that control its arithmetic. The correspondence is explicit and gives a simple way to pass between these objects in practice.

\subsection{Correspondence}

First, $C$ has semistable reduction over some finite extension $F/K$:
it has a model over the ring of integers of $F$ with
stable special fibre $\bar C$. Thus, $\bar C$ has only ordinary double points as singularities, and 
$\Aut(\bar C)$ is finite (assuming $g\ge 2$ for the moment). 
Associated to $\bar C$ is its dual graph $G$, 
with a vertex for each geometric irreducible component, decorated with its genus
and an edge for each intersection. It is often referred 
to as a `semistable type' of~$C$; 
e.g., in genus~2 there are seven types (omitting genus 0 markings):

\begin{center}
\begin{tabular}{ccccccc}
\TwBG
\end{tabular}
\end{center}

Second, $G$ has an involution $\iota$ that comes from the hyperelliptic 
involution $y\mapsto -y$ on $C$, and the topological quotient $G/\langle \iota\rangle$ is a tree, 
say $T$.
It has genus markings on the vertices as well, and a natural 2-colouring:
colour points over which $G\to T$ is 2:1 yellow, and the branch locus blue. In 
genus~2, the corresponding trees are

\begin{center}
\def\GraphScale{0.33}
\begin{tabular}{ccccccc}
\TwBT
\end{tabular}
\end{center}

Third, the set $X\!\subset\!\bar K$ of the $2g\!+\!1$ or $2g\!+\!2$ roots of the 
defining polyno\-mial~$f(x)$
gives another natural combinatorial invariant --- how the roots `cluster' together.
Call a non-empty subset $\s\subset X$ a \emph{cluster} if it is of the form
$X\cap\text{(some disc in $\bar K$)}$, and view 
$X$ abstractly as a finite set with a collection $\Sigma$ of clusters $\s\subset X$, a 
\emph{cluster picture}. Different presentations $y^2=f(x)$ of the same curve may give
different cluster pictures; however there is an equivalence relation induced by M\"obius transformations of the roots. When $|X|=6$ ($g=2$), there are seven 
equivalence classes, represented by 

\begin{center}
\def\clpicscale{0.7}\rootsize{0.4}
\begin{tabular}{ccccccc}
\TwBS
\end{tabular}
\end{center}

\noindent
The leftmost one illustrates the 6 roots being all equidistant, in the next one the last two roots are closer to each other than to the other four, and so on.

\smallskip

The three sets of 7 pictures raise the obvious question, 
to which the answer turns out to be `Yes'. 
There is an established combinatorial notion of a `hyperelliptic graph'. In this paper, we
introduce `BY (blue/yellow) trees' and `cluster pictures', again in a 
combinatorial fashion, and formally define genus and equivalence. 
We then prove

\begin{theorem}[Main correspondence]\label{icorr1}
There is an explicit genus-preserving one-to-one correspondence between 
\begin{itemize}
\item Hyperelliptic graphs up to isomorphism,
\item BY trees up to isomorphism,
\item Cluster pictures up to equivalence.
\end{itemize}
\end{theorem}

In order to work with a fixed model of a hyperelliptic curve, it is also natural to ask
for a graph-theoretic counterpart of a cluster picture that determines it up to isomorphism, rather than up to equivalence. The right notions turn out to be 
\emph{open BY tree} / \emph{open hyperelliptic graph}, with one extra edge 
or $\iota$-orbit of edges with no endpoint (`going off to $\infty$'). 
An open hyperelliptic graph $G$ has a unique largest 
hyperelliptic subgraph, its \emph{core} $\tilde G$, and
$G, G'$ are called \emph{equivalent} if $\tilde G\iso \tilde G'$; similarly for 
BY trees $T$,~$\tilde T$.

The refined version of Theorem \ref{icorr1} is

\begin{theorem}[Open correspondence]\label{icorr2}
There is an explicit genus-preserving and equivalence-preserving 
one-to-one correspondence between 
\begin{itemize}
\item Open hyperelliptic graphs up to isomorphism,
\item Open BY trees up to isomorphism,
\item Cluster pictures up to isomorphism.
\end{itemize}
The construction is summarized in Table \ref{tabdicto}.
\end{theorem}

For example, here are open hyperelliptic graphs with core \OZGb,
open BY trees with core \OZTb, and the corresponding cluster pictures
that form one full equivalence class:

\vspace{-0.3em}

\begingroup
\def\GraphScale{0.35}
\begin{center}
\scalebox{0.9}{\hbox{
\begin{tabular}{ccccccc}
\OZG\cr 
\OZT\\[2pt]
\OZS\cr 
\end{tabular}}}
\end{center}
\endgroup

\subsection{Invariants}

The dual graph $G$ of a semistable curve carries 
several important invariants, notably

\begin{itemize}
\item 
A \emph{metric} (edges have length),
\item
\emph{Automorphisms} (e.g. coming from the action of Galois),
\item
Homology \emph{lattice} $\Lambda=H_1(G,\Z)$,
\item
Symmetric positive-definite \emph{pairing}%
\footnote{the monodromy pairing of Grothendieck \cite{Gro72}}
on $\Lambda$ induced by the metric,
\item
\emph{Tamagawa group}\footnote{also known in graph theory as 
the Jacobian (see \S\ref{ss:jacgraphs}), Picard group or the sandpile group of $G$;
for curves over local fields this is the group of connected components of the 
N\'eron model of the Jacobian, and the size of its Frobenius invariants 
is called the Tamagawa number}
$\Phi(G)=\Lambda^\vee/\Lambda$. 
\end{itemize}

All of these have counterparts for BY trees and cluster pictures:

\def\bu{\smallskip\noindent $\bullet$\hspace{0.5em}}

\bu
The analogue of an (open or not) \emph{metric} hyperelliptic graph $G$ / 
BY tree~$T$ (edges have length) is a \emph{metric cluster} picture $\Sigma$,
with distances between clusters.

\bu
For a BY tree $T$, the counterpart of $\Lambda_G=H_1(G,\Z)$ is 
the relative homology group $\Lambda_T=H_1(T,T_b,\Z)$ with respect
to the blue part $T_b\subset T$. For a cluster picture $\Sigma$, $\Lambda_{\Sigma}$ is, essentially, 
a permutation module on certain clusters (`even but not \"ubereven').

\bu
In all three settings, the metric determines a pairing on $\Lambda$. 
When all edge lengths of $G$ are integers, we call $G$ \emph{integral} and
this notion also transports to $T$ and $\Sigma$ as well. In this case,
the pairing is $\Z$-valued, and we can define the 
\emph{Tamagawa group} $\Phi=\Lambda^\vee/\Lambda$.

\bu
The group $\Aut G$ corresponds to
the group $\Aut T$ of automorphisms of $T$ together with a choice of sign for each yellow component, 
and to the group $\Aut\Sigma$ of permutations of clusters 
in $\Sigma$ with a choice of (compatible) signs for clusters of even size.

%

 
\begin{theorem}
The correspondence in Theorem \ref{icorr2} extends to the metric case. 
Suppose $G$, $T$ and $\Sigma$ correspond to one another. 
Choose a section $s: G/\langle \iota\rangle\to G$. Then there are 
canonical isomorphisms
\begin{equation*}
\label{}
  \Aut(\tilde G) \iso \Aut(\tilde T) \iso \Aut(\Sigma)
\end{equation*}
and canonical $\Aut(\cdot)$-equivariant isomorphisms
\begin{equation*}
  \Lambda_{\tilde G} \iso \Lambda_{\tilde T} \iso \Lambda_{\Sigma}
\end{equation*}
as lattices with a pairing. If $G$ is integral, they induce isomorphisms
\begin{equation*}
  \Phi(\tilde G) \iso \Phi(\tilde T) \iso \Phi(\Sigma).
\end{equation*}
\end{theorem}

The Tamagawa group $\Phi(\tilde G)$ is also equivariantly isomorphic to the graph-theoretic
Jacobian of $G$, see Proposition \ref{component group vs jacobian}.  
As one application, we get a simple description of the 2-torsion in
$\Phi(\tilde G)$.

\begin{corollary}
\label{itwotors}
Let $G$ be an integral hyperelliptic graph of genus $\ge 2$, 
$G^{\langle \iota\rangle}$ the set of fixed 
points of the involution $\iota$, 
and $\cW$ the set of those connected components of $G^{\langle \iota\rangle}$ 
that contain a point of integer distance from a vertex.
Then we have isomorphisms of $\textup{Aut}~G$-modules
$$
  \Phi(G)[2]\iso 
  \begin{cases}
  0~~&~~\text{if }\cW=\emptyset~~\text{and}~~\rk H_1(G,\Z)~~\text{even,}\\
  \F_2~~&~~\text{if }\cW=\emptyset~~\text{and}~~\rk H_1(G,\Z)~~\text{odd,} \\
  \ker(\F_2[\mathcal{W}]\stackrel{\textup{sum}}{\longrightarrow}\F_2)~~&~~\text{otherwise.}
  \end{cases}
$$
(Here `\textup{sum}' denotes the sum of the coefficients map.)
\end{corollary}

There is an algorithm due to Betts \cite{Betts} that computes, for
a BY tree $T$, the Tamagawa group $\Phi(T)$ and the group of 
invariants $\Phi(T)^{F=1}$ for $F\in\Aut T$. By the theorem above, it gives a
way to compute $\Phi(G)$ and $\Phi(\Sigma)$ for hyperelliptic graphs and 
cluster pictures as well.


%
%
%

\subsection{Hyperelliptic curves in odd residue characteristic} \label{hyp curves overview}

The present paper is purely combinatorial, and was motivated by its geometric 
counterpart~\cite{M2D2} that studies the arithmetic and geometry of hyperelliptic curves
over local fields. We briefly sketch the results for the interested reader.


Let $K$ be a local field of odd residue characteristic $p$, 
with valuation $\val_K$ on the separable closure $\Kbar$.
As before, let $C/K$ be a hyperelliptic curve
$$
  C: y^2 = f(x)
$$
of genus $\ge 2$.
Suppose it is semistable over some finite Galois extension $F/K$. 
The dual graph $G_C$ of the special fibre of the minimal regular model of $C/F$ is naturally a 
metric hyperelliptic graph, and it has a Galois action%
\footnote{This is slightly non-trivial, and relies on uniqueness of the minimal regular model of $C$ over $F$}
$$
  \GKK\to\Aut G_C.
$$
The cluster picture $\Sigma_C=(X,\Sigma)$ is given by the collections of roots that are cut out by $p$-adic discs, 
$$
  X      = \bigl\{\text{roots of $f$ in $\Kbar$}\bigr\}, \qquad 
  \Sigma = \bigl\{\s\!=\!X\cap D \bigm| D\!\subset\!\Kbar \text{ disc}, \s\ne\emptyset \bigr\}.
$$
It carries a metric, with the distance between clusters determined by the $p$-adic distances between the roots:
$$
  \delta(\s, \s') = 
 2\cdot\text{diameter}(\s\cup\s') - \text{diameter}(\s) - \text{diameter}(\s'),
$$
where for $X'\subseteq X$, diameter$(X')=\min_{r,r'\in X'}(\val_K(r-r'))$.
An automorphism $\sigma\in\GKK$ acts on $X$ and permutes the clusters.
Moreover, one can assign a natural sign $\epsilon_{\sigma}(\s)=\pm 1$ for every even cluster $\s$.
This gives a homomorphism 
$$
  \GKK\to\Aut \Sigma_C.
$$
Having produced both the hyperelliptic graph $G_C$ and the cluster picture $\Sigma_C$, we 
can now state

\begin{theorem}[\cite{M2D2}]
The core of the hyperelliptic graph associated to the cluster picture $\Sigma_C$  by Theorem \ref{icorr2} is 
isomorphic to the dual graph $G_C$. The isomorphism preserves the metric and the $\GKK$-action.
\end{theorem}

Thus, from a cluster picture of $C/K$, which is an elementary invariant constructed from the roots of $f(x)$,
we recover the semistable model of $C/F$ together with the Galois action. 
This allows us to determine some of the main arithmetic invariants of $C$,
such as 

\begin{itemize}
\item
Necessary and sufficient conditions for $C/K$ to be semistable;
\item
The Galois representation $\H(C/\Kbar,\Q_l)$;
\item 
The conductor of $C$;
\item
Equations for the minimal regular model of $C/F$;
\item
The Tamagawa number of the Jacobian of $C$ over $F$.\footnote{
Note also that Corollary \ref{itwotors} fully describes the 2-torsion of the 
Tamagawa group.}
\end{itemize}

\subsection{Classification of semistable types and naming convention}


Recall that one of the main motivations for studying hyperelliptic graphs, BY trees and cluster pictures was to produce a classification of semistable types of hyperelliptic curves. 
The results in this paper allow one to produce such a classification in any genus; the tables in \S\ref{s:tables} classify objects up to genus 3 and objects together with an automorphism in genus 1 and 2. Keeping track of an automorphism is important, as the action of Frobenius on the dual graph of the special fiber of the minimal regular model of hyperellptic curves is key to the study of their arithmetic.
One standard application of such classifications is that they enable one to use a systematic case by case analysis.
For example, \cite{CMThesis} employs the one in the present paper to prove a general result  on the parity of ranks of Jacobians of genus 2 curves. 

We end by proposing a naming convention for BY trees or, equivalently,  hyperelliptic graphs and cluster pictures. In the context of semistable curves of genus $1$ or $2$, our notation is compatible with that of Kodaira \cite{Kodaira} and 
Namikawa--Ueno \cite{NU}. As an illustration, the seven types in genus 2 shown previously get the names

\begin{center}
\begin{tabular}{c@{\qquad}c@{\qquad}c@{\qquad}c@{\qquad}c@{\qquad}c@{\qquad}c}
\TwBn
\end{tabular}
\end{center}
(with the same ordering). Moreover, we also allow an arbitrary automorphism to be encoded in the type. For example, for type $1_n$ above, there are two choices of automorphisms giving
$$
 \hetype{1_n^+}\qquad \hetype{1_n^-}.
$$
These correspond to a genus 2 curve whose reduction has one split or non-split node, 
respectively.

\medskip\noindent{\bf Related work.} 
Hyperelliptic graphs and their link to hyperelliptic curves are well-known 
\cite{BN,Caporaso,Chan,KY}. Our definition (\ref{defhypgraph}) 
is as in \cite{Caporaso,KY}; it is stronger than that of \cite{Chan} which does not 
require condition (3).\footnote{Condition (3) of \ref{defhypgraph} 
is needed for geometric reasons, for 
otherwise there are no hyperelliptic curves with such special fibres; 
see \cite{KY} Thm 1.2 and compare 36 types in genus 3 of \cite{Chan}
with 32 in Table \ref{tabg3}. From the point of view of our 1-1 correspondence, it is 
forced automatically by cluster pictures.} 
Cluster pictures (and, to some extent, BY trees) appear implicitly in \cite{Bosch} 
in setting of rigid geometry. As far as we are aware, neither 
the correspondence nor an analysis of automorphism groups (which is important when
the residue field is not algebraically closed) have been studied.

\subsection{Layout}
\S\ref{s:background} recalls terminology about graphs and properties of their homology, 
relative homology and the natural pairing coming from edge lengths.
\S\ref{s:objects} introduces hyperelliptic graphs, BY trees and cluster pictures.
\S\ref{s:121o} and \S\ref{s:121c} state and prove the one-to-one correspondence 
between hyperelliptic graphs, BY trees and cluster pictures.
\S\ref{s:homology} shows that it preserves the lattice $\Lambda$ as a module under 
automorphism groups.
\S\ref{s:components} discusses Tamagawa groups of hyperelliptic graphs.
The naming convention is addressed in \S \ref{s:naming}. The classification of semistable types is discussed in \S \ref{s:tables}. 

\begin{acknowledgements}
We would like to thank the EPSRC and the Royal Society for their support, and the Warwick Mathematics 
Institute where parts of this research were carried out.
\end{acknowledgements}

\comment

\bigskip
\bigskip

\subsection{Hyperelliptic curves in odd residue characteristic} 

Our motivation was to describe arithmetic invariants of hyperelliptic curves over local fields. 
As observed in \cite{M2D2}, cluster pictures appear naturally when trying to describe the special fibre. The equivalence in the present paper has been developed because many invariants of hyperelliptic curves are described using the dual graph of the special fibre. 

We now give a sketch of the kind of results one obtains. A precise application focused summary will be given in \cite{Hyxit}.

For the rest of the discussion in the subsection, $C/K$ will be a hyperelliptic curve over a local field of odd residue characteristic.

\subsubsection{Hyperelliptic graphs}
Hyperelliptic graphs are defined as purely combinatorial objects (see Definition \ref{defhypgraph}). It is known that the dual graph $G$ of the special fibre of a semistable hyperelliptic curve is always a hyperelliptic graph. Moreover, all hyperelliptic graphs arise this way. (See \cite{}, \cite{}).

\subsubsection{Cluster pictures}

Given 
$$
  C:y^2=cf(x) = c\prod_{r \in \cR}(x-r),
$$ 
the cluster picture $\Sigma$ is all the non-empty subsets of $\cR$ of the form $\cR \cap D$ for some disc $D$ in $\bar{K}$. To extract certain arithmetic invariants of $C$, we record the distance between clusters (given by $\delta(\s, \s') = v_K(r - r')$ for $r \in \s$ and $r' \in \s'$ with not both in $\s \cap \s'$), the Galois action on clusters (induced by that on $\cR$) and for each $\sigma \in \Gal(\bar{K}/K)$ and even cluster $\c$, a sign $\epsilon_{\sigma}(\c)$ (see \cite{M2D2}) that makes $\sigma$ an automorphism of the cluster picture (Definition \ref{autc}).

\subsubsection{Semistability criterion}
Using $\Sigma$ it is easy to decide whether $C/K$ is semistable via the knowledge of distances between all clusters, the action of inertia on $\Sigma$ and the valuation of the leading term $c$. (See \cite{M2D2}).
\subsubsection{Minimal regular model}
If $C$ is semistable over $K$ then the hyperelliptic graph associated to $\Sigma$ is indeed the dual graph of the special fibre $\bar{C}$ of the minimal regular model of $C$. Moreover the minimal regular model and its special fibre can be given explicitly.  
When $C/K$ is not semistable then one can recover the minimal regular model of $C$ over a suitable field extension $F/K$ where $C$ becomes semistable, together with the induced Galois action. (See \cite{M2D2}). 
\subsubsection{Homology and Galois representation}
Thanks to the correspondence given in Theorem 1.2, one can read off the homology of $G$ as a Galois module from $\Sigma$, e.g. the dimension of $\Lambda_{\Sigma} \simeq \Lambda_{G} = H_1(G,\Z)$ is essentially given by the number of even clusters in $\Sigma$. One can also describe the Galois representation $H^1_{et}(C/\bar{K},\Q_{\ell}$) and compute the conductor in terms of $\Sigma$ together with the automorphisms induced by the absolute Galois group. (See \cite{M2D2}, \cite{Weil2}).
\subsubsection{Tamagawa numbers}
For semistable curves $C$, Betts \cite{Betts} has developed an algorithm to find the Tamagawa group of the Jacobian of $C$ and a formula for the Tamagawa number in terms of the BY tree of $C$. Since the cluster picture $\Sigma$ is easy to compute and thanks to the explicit equivalence between BY trees and cluster pictures, this gives a practical way to compute Tamagawa numbers. (See \cite{Hyxit}).


%

\endcomment

\subsection{Notation}

Throughout the paper we use the following notation:\\[-6pt]

\begin{tabular}{llllll}
$G$      & hyperelliptic graph / open hyperelliptic graph (\S\ref{ssHG}) \cr
$T$      & BY tree / open BY tree (\S\ref{ssBY}) \cr
$\Sigma$ & cluster picture (\S\ref{ssCP}) \cr
$\s$     & cluster (element of $\Sigma$)\cr
$g$      & genus function on vertices of $G$ or $T$ / clusters in $\Sigma$, and the \cr
         & total genus of $G$, $T$ or $\Sigma$ 
           (Definitions \ref{hggenus}, \ref{bygenus}, \ref{genus of cluster})\cr
$\Lambda$ & homology lattice of $G$, $T$ or $\Sigma$ 
           (Definitions \ref{lambdaHG}, \ref{lambdaBY}, \ref{lambdaCP})\cr
$\Lambda^\vee$ & $=\Hom(\Lambda,\Z)$, the dual lattice\cr
$\delta$ & distance function on $G$, $T$ or $\Sigma$ 
           (Definitions \ref{defmetrichypgraph}, \ref{defmetricbytree}, \ref{de:distance})\cr
$\tilde G$, $\tilde T$ & core of an open hyperelliptic graph/BY tree 
           (Definitions \ref{coreHG}, \ref{coreBY})\cr
$T_b, T_y$ & blue/yellow part of $T$\cr
$\iota$      & hyperelliptic involution on a hyperelliptic graph $G$\cr
$s$      & continuous section $G/\langle \iota\rangle\to G$  to the quotient map (\S\ref{sections section})\cr
$\Z[X]$   & free abelian group on $X$
\end{tabular}


%

\def\GraphScale{0.4}
\def\SnakeWiggle{3pt}
\def\clustersep{1.3pt}
\rootsize{0.5}

\section{Background on metric graphs}\label{s:background}

\subsection{Graphs}
\label{ssgraphs}

In what follows, the word \emph{graph} refers to a topological space $G$ homeomorphic
to a finite (combinatorial) graph. It comes with a set $V(G)$ of \emph{vertices} 
(containing all points $x\in G$ of degree $\ne 2$) and \emph{edges} $E(G)$. 
\emph{Graph isomorphisms} are homotopy classes of homeomorphisms that 
preserve vertices and edges.
Loops and multiple edges are allowed, though note that in the topological setting the 
action of $\Aut G$ might permute multiple edges and reverse direction of loops.

By a \emph{metric graph} we mean a topological graph $G$ along with a function $l:E(G)\rightarrow \mathbb{R}_{>0}$ which assigns a length to each edge.
This may be extended into a metric on $G$. We write $\delta(v,v')$ for the shortest distance between $v,v'\in V(G)$. We require isomorphisms and automorphisms of metric graphs to preserve lengths. 


\subsection{Homology of graphs} \label{homology of graphs}

For a topological space $X$, $H_i(X)$ denotes the $i$-th singular homology group of the space $X$ with coefficients in $\mathbb{Z}$. For a subspace $A\subseteq X$, $H_i(X,A)$ denotes the $i$-th relative (singular) homology group, again with coefficients in $\mathbb{Z}$. Most calculations will be carried out via simplicial homology. This will give the same answer where used: see \cite[Theorem 2.27]{MR1867354}. We refer to Section 2 of op. cit. for more details and proofs of everything outlined below.

Let $G$ be a graph. We make $G$ into a $\Delta$-complex by taking the $0$-simplices (resp. $1$-simplices) to be the set of vertices (resp. edges, along with their endpoint(s)) of $G$ and write $C_0(G)$ (resp. $C_1(G)$) for the free $\mathbb{Z}$-module on the $0$-simplices (resp. $1$-simplices) of $G$. For each choice of orientation on the edges of $G$, which for a non-loop edge $e$ amounts to a choice of `nose' $e_+$ and `tail' $e_-$, we have an associated boundary map $d:C_1(G)\rightarrow C_0(G)$ sending a non-loop edge $e$ to $e_+-e_-$, and sending loops to $0$. We then have
\[H_1(G)=\text{Ker}(d)\subseteq C_1(G).\]
Given two choices of orientation, there is a canonical isomorphism between the associated homology groups and so $H_1(G)$ is independent of the choice of orientation. For the rest of this section, fix an orientation on $G$. 


\subsubsection{Action of automorphisms}

Let $\text{Aut }G$ be the group of automorphisms of $G$. Then $\text{Aut }G$ acts on $C_0(G)$ via its action on the set of vertices, and on $C_1(G)$ via its action on the set of edges save that now we add signs to this action to take account of the orientation. Explicitly, if $\sigma\in \text{Aut }G$ maps the (unsigned) non-loop edge $e$ to $e'$ then we define the action on $e\in C_1(G)$ as
\[\sigma(e)=\begin{cases} e'~~&~~\sigma(e_+)=e'_+\\ -e'~~&~~\sigma(e_+)=e'_-.\end{cases}\]
The action on loops is defined similarly, introducing a minus sign if $\sigma$ maps a loop with its positive orientation to a loop with its negative orientation. 
This action commutes with the boundary map and induces an action of $\text{Aut }G$ on $H_1(G)$. Whilst the action on $C_1(G)$ depends on the choice of orientation, the induced action on $H_1(G)$ does not (upon canonically identifying the homology groups arising from different choices of orientation). 

\subsubsection{Length pairing on homology of metric graphs}\label{sss:lengthpairing}

Suppose now that $G$ is a metric graph with associated length function $l$. 
Define the \emph{length pairing} on $C_1(G)$ by setting 
\[\left \langle e,e'\right \rangle = \begin{cases} l(e) ~~&~~ e=e,'\\0~~&~~e\ne e',\end{cases}\]
and extending bilinearly. This is independent of the orientation on the edges of $G$. The induced pairing on $H_1(G)$ is positive definite (since the pairing on $C_1(G)$ is) and invariant under the action of $\text{Aut }G$ (since the same is true of the pairing on $C_1(G)$, for any choice of orientation). 

In particular, $H_1(G)$ is a finitely generated, torsion free $\mathbb{Z}$-module which carries a canonical action of $\text{Aut }G$ and, in the metric case, a positive definite real valued, invariant pairing.

\subsubsection{Relative homology}\label{se:RH}

Let $H$ be a subgraph of $G$ (that is, a closed subspace of $G$ which is a union of edges and vertices of $G$). In the metric case, we put the induced metric on $H$ so that it is also a metric graph. Then $H$ has a natural structure of a $\Delta$-complex inherited from that on $G$. Having fixed an orientation on the edges of $G$, we have an induced orientation on the edges of $H$. The boundary operator $d:C_1(G)\rightarrow C_0(G)$ then maps $C_1(H)$ into $C_0(H)$ and we have
\[H_1(G,H)=\ker\left(\frac{C_1(G)}{C_1(H)}\longrightarrow \frac{C_0(G)}{C_0(H)}\right).\]

Writing $\text{Aut}_H~G$ for the subgroup of automorphisms of $G$ preserving $H$, the action of $\text{Aut }G$ on $C_1(G)$ defined above induces an action of $\text{Aut}_H~G$ on $H_1(G,H)$. Again, this is independent of the choice of orientation on $G$. 

In the metric case, define the $\emph{relative length pairing}$ on $C_1(G)$ by setting
\[\left \langle e,e'\right \rangle = \begin{cases} l(e) ~~&~~ e=e',e\notin H,\\0~~&~~\text{otherwise,}\end{cases}\]
and extending bilinearly. This induces a positive definite pairing on $H_1(G,H)$ which is invariant under the action of $\text{Aut}_H~G$. In particular, as with $H_1(G)$, $H_1(G,H)$ is a finitely generated, torsion free (as $C_1(H)$ is a direct summand of $C_1(G)$) $\mathbb{Z}$-module equipped with an action of $\text{Aut}_H~G$ and, in the metric case, a positive definite real valued, invariant pairing.

For an example illustrating all the definitions above, see Example \ref{Miller}.

\subsubsection{Subdivision of edges} \label{subdivision}

It will often be convenient to define a new $\Delta$-complex structure on $G$ by `subdividing' certain edges. That is, we take points $x_1,...,x_m$ lying on edges of $G$ and define a $\Delta$-complex structure on $G$ by taking the set of $0$-simplices to be the set $V(G)\cup \{x_1,...,x_m\}$, and redefining the set of $1$-simplices accordingly. Let us temporarily denote $G$, along with the new $\Delta$-complex structure as ${G^*}$, and fix an orientation on the $1$-simplices of ${G^*}$ (not necessarily induced from the orientation on the $1$-simplices of $G$). Then there is a natural map $C_1(G)\rightarrow C_1({G^*})$ sending an edge $e\in C_1(G)$ to the (signed) sum over its subdivisions. This map induces an isomorphism $H_1(G)\stackrel{\sim}{\rightarrow} H_1({G^*})$. If the set $\{x_1,...,x_m\}$ is preserved by an automorphism $\sigma$ of $G$, then $\sigma$ induces an action on $H_1({G^*})$ by the same formula as in the case of $G$ and the isomorphism $H_1(G)\cong H_1({G^*})$ defined above preserves this action. Moreover, in the metric case, each $1$-simplex of ${G^*}$ has an associated length and we thus obtain a length pairing on $H_1({G^*})$ by the same formula as for $G$. Since the length of an edge of $G$ is the sum of the lengths of its subdivisions, the isomorphism above identifies the pairings also. Consequently, we will frequently subdivide edges without further comment, with the caveat that when computing actions of automorphisms, we use subdivisions which are preserved by the automorphisms of interest. 

The above discussion applies equally well to the relative homology group of $G$ with respect to a subgraph $H$ (along with its automorphism action and, in the metric case, pairing), provided that when defining the simplicial complex structure on $G$ by subdividing edges, we give $H$ the simplicial complex structure inherited from the new one on $G$ as before.  

\def\GraphScale{0.4}
\def\SnakeWiggle{3pt}
\def\clustersep{1.3pt}
\rootsize{0.5}

\section{Hyperelliptic graphs, BY trees and cluster pictures}\label{s:objects}

\def\vz{\vspace{0pt}}
\def\clustersep{1.5pt}


\stepcounter{equation}
\begin{table}[t]
\begin{center}
\ \ \includegraphics{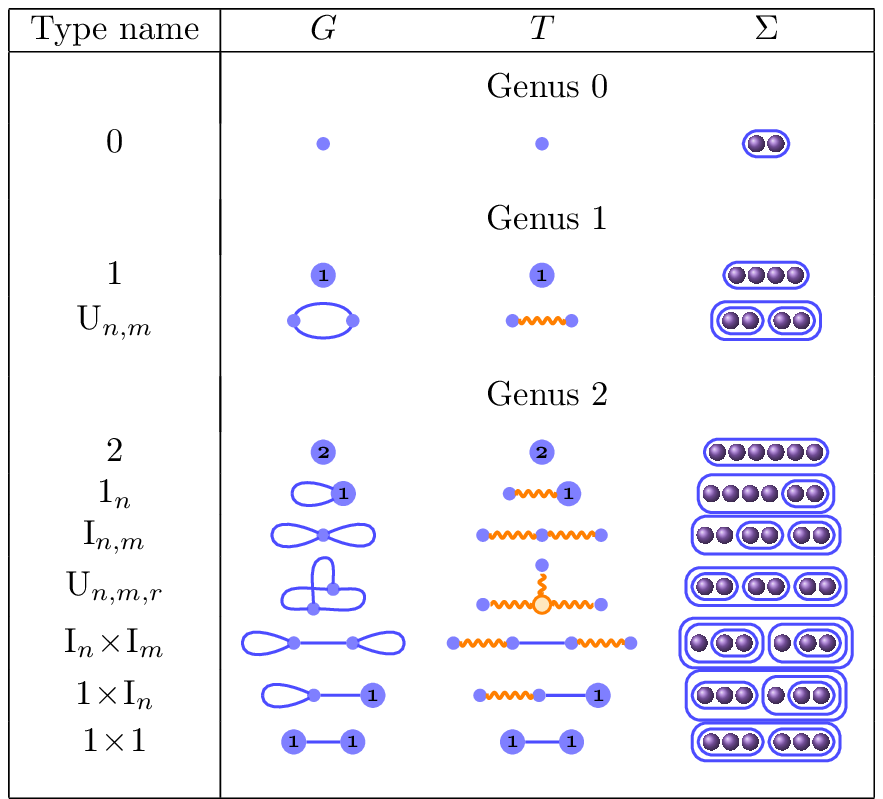}
\end{center}
\medskip
\caption{All type names, hyperelliptic graphs, BY trees and 
balanced cluster pictures in genus $0$, $1$ and $2$}
\label{g012table}
\end{table}

\subsection{Hyperelliptic graphs}
\label{ssHG}

\begin{definition}[Hyperelliptic graph]
\label{defhypgraph}
Let $G$ be a connected graph equipped with
\begin{itemize}
\item
$g: V(G)\to \Z_{\ge 0}$, a function that assigns a \emph{genus} to every vertex,
\item
$\iota: G\to G$ an involution (graph isomorphism of order $\leq$ 2).
\end{itemize}
We say $G$ (or, more precisely, $(G,g,\iota)$) is a \emph{hyperelliptic graph} if
\begin{enumerate}
\item 
vertices of non-zero genus are $\iota$-invariant;
\item
genus 0 vertices have degree $\ge 3$;
\item
$2g(v)+2\ge$ \# $\iota$-invariant edges at $v$, for every vertex $v$;
\item
the topological quotient $G/\langle \iota\rangle$ is contractible (that is, a tree).
\end{enumerate}
In addition, it is convenient to declare two exceptions
\begin{equation}\label{circlegraph}
  \pb{\excGzero} \qquad \qquad \pb{\excGone}
\end{equation}
(with $\iota$=reflection in the $x$-axis and vertices of genus 0) 
to be hyperelliptic graphs as well, although they violate (2).
\end{definition}

\begin{definition}[Open version]\label{opendefhypgraph}
An \emph{open} hyperelliptic graph $G$ is a connected graph $G$ with at least one 
open\footnote{that is, open on one end; such graphs are called `graphs with legs' in \cite{Caporaso}} edge, 
a genus function
$g:V(G)\rightarrow \mathbb{Z}_{\geq 0}$ and an involution $\iota$, satisfying conditions $(1)$-$(3)$ of Definition \ref{defhypgraph} for all vertices, and such that the quotient $G/\langle \iota\rangle$ is a tree with a unique open edge (in particular, $G$ has either a unique open edge, or a pair of open edges swapped by $\iota$).  
In addition, we declare the graph 
\begin{equation}\label{opencirclegraph}
  \excGzeroOpen
\end{equation}
(with vertex of genus 0 and $\iota$=reflection in the $x$-axis) 
to be an open hyperelliptic graph, though it violates (2).
\end{definition}

\begin{remark}
To stress the distinction from the open case, we will sometimes also refer 
to hyperelliptic graphs (as in Definition \ref{defhypgraph}) 
as \emph{closed hyperelliptic graphs}. 
\end{remark}

\begin{remark}
Any open hyperelliptic graph 
is homeomorphic to a connected component of $\bar{G} \setminus I$ for 
some hyperelliptic graph $\bar{G}=(\bar{G},g,\iota)$ for $I\subset V(\bar G)$ 
an $\iota$-orbit of vertices, together with the induced genus function 
and action of~$\iota$. In other words, $G$ is a hyperelliptic graph with 
an extra `open edge/pair of open edges'.
The missing vertex/pair of vertices are referred to as $\infty$ or $\infty^+, \infty^-$. 


We draw hyperelliptic graphs as follows, with numbers indicating the genus $g(v)$ when it is positive
(and omitted when $g(v)=0$).
$$
\begin{tabular}{c@{\qquad}c}
  \OnexInG & \OnexInOG \cr
  (closed) hyperelliptic graph & open hyperelliptic graph\cr
\end{tabular}
$$
\end{remark}

\begin{definition}[Genus]
\label{hggenus}
For both closed and open hyperelliptic graphs, the \emph{genus of $G$} is given  by
$$
  g(G) = \rk H_1(G) + \sum_{v\in V(G)} g(v).
$$
\end{definition}

\begin{definition}[Isomorphism] \label{iso of hyp graph}
Two hyperelliptic graphs (closed or open) are \emph{isomorphic} 
if there is a homeomorphism between them that preserves the defining data
(vertices, edges, genus markings, commutes with $\iota$).
We write $\Aut G$ for the group of \emph{automorphisms} of $G$.
(Recall that they are considered up to homotopy; see \S\ref{ssgraphs}.)
\end{definition}

\begin{remark} \label{hyp inv central}
In every hyperelliptic graph $G$ that is not the exceptional circle graph 
from \eqref{circlegraph},
$\iota$ is the unique involution in $\Aut G$ that fixes all the vertices
of positive genus, and 
such that the quotient $G/\langle \iota \rangle$ is a tree
(see e.g. \cite{Kausz} Prop 1.4). In particular, it is central, and 
graph isomorphisms commute with $\iota$ automatically.
\end{remark}

\begin{example}
\label{exhg012}
Table \ref{g012table} (2nd column) lists all
hyperelliptic graphs of genus 0,1 and 2, and 
Table \ref{tabg1a} open ones of genus 0 and 1, up to isomorphism.
\end{example}

\begin{definition}[Core]
\label{coreHG}
A \emph{(closed) hyperelliptic subgraph} $H$ of a (open or closed) 
hyperelliptic graph $G$ is a (closed) hyperelliptic graph $H$ such that
\begin{itemize}
\item
As a topological space, $H$ is a union of vertices and edges of $G$,
closed in $G$, and closed under $\iota$.
\item
The vertices of $H$ are exactly those vertices of $G$ that are in $H$,
except for those that have genus 0 and degree 2 in $H$. The latter
become points on the edges of $H$. (When $H$ is the second exceptional graph 
\eqref{circlegraph}, we declare its two $\iota$-invariant points to be genus 0 vertices.)

\item
The genus of a vertex of $H$ is the same as its genus in $G$.
\end{itemize}
%
The \emph{core} $\tilde G$ of an open hyperelliptic graph $G$ is its maximal 
closed hyperelliptic subgraph. By Proposition \ref{Tclosure} below, 
it is unique, has the same genus as $G$, and can be easily obtained from $G$
by removing a few vertices and edges near $\infty$. 
\end{definition}

\begin{definition}[Equivalence] \label{equiv def hyp}
We say that two open hyperelliptic graphs are \emph{equivalent} if they have isomorphic cores.
\end{definition}

\begin{example}
\label{ex1:0g}
Take a single vertex 
of genus 1 with a loop, 
and let $\iota$ reverse the direction of the loop:
$$
  \tilde G \quad = \quad  \OZGb 
$$
It is a (closed) hyperelliptic graph $\tilde G$ of genus 1+1=2. There are, up to isomorphism, 
7 open hyperelliptic graphs $G$ with $\tilde G$ as their core:

\begingroup
\def\GraphScale{0.35}
\begin{center}
\begin{tabular}{ccccccc}
\OZG\cr 
\end{tabular}
\end{center}
\endgroup

\noindent
Graphs \#1,\#2,\#5 have $\Aut G=\langle 1,\iota\rangle\iso C_2$, and the other four 
$\Aut G\iso C_2^2$.
\end{example}

\begin{definition}[Metric version]\label{defmetrichypgraph}
A \emph{metric hyperelliptic graph} $G$ (closed or open) is a hyperelliptic graph equipped with an $\iota$-invariant length function $\delta: E(G)\to\R_{>0}$ 
on the edges (excluding the open edge).
In this case, we write $\delta(v,v')$ for the shortest distance between $v,v'\in V(G)$.
We require isomorphisms and automorphisms of metric graphs to preserve $\delta$,
and for a hyperelliptic subgraph $H\subset G$ we require%
\footnote{When $H$ is the exceptional circle graph from \eqref{circlegraph}, the vertices
  might change, and so we require instead that the total length of the circle 
  is the same in $H$ and in $G$}
$\delta_H(v,v')=\delta_G(v,v')$ for the vertices of $H$. Similarly, we say that two open metric hyperelliptic graphs are equivalent if there is an isomorphism between their cores which preserves distance. 
\end{definition}

\begin{definition}[The homology lattice $\Lambda$]
\label{lambdaHG}
Let $G$ be a closed or open hyperelliptic graph. We set $\Lambda_G=H_1(G)$.  Recall from \S\ref{homology of graphs}
that a (closed) hyperelliptic graph comes with a natural action of $\Aut G$ and, if $G$ is a metric hyperelliptic graph, a non-degenerate, real valued, $\Aut G$-invariant pairing. The same is also true of the open case. Indeed, in Lemma \ref{le:H1iso} we will show that if $G$ is an open hyperelliptic graph with core $\tilde{G}$, then $H_1(G)$ is canonically isomorphic to $H_1(\tilde{G})$. Since automorphisms of $G$ induce automorphisms of the core via restriction, we may use the closed case to equip $\Lambda_G$ with an action of $\Aut G$ and, in the metric case, with a non-degenerate, real valued, $\Aut G$-invariant pairing.
\end{definition}

\begin{example}\label{Windmill}
Consider the following open hyperelliptic graph $G$:

\bigskip

\noindent
\begin{minipage}{0.37\textwidth}
\WINDMILL
\end{minipage}%
\begin{minipage}{0.62\textwidth}
\begin{tabular}{p{\textwidth}}
The graph has vertices $v_1, v_x, v_2, v_3^+, v_3^-$ of genera $2,0,0,0,0$;
 loops $\l_1, \l_2$ of lengths $5,5$; edges $\l_3, \l_4, e_1, e_2^+, e_2^-, e_3^+, e_3^-$ of lengths $6,6,1,1,1,1,1$ and two open edges $e_{\infty}^+, e_{\infty}^-$ going off to infinity from $v_x$. The involution $\iota$ fixes $v_1, v_2, v_x, e_1$, swaps $v_3^+$ with $ v_3^-$, $e_2^+$ with $e_2^-$, $e_3^+$ with $e_3^-$ and $e_\infty^+$ with $e_\infty^-$, and reverses the directions of $\l_1,\l_2,\l_3,\l_4$.
$G$ admits an automorphism $\sigma\in\Aut_G$ of order 4, that fixes all vertices except $v_3^+$ and $v_3^-$, fixes the edges $e_1, e_2^+, e_2^-$ pointwise, swaps $e_3^+$ and $e_3^-$, reverses the directions of $\l_3$ and $\l_4$, and has an order four action on the loops, sending $\l_1\to -\l_2 \to -\l_1 \to \l_2$ (where $-\l$ denotes $\l$ with the opposite orientation).
\end{tabular}
\end{minipage}%

The lattice $\Lambda_G$ is
$$
  \Lambda_G = H_1(G) = \langle \l_1, \l_2, \l_3+e_3^+ - e_3^-, \l_4 + e_3^+ - e_3^-, e_2^+ - e_2^-\rangle \simeq \Z^5,
$$
where we have picked an orientation for each edge and loop (edges $e_2^{\pm}, e_3^{\pm}$ going bottom-to-top, $\l_3, \l_4$ right-to-left, and the loops $\l_1, \l_2$ oriented clockwise) and where $+$ means concatenation and $-x$ is $x$ with the opposite orientation.
The loops in the above basis have lengths $5,5,8,8,2$, and have trivial intersections, except for the third and fourth basis elements whose intersection has length 2. Thus the length-pairing (see Section \ref{sss:lengthpairing}) and the action of $\sigma$ on $\Lambda_G$ are
$$
\langle\cdot,\cdot\rangle=
\begin{pmatrix}
5&0&0&0&0\\
0&5&0&0&0\\
0&0&8&2&0\\
0&0&2&8&0\\
0&0&0&0&2\\
\end{pmatrix},
\qquad\qquad\qquad
\sigma=
\begin{pmatrix}
0&-1&0&0&0\\
1&0&0&0&0\\
0&0&-1&0&0\\
0&0&0&-1&0\\
0&0&0&0&1\\
\end{pmatrix}.
$$
\end{example}

\subsection{BY trees}
\label{ssBY}

\begin{definition}[BY tree]
\label{defbytree}
A BY tree is a finite tree $T$ with a genus function $g: V(T)\to\Z_{\ge 0}$ on vertices 
and a 2-colouring blue/yellow on vertices and edges such that
\begin{enumerate}
\item 
yellow vertices have genus 0, degree $\ge 3$, and only yellow edges; 
\item
blue vertices of genus 0 have at least one yellow edge;
\item
at every vertex, $2g(v)+2\ge$ \# blue edges at $v$.
\end{enumerate}
Note that all leaves are blue. 

In diagrams, yellow edges are drawn squiggly and yellow vertices 
hollow for the benefit of viewing them in black and white.
\end{definition}

\begin{example}
\label{exby012}
Table \ref{g012table} (third column) lists all 
BY trees of genus 0,1 and 2 up to isomorphism. 
\end{example}

\begin{notation}
As a topological space, $T=T_b\coprod T_y$ with
$T_b$ the \emph{blue part}, and $T_y$ the \emph{yellow part}.
(Thus $T_b\subset T$ is a closed subset.)
\end{notation}

\begin{definition}[Open version]\label{opendefbytree}
An \emph{open BY tree} $T$ is a finite tree $T$ with a unique open edge, a genus function $g: V(T)\to\Z_{\ge 0}$ on vertices 
and a 2-colouring blue/yellow on vertices and edges, satisfying conditions (1), (2) and (3) of Definition \ref{defbytree}. 
\end{definition}


In other words, as for hyperelliptic graphs,  an open BY tree $T$ is a BY tree with one `missing' vertex, that we will refer to as
$\infty$. Again, we sometimes refer to BY trees of Definition \ref{defbytree} as \emph{closed}, to distinguish them from the open ones.

\begin{definition}[Isomorphic BY trees]
Two (closed or open) BY trees are \emph{isomorphic} 
if there is a homeomorphism between them that preserves the defining data
(vertices, edges, genus markings, colouring).
\end{definition}

\begin{definition}
\label{bygenus}
For a closed or open BY tree $T$, the \emph{genus} of $T$ is
$$
  g(T) = \rk H_1(T,T_b) + 
    \sum_{v\in V(T)} g(v),
$$
where the first term is the (singular) relative homology group (see \S \ref{se:RH}).
\end{definition}

\begin{remark} \label{reduced homology remark}
The relative  homology sequence
$$
  0=H_1(T) \lar H_1(T,T_b) \lar H_0(T_b) \lar H_0(T)=\Z
  $$
  gives an isomorphism $H_1(T,T_b)\iso \tilde{H}_0(T_b)$ where the latter group is the reduced homology of $T_b$ in degree zero. In particular,  $\rk H_1(T,T_b)$ is equal to one less than the number of connected components of $T_b$.
\end{remark}

\begin{definition}[Core]
\label{coreBY}
A \emph{(closed) BY subtree} $T'$ of a (closed or open) 
BY tree $T$ is a (closed) BY tree $T'$ such that
\begin{itemize}
\item
As a topological space, $T'$ is a union of vertices and edges of $T$,
and is closed in $T$.
\item
The vertices of $T'$ are exactly those vertices of $T$ that are in $T'$,
except for those of genus 0 that in $T'$ have degree 2 and incident edges of the
same colour as the vertex. These exceptional vertices
become points on the edges of $T'$.
\item
The genus of a vertex of $T'$ is the same as its genus on $T$.
\end{itemize}
%
The \emph{core} $\tilde T$ of an open BY tree $T$ is its maximal 
closed BY subtree. Again, we will see later 
(Proposition \ref{Tclosure}) that it is unique, has the same genus,
and is obtained from $T$ by removing a few vertices and edges near~$\infty$.
\end{definition}

\begin{definition}[Equivalence] \label{equiv def BY}
We say that two open BY trees  are \emph{equivalent} if they have isomorphic cores.
\end{definition}


\begin{definition}
\label{autb}
An \emph{isomorphism of (closed or open) BY trees} $T\to T'$ 
is a pair $(\alpha, \epsilon)$ where 
\begin{itemize}
\item 
$\alpha$ is a graph isomorphism $T\to T'$ 
(in the open case, $T\cup\{\infty\}\to T'\cup\{\infty\}$ with $\alpha(\infty)=\infty$)
that preserves the genera of the vertices and the colours, and
\item
$\epsilon(Z)=\pm 1$ is a collection of signs for every connected component $Z$ of 
the yellow part $T_y\subset T$. 
\end{itemize}
Equivalently, 
$\epsilon$ is a collection of signs $\epsilon(v)\in \{\pm 1\}$ and $\epsilon(e)\in \{\pm 1\}$ for every 
yellow vertex and yellow edge, such that $\epsilon(v)=\epsilon(e)$ whenever $e$ 
ends at $v$.
The isomorphisms are composed by a cocycle rule
$$
  (\alpha, \epsilon_\alpha)\circ (\beta, \epsilon_\beta) = 
    \bigl(\alpha\circ\beta, \bullet\mapsto\epsilon_\beta(\bullet)\epsilon_\alpha(\beta(\bullet))\bigr).
$$
An \emph{automorphism} of $T$ is an isomorphism from $T$ to itself. We write $\Aut T$
for the group of automorphisms.
(As all $\epsilon$ may be chosen to be $+1$, 
this extended notion of an isomorphism does not affect the earlier definition of
BY trees being isomorphic.)
\end{definition}

\begin{example}
\label{ex1:0t}
Take a tree $\tilde T$ on 2 blue vertices, one of genus 0, and one of genus 1, 
with one yellow edge between them:
$$
  \tilde T \quad =  \quad \OZTb
$$
It is a (closed) BY tree of genus 1+1=2. There are, up to isomorphism, 
7 open BY trees $T$ with $\tilde T$ as their core:

\begingroup
\def\GraphScale{0.35}
\begin{center}
\begin{tabular}{ccccccc}
\OZT\cr 
\end{tabular}
\end{center}
\endgroup

\noindent
Trees \#1,\#2,\#5 have $\Aut T=\Aut\tilde T\iso C_2$, and the other four 
$\Aut T\iso C_2^2$.
\end{example}

\begin{remark}
\label{autbr}
If $T_y^{(1)},...,T_y^{(r)}$ are the connected components of $T_y\subset T$, then, by definition,
$$
  \Aut T = \Aut_0(T) \ltimes (\Z/2\Z)^r,
$$
where $\Aut_0(T)$ consists of those elements for which $\epsilon(T_y^{(j)})=+1$ for all~$j$. 
Equivalently, $\Aut_0(T)$ is the group of (homotopy classes of) 
homeomorphisms $T\to T$ that preserve $V(T), E(T), T_b, T_y$ and $g$.
\end{remark}

\begin{definition}[Metric version]\label{defmetricbytree}
A \emph{metric (open or not) BY tree} is one with a length function
$\delta: E(T)\to\R_{>0}$ on the edges (excluding the open one).  
We denote by $\delta(v,v')$ the distance between $v,v'\in V(T)$, and 
we require isomorphisms/automor\-phisms of metric trees to preserve $\delta$. Similarly, we say that two open metric BY trees are equivalent if there is an isomorphism between their cores which preserves distance. 
\end{definition}

\begin{definition}[The lattice $\Lambda$]
\label{lambdaBY}
Let $T$ be a (open or not) BY tree. We set $\Lambda_T=H_1(T,T_b)$.  In the closed case, as detailed in Section \ref{homology of graphs}, it comes with a natural action of $\text{Aut}_0~ T$ and, if $T$ is a metric BY tree, a non-degenerate, real valued, $\text{Aut}_0 T$-invariant pairing. We extend the action of $\text{Aut}_0~ T$ to an action of the full automorphism group $\Aut T$ as follows. Let $\sigma=(\sigma_0,\epsilon_\sigma)$ be an element of $\Aut T$ and let $e$ be a yellow edge, viewed as an element of $C_1(T)$. Then we set $\sigma(e)=\epsilon(Z)\sigma_0(e)$ where $Z$ is the connected component of $T_y$ containing $e$ and the action of $\sigma_0$ on $C_1(T)$ is as in Section \ref{homology of graphs}. This induces the sought action of $\Aut T$ on $H_1(T,T_b)$. 

We will show in Lemma \ref{le:H1iso} that if $T$ is an open BY tree with core $\tilde{T}$, then $H_1(T,T_b)$ is canonically isomorphic to $H_1(\tilde{T},\tilde{T}_b)$. Since automorphisms of $T$ induce automorphisms of the core via restriction, we may use the closed case to equip $\Lambda_T$ with an action of $\Aut T$ and, in the metric version, a non-degenerate, real valued, $\Aut T$-invariant pairing also.
\end{definition}

\begin{example}\label{Miller}
Consider the following open BY tree $T$:

\bigskip

\noindent
\begin{minipage}{0.40\textwidth}
\MILLER
\end{minipage}%
\begin{minipage}{0.59\textwidth}
\begin{tabular}{p{\textwidth}}
The graph $T$ has vertices $u_1,\ldots u_4$, $w_1,w_2,w_3,x$ of genera $0,0,0,0,2,0,0,0$. There is an edge from $x$ going off to infinity. $T$ admits an automorphism $\sigma=(\alpha,\epsilon)\in\Aut_T$ of order 4, where $\alpha$ swaps $u_1$ and $u_2$, and fixes all the other vertices and the sign function $\epsilon$ is given by 
$$\epsilon(u_1 x)=1,\quad \epsilon(u_2 x)=-1,\quad \epsilon(w_2 x)=1,$$
$$\epsilon(w_3)=\epsilon(w_3 w_2)=\epsilon(u_3 w_3)=\epsilon(u_4 w_3)=-1,$$
where $yz$ denotes the edge between the vertices $y$ and $z$.
\end{tabular}
\end{minipage}%
\medskip
%
%

The lattice $\Lambda_T = H_1(T,T_b)$ is the relative homology of $T$ with respect to its blue part. In other words it consists of 1-chains in $T$ whose boundary is blue (see Section \ref{se:RH}).
Writing $[a,b]$ for the shortest path from $a$ to $b$, and 
$$c_1=[u_1,x], \quad c_2=[u_2,x],\quad  c_3=[u_3,w_2],\quad c_4=[u_4,w_2],\quad c_5=[w_2,x],$$
the lattice is given by
$$
  \Lambda_T = H_1(T,T_b) = \langle c_1, c_2, c_3, c_4, c_5 \rangle \simeq \Z^5.
$$
The cycles have lengths $5,5,8,8,2$, and have trivial intersections, except for the third and fourth basis elements whose intersection has length 2. Clearly $\alpha$ swaps $c_1$ and $c_2$ and fixes the other basis vectors, while the sign $\epsilon$ in $1$ on $c_1$ and $c_5$, and $-1$ on $c_2, c_3, c_4$. Thus the length-pairing (see Section \ref{sss:lengthpairing}) and the action of $\sigma$ on $\Lambda_T$ are
$$
\langle\cdot,\cdot\rangle=
\begin{pmatrix}
5&0&0&0&0\\
0&5&0&0&0\\
0&0&8&2&0\\
0&0&2&8&0\\
0&0&0&0&2\\
\end{pmatrix},
\qquad\qquad\qquad
\sigma=
\begin{pmatrix}
0&-1&0&0&0\\
1&0&0&0&0\\
0&0&-1&0&0\\
0&0&0&-1&0\\
0&0&0&0&1\\
\end{pmatrix}.
$$
\end{example}

\subsection{Cluster pictures}
\label{ssCP}

\begin{definition}[Cluster picture]
\label{defclpic}
Let $X$ be a finite set and $\Sigma\subset \cP(X)$ a collection of non-empty subsets of $X$;
elements of $\Sigma$ are called \emph{clusters}. 
Then $\Sigma$ (or $(X,\Sigma)$) is a \emph{cluster picture} if
\begin{enumerate}
\item
Every singleton (`root') is a cluster, and $X$ is a cluster. 
\item
Two clusters are either disjoint or contained in one another.
\end{enumerate}
We say that $\Sigma$ has \emph{genus} $g$ if $|X|\in\{2g+1,2g+2\}$.
\end{definition}

Two cluster pictures $(X,\Sigma_X)$ and $(Y,\Sigma_Y)$ are \emph{isomorphic}
if there is a bijection $X\to Y$ that takes $\Sigma_X$ to $\Sigma_Y$.
(So we may take $X=\{1,...,n\}$, 
and just consider its cluster pictures, up to $S_n$-permutations.)

\begin{remark} 
\label{clreal}
Let $X=\{r_1,...,r_n\}\subset K$ be a finite subset of a field with a (non-trivial) valuation.
Then the non-empty subsets of $X$ that are cut out by discs
in $\bar K$ form a cluster picture. Conversely, every cluster picture arises in this way: 
for any $K$ that has at least $n$ elements in its residue field one 
can find $X\subset K$ that realises it.
\end{remark}

\begin{example}
\label{ex1:0s1}
Let $X=\{1,2,3,4,5,6\}$ and let 
$$
  \Sigma =\Bigl\{ \{1\},...,\{6\},\{5,6\},X\Bigr\}.
$$
Thus, apart from the required singletons and $X$, there is one extra cluster $\s=\{5,6\}$. 
We draw $\Sigma$ with ovals around every $\s\in\Sigma$ with $|\s|>1$:
$$
  \Sigma \quad =  \quad \OZSb
$$
It is a cluster picture of genus 2.
In the language of the last remark, it is realised, for example, 
by $\{1,2,3,4,-p,p\}\subset\Q_p$ ($p\ge 5$). 
\end{example}

\begin{definition}[Children] \label{children}
If $\s'\subsetneq \s$ is a maximal subcluster, we write $\s'<\s$ and 
refer to $\s'$ as a \emph{child} of $\s$, and $\s$ as the \emph{parent} of $\s'$. 
\end{definition}

\begin{definition}[Types of clusters]
A cluster $\s$ is \emph{proper} if $|\s|>1$ or
$|\s|=|X|=1$, a \emph{twin} if $|\s|=2$,
and it is \emph{odd}/\emph{even} if its size is odd/even.
A proper cluster is \emph{\"ubereven} if it has no odd children.
\end{definition}

\begin{definition}[Genus] \label{genus of cluster}
A non-\"ubereven cluster has \emph{genus} $g=g(\s)$ if it has $2g+1$ or $2g+2$ 
odd children; \"ubereven clusters are declared to have genus~0.
\end{definition}

\begin{definition}[Balanced]
A cluster picture $\Sigma$ is \emph{balanced} if $|X|$ is even, 
$X$ is the only cluster of size $>\frac{|X|}2$, 
and there are either 0 or 2 clusters of size $\frac{|X|}2$.
\end{definition}

\begin{example}
\label{excp012}
Table \ref{g012table} (last column) lists all 
balanced cluster pictures of genus 0,1 and 2 up to isomorphism.
\end{example}

\begin{definition}
\label{autc}
An \emph{isomorphism} $\Sigma\to\Sigma'$ is an equivalence class of  
pairs $(\alpha, \epsilon)$ with
$\alpha: \Sigma\to\Sigma'$ a bijection that preserves cluster sizes and inclusions, 
and $\epsilon(\s)=\pm 1$ a collection of signs for even clusters $\s\in\Sigma$, 
such that
$$
  \epsilon(\s')=\epsilon(\s) \qquad \text{for $\s$ \"ubereven}, \s'<\s.
$$
Here we say pairs $(\alpha,\epsilon)$ and $(\alpha',\epsilon')$ are equivalent if $\alpha$ and $\alpha'$ induce the same map between the sets of proper clusters of $\Sigma$ and $\Sigma'$, and $\epsilon=\epsilon'$. 
We compose isomorphisms by a cocycle rule
$$
  (\alpha, \epsilon_\alpha)\circ (\beta, \epsilon_\beta) = 
    \bigl(\alpha\circ\beta, \s\mapsto\epsilon_\beta(\s)\epsilon_\alpha(\beta(\s))\bigr).
$$
An \emph{automorphism} of $\Sigma$ is an isomorphism from $\Sigma$ to itself, and we write $\Aut \Sigma$
for the group of automorphisms. (As all $\epsilon$ may be chosen to be $+1$, 
this extended notion of an isomorphism does not affect the definition of
being isomorphic.) Equivalently, one may think of an element of $\Aut\Sigma$ as a pair $(\sigma,\epsilon_\sigma)$ where $\sigma$ is a permutation of the proper clusters of $\Sigma$ and $\epsilon_\sigma$ is a collection of signs as above, composition of two such being given by the same formula. We will frequently take this viewpoint without further comment in what follows. 
\end{definition}

\begin{remark}
\label{autsr}
As for BY trees (Remark \ref{autbr}), we have
$$
  \Aut \Sigma = \Aut_0(\Sigma) \ltimes (\Z/2\Z)^r,
$$
where $\Aut_0(\Sigma)$ consists of elements for which 
$\epsilon(\s)=+1$ for all even clusters $\s\in\Sigma$.
The number $r$ is the number of equivalence classes of even clusters for the equivalence 
relation generated by `$\s'\sim\s$ if $\s$ is \"ubereven and~$\s'<\s$'.
\end{remark}


\begin{definition}[Equivalence]
\label{Sequivdef}
We say that cluster pictures $(X_1,\Sigma_1)$ and $(X_2,\Sigma_2)$ are \emph{equivalent} if $(X_2,\Sigma_2)$ is isomorphic to a cluster picture obtained from $(X_1,\Sigma_1)$  in a finite number
of the following steps $(X,\Sigma)\to (X',\Sigma')$. 
\begin{itemize}
\item[(i)]
(`add cocluster')
$X'\!=\!X$, $\Sigma'\!=\!\Sigma\cup\{X\setminus\s\}$ for some $\s<X$, $|X|\!=\!2g+2$.
\item[(ii)]
(`remove cocluster') 
$X'=X$, $\Sigma'=\Sigma\setminus\{\s\}$ for some $\s<X$ with $|\s|\ge 2$,
when $X$ has exactly two children, $|X|=2g+2$.
\item[(iii)]
(`$2g\!+\!1\to 2g\!+\!2$')
$X'=X\coprod\{r\}$, $\Sigma'=\Sigma\cup\{X'\}$, when $|X|=2g+1$.
\item[(iv)]
(`$2g\!+\!2\to 2g\!+\!1$')
$X=X'\coprod\{r\}$, $\Sigma=\Sigma'\cup\{X\}$, when $|X|=2g+2$.
\end{itemize}
\end{definition}

In general, note that 
\begin{itemize}
\item 
the genus of a cluster picture is preserved under equivalence;
\item
(ii)=(i)$^{-1}$ and (iv)=(iii)$^{-1}$, hence this is an equivalence relation;
\item
(i) does nothing when $X$ has only 2 children;
\end{itemize}

\begin{example}
\label{ex1:0s2}
There are, up to isomorphism, 
7 cluster pictures $\tilde\Sigma$  in the equivalence class of 
$\Sigma$ from Example \ref{ex1:0s1}:

\begingroup
\def\clpicscale{0.8}
\rootsize{0.4}
\begin{center}
\begin{tabular}{ccccccc}
\OZS\cr 
\end{tabular}
\end{center}
\endgroup

\noindent
Cluster pictures \#1,\#2,\#5 have $\Aut\Sigma \iso C_2$, 
and the other four $\Aut \Sigma\iso C_2^2$. All 7 have trivial $\Aut_0(\Sigma)$,
so the automorphisms come from choices of signs on even clusters.
\end{example}


%

As for hyperelliptic graphs and BY trees, we have a metric version:

\begin{definition}[Metric version]\label{de:distance}
A cluster picture $(X,\Sigma)$ is \emph{metric} when 
every pair of proper clusters $\s<\fr$ is assigned a 
\emph{distance} $\delta(\s,\fr)=\delta(\fr,\s)\in\R_{>0}$.
%
%
The distance function clearly extends to every pair of proper clusters: 
if $\s$ and $\fr$ are distinct proper clusters with least common ancestor $\u$ 
(no child of $\u$ contains 
both $\s$ and $\fr$, but $\u$ does), so that 
$$
  \s<\s_1<...<\s_{k-1}<\u>\fr_{m-1}>...>\fr_1>\fr,
$$
we let $\delta(\s,\fr)$ be the sum of the $k+m$ distances between adjacent 
clusters in the chain; we let $\delta(\s,\s)$=0.

An \emph{isomorphism as metric cluster pictures} is one 
that preserves $\delta$.


We say that two metric cluster pictures are \emph{equivalent} if one is (up to isomorphism) obtained from the other by a finite number of `metric versions' of the moves (i)-(iv) of Definition \ref{Sequivdef}, in which we allow any metric $\delta'$ on $(X',\Sigma')$ such that
\begin{itemize}
\item $\delta'(\fr,\fr')=\delta(\fr,\fr')$ for every $\fr,\fr'\in\Sigma\cap\Sigma'$ with (for moves (i)-(ii))~$\fr,\fr'\!\neq\!~\!X$,
\item $\delta'(\fr,X\setminus \s)=\delta(\fr,X)$ for all $\fr\in\Sigma'\setminus\{X,X\setminus\s\}$ in move (i),
\item $\delta'(\fr,X)=\delta(\fr,\s)$ for all $\fr\in\Sigma'\setminus\{X\}$ in move (ii).
\end{itemize}
\end{definition}

In diagrams, distances are shown using subscripts on clusters: a cluster gets a subscript indicating the distance to its parent (cf. Example \ref{Grain}).

\begin{remark}
The clusters arising in Remark \ref{clreal} are naturally metric cluster pictures.
In this setting, for a cluster $\s$ we define its `depth' as $\delta(\s) = \min_{r,r'\in\s}(\val_K(r-r'))$, and for a pair $\s<\s'$ set the distance to be given by the `relative depth', $\delta(\s,\s')=\delta(\s)-\delta(\s')$. This is the same as the formula given in \S\ref{hyp curves overview} in terms of the diameter.
\end{remark}



We now define the lattice $\Lambda_{\Sigma}$ for a cluster 
picture $(X,\Sigma)$. We need one preliminary piece of notation.

\begin{notation} \label{the set cal E}
Let $(X,\Sigma)$ be a cluster picture. Let
$$
  \cE_\Sigma=\{\textup{even, non-\"ubereven clusters }\s\neq X\}.
$$
Further, for $\s\in \cE_\Sigma$, 
write $\hat\s$ for the smallest non-\"ubereven cluster strictly containing $\s$. If no such cluster exists, we set $\hat\s=X$.
\end{notation}

\begin{definition}[The lattice $\Lambda$]
\label{lambdaCP}
Let $(X,\Sigma)$ be a cluster picture. If $X$ is not an \"ubereven cluster,
let
$$ 
  \Lambda_\Sigma=\Z[\cE_\Sigma]  =
    \bigl\{\sum_{\s\in\cE_\Sigma} \lambda_\s \s \bigm|\lambda_\s\in \Z\bigr\}.
$$ 
If $X$ is \"ubereven, let 
$$
  \Lambda_\Sigma=\Bigl\{\sum_{\s\in\cE_\Sigma} \lambda_\s \s\in \Z[\cE_\Sigma]~\bigm|~\sum_{\hat\s=X} \lambda_{\s}=0\Bigr\}.
$$
Further, define an action of $\sigma=(\sigma_0,\epsilon_\sigma)\in \Aut\Sigma$ on $\Lambda_\Sigma$ by 
\[\sigma\cdot \s=\epsilon_\sigma(\s)\sigma_0(\s).\]
In the metric case, define a pairing on $\Lambda_\Sigma$ by setting, for $\s_1,\s_2\in \cE_\Sigma$,
$$
\left \langle \s_1,\s_2\right \rangle=
\begin{cases} 
2\delta(\s_1\wedge \s_2,\hat\s_1)~~&~~\hat\s_1=\hat\s_2,\\0~~&~~\hat\s_1\neq\hat\s_2,
\end{cases}
$$
where $\s_1 \wedge \s_2$ denotes the least common ancestor of $\s_1$ and $\s_2$.
\end{definition}

\begin{example}\label{Grain}
Consider the following cluster picture $\Sigma$:
\rootsize{0.5}
\def\vz{\vspace{1pt}}
\def\clustersep{5pt}
$$
\GRAINthree
$$
\def\clustersep{1.3pt}
\rootsize{0.5}
\def\vz{\vspace{0pt}}
\def\clustersep{1.5pt}

\noindent The cluster picture $\Sigma$ has clusters $\t_1,\ldots, \t_4$ (twins), $\s_1, \s_2, \s_3$, $X$ of genera $0,0,0,0,2,0,0,0$ with distances between children and parents $\delta(\t_1,X)=\delta(\t_2,X)=5/2$, $\delta(\t_3,\s_3)=\delta(\t_4,\s_3)=3$, $\delta(\s_3,\s_2)=1, \delta(\s_2,X)=1$, $\delta(\s_1,X)=2$.

This cluster picture admits an automorphism $\sigma=(\alpha,\epsilon)\in\Aut_\Sigma$ of order 4,
where $\sigma$ swaps $\t_1$ and $\t_2$ (which is indicated by the black line between the two leftmost twins) and fixes all other clusters. The sign function $\epsilon$ is 
$$
\epsilon(\t_1)=1,\quad
\epsilon(\t_2)=-1,\quad
\epsilon(\s_2)=1, \quad
\epsilon(X) = 1,
$$
$$
\epsilon(\t_3)=
\epsilon(\t_4)=
\epsilon(\s_3)=-1,
$$
indicated by the $+$ and $-$ on top of the respective even clusters.

Here the set of even non-\ub \ clusters (that are not $X$) is $\cE_\Sigma=\{\t_1,\t_2,\t_3,\t_4,\s_2\}$, and since $X$ is not \ub,
$$
  \Lambda_\Sigma = \langle \t_1, \t_2, \t_3, \t_4, \s_2 \rangle \simeq \Z^5.
$$
By definition
$$
 \hat \t_1 = \hat\t_2 = \hat\s_2 = X \qquad \text{and} \qquad \hat\t_3=\hat\t_4=\s_2,
$$
and 
$$
\t_1\wedge\t_2= \t_1\wedge \s_2 = \t_2\wedge \s_2= X \qquad \text{and}\qquad \t_3\wedge\t_4=\s_3.
$$
Thus by definition of the pairing on $\Lambda_\Sigma$ and since the action of $\sigma$ on $\Lambda_\Sigma$ is a signed permutation given by $\alpha$ and $\epsilon$, 
$$
\langle\cdot,\cdot\rangle=
\begin{pmatrix}
5&0&0&0&0\\
0&5&0&0&0\\
0&0&8&2&0\\
0&0&2&8&0\\
0&0&0&0&2\\
\end{pmatrix},
\qquad\qquad\qquad
\sigma=
\begin{pmatrix}
0&-1&0&0&0\\
1&0&0&0&0\\
0&0&-1&0&0\\
0&0&0&-1&0\\
0&0&0&0&1\\
\end{pmatrix}.
$$
\end{example}

\def\GraphScale{0.4}
\def\SnakeWiggle{3pt}
\def\clustersep{1.3pt}
\rootsize{0.5}

\section{One-to-one correspondence (open case)}\label{s:121o}

\stepcounter{equation}
\begin{table}[t]
\begin{center}
\ZeroOneTable
\end{center}
\medskip
\caption{Open hyperelliptic graphs,
open BY trees and cluster pictures in genus 0 and 1 up to isomorphism}
\label{tabg1a}
\end{table}

In this section we explain how to pass between cluster pictures, hyperelliptic graphs and BY trees.
We construct maps $\GG$, $\TT$ and $\SS$ (which become inverse to each other upon passing to isomorphism classes)
between objects according to the following diagram:

\begin{center}
\scalebox{0.85}{$
\begin{gathered}
\xymatrix{\left\{ \text{open hyperelliptic graphs}\right\} \ar@/^{1.5pc}/[rr]^{\TT} &  & \left\{ \text{open BY trees}\right\} \ar@/_{-1.5pc}/[ll]^{\GG}\ar@/^{1.5pc}/[rr]^{\SS} &  & \left\{ \text{cluster pictures}\right\}. \ar@/_{-1.5pc}/[ll]^{\TT}}
\end{gathered}
$}
\end{center}
In addition, for $\underline{F}$ one of $\GG$, $\TT$ or $\SS$, and $X$ an object on which it is defined, we construct a map $\underline{F}:\Aut X\rightarrow \Aut \underline{F}(X)$, which turns out to be an isomorphism in each case. Here there is a subtlety: for an open hyperelliptic graph $G$, the map $\Aut G\rightarrow \Aut \TT(G)$ depends on 
a choice of section $G/\langle\iota\rangle\to G$. 
The exact dependence of the map on the choice is examined in Proposition \ref{dependence of s}. By contrast, all other maps on automorphism groups are canonical. 
The constructions are summarised in Table \ref{tabdicto} and illustrated in Examples \ref{WM} and \ref{GM}.

The main result is the following. 

\begin{theorem}
\label{combmain1}
The maps $\GG$, $\TT$ and $\SS$ defined in Constructions \ref{GtoT}, \ref{TtoG}, \ref{StoT} and \ref{TtoS} give a genus preserving one-to-one correspondence between isomorphism classes of 
(either metric or not)
\begin{itemize} 
\item[(i)]
Cluster pictures,
\item[(ii)]
Open hyperelliptic graphs,
\item[(iii)]
Open BY trees.
\end{itemize}
Moreover, the associated maps on automorphism groups (see Constructions \ref{GtoT}, \ref{TtoG}, \ref{StoT} and \ref{TtoS} and the preceding paragraph) are isomorphisms.


\end{theorem}

\begin{proof}
Combine Proposition \ref{by tree hyp graph proposition} (`hyperelliptic graphs $\leftrightarrow$ BY trees') with Proposition \ref{cluster picture by tree proposition} (`BY trees $\leftrightarrow$ cluster pictures').
\end{proof}


%

\subsection{Hyperelliptic graphs $\leftrightarrow$ BY trees}
\label{sshto}


We begin with the maps between open hyperelliptic graphs and open BY trees, as well as the associated maps on automorphism groups. In fact, the constructions apply equally well in the closed version. Since both versions will be relevant later, we cover both here. Before detailing the constructions, we briefly discuss the notion of a `section to the quotient map' for a hyperelliptic graph. 

\begin{remark}[Sections to the quotient map] \label{sections section}
Let $G$ be a (open or not) hyperelliptic graph and $\pi:G\rightarrow G/\left\langle \iota \right\rangle$ be the quotient map. In order to construct the map between automorphisms of $G$ and automorphisms of the associated BY tree $T=G/\left\langle \iota \right\rangle$, it will be necessary to choose a continuous map $s: T \rightarrow G$ such that $\pi\circ s=\id$. That is, a continuous section to $\pi$ (we henceforth refer to $s$ as a \emph{section}, the continuity being understood). More concretely, such a choice amounts to the following. Write $G_y=\pi^{-1}(T_y)$ and choose a decomposition  $G_y=G_y^+\coprod G_y^-$,
such that $\pi: G_y^\pm\to T_y$ is a homeomorphism (that is, a choice of `top' and `bottom' 
above every connected component of $T_y$). 
Such a choice determines a section $s$ by sending $x\in T_b$ to its unique preimage in $G$, and $x\in T_y$ to its unique preimage in $G_y^+$. Conversely, a section $s$ determines such a decomposition by taking $G_y^+$ to be $s(T_y)$.
\end{remark}
\smallskip 
\smallskip
\begin{construction}[$\TT(G)$] ~
\label{GtoT}

\noindent {\em{Objects}}:
Let $(G, g, \iota)$ be a (closed) hyperelliptic graph, without loss of generality considered to be metric.
Define 
$$
  T=\TT(G)=G/\langle \iota\rangle,
$$ 
the topological quotient. It is a tree by Definition \ref{defhypgraph} (4),
and we colour the branch locus $T_b$ of the quotient map $\pi: G\to T$ blue and $T_y=T\setminus T_b$ yellow.
In other words, under $\pi$, blue points have one preimage and yellow points have
two preimages.
We make $T$ into a graph as follows:


Write $V(G)=\{v_1,...,v_k,v_{k+1},\iota(v_{k+1}),\ldots,v_n,\iota(v_n)\}$,
with the first $k$ vertices $\iota$-invariant, followed by the pairs swapped by $\iota$. 
Then $v_1,...,v_k$ give vertices $\bar v_1,...,\bar v_k\in V(T)$, each $\bar{v}_j$ declared blue of genus $g(v_j)$, 
and each pair $\{v_j,\iota(v_j)\}$ for $j>k$ gives one vertex $\bar v_j\in V(T)$, 
declared yellow of genus~0.

Next, take an edge of $G$, say of length $d$.
It is either mapped to another edge by $\iota$, is $\iota$-invariant, or is $\iota$-anti-invariant. Then
\begin{itemize}

\item
each $\iota$-invariant edge $vw$ gives a blue edge $\bar v\bar w$
of length $2d$;
\item
each $\iota$-anti-invariant edge $v \iota(v)$ (allowing for $\iota(v)=v$ in the case of loops) gives a yellow edge of length $d$
from $\bar v$ to an extra  blue genus 0 leaf of $T$; 
\item 
each swapped pair of edges
$v w$ and $\iota(v) \iota(w)$ gives a yellow edge $\bar v \bar w$ of length $2d$.
\end{itemize}
Finally, if $G$ 
is an open (possibly metric) hyperelliptic graph, 
we define $\TT(G)$ in the same way (treating $w$ as $\infty$ for the open edge(s)).

\smallskip
\noindent {\em{Automorphisms}}:
Let $G$, $T$ and $\pi:G\rightarrow T$ be as above. To define the map $\Aut G \rightarrow \Aut T$ we begin by choosing a  section $s:T\rightarrow G$ to $\pi$. 
Now given $\sigma \in \Aut G$ we construct an automorphism  $\TT(\sigma)=(\TT(\sigma)_0,\epsilon_{\TT(\sigma)})$ of $T$ as follows (if we wish to record the choice of section, we write $\TT(\sigma;s)$).  Since $\sigma$ necessarily commutes with the hyperelliptic involution $\iota$ (see Remark \ref{hyp inv central}), it induces a graph-theoretic automorphism $\bar{\sigma}$ of $T$ which preserves genus and colour by construction. We set $\TT(\sigma)_0=\bar{\sigma}$. Next, let $Z$ be a connected component of $T_y$. Then we define \[\epsilon_{T(\sigma;s)}(Z)=\begin{cases}1~~&~~s(\bar{\sigma}(Z))=\sigma(s(Z)),\\-1~~&~~\text{else}.\end{cases}\] 
\end{construction}

\begin{proposition} \label{is a BY tree}
If $G$ is a hyperelliptic graph, then $\TT(G)$ is a BY tree. The same 
holds in the open case.
\end{proposition}

\begin{proof}
It is easy to check the claim in genus 0 and 1
(see Tables \ref{g012table}, \ref{tabg1a}),
so assume $g>1$ from now on.
We check the three conditions of Definition \ref{defbytree}:

(1) \emph{Yellow vertices have genus 0, degree $\ge 3$, and only yellow edges}:
 all yellow vertices come from pairs $v_j, \iota(v_j)$ consisting of two distinct vertices of $G$ swapped by $\iota$, and are declared to have genus $0$ in the construction.  Moreover, yellow vertices have only yellow outgoing edges, because blue edges only come from edges between 
$\iota$-invariant vertices. Finally, $G\to T$ is a two-to-one topological cover at $\bar v_j$, and 
so $\deg\bar v_j=\deg v_j \ge 3$ by Definition \ref{defhypgraph} (2).

(2) \emph{Blue vertices of genus 0 have at least one yellow edge}: 
blue vertices either come from loops, in which case they have a yellow edge by construction, or 
from $\iota$-invariant vertices $v_j$. Suppose $g(v_j)=0$ but $\bar v_j$ has only blue 
outgoing edges. Then $\deg v_j\ge 3$ by \ref{defhypgraph} (2), so $v_j$ 
has at least three $\iota$-invariant outgoing edges, contradicting Definition \ref{defhypgraph} (3).

(3) \emph{At every vertex, $2g(v)+2\ge$ \# blue edges at $v$}: if $v$ is yellow, or a blue leaf 
coming from a loop, then it has only yellow edges and there is nothing to prove. 
If $v$ is blue, the blue edges from $\bar v$ are in one-to-one correspondence 
with the $\iota$-invariant edges from $v$, and the claim follows from Definition \ref{defhypgraph} (3).
\end{proof}

We now examine the dependence of the map $T(-,s):\Aut G\rightarrow \Aut T$ constructed above on the choice of section $s$. For this it will be useful to have the following definition. 

\begin{definition} \label{defi of psi}
Let $G$ be a hyperelliptic graph (open or not), let $T=\TT(G)$ be the associated BY tree, and let $\pi:G\rightarrow T$ be the quotient map. Further, let $s$ and $s'$ be two sections to $\pi$. We define the automorphism $\psi_{s,s'}=(\id,\epsilon_{s,s'})$ of $T$ by setting, for a component $Z$ of $T_y$,  
\[\epsilon_{s,s'}(Z)=\begin{cases} 1~~&~~s(Z)=s'(Z)\\-1~~&~~s(Z)\neq s'(Z).\end{cases}\]
\end{definition}

\begin{proposition} \label{dependence of s}
Let $G$ be a hyperelliptic graph (possibly open), let $T=\TT(G)$ be the associated BY tree, and let $\pi:G\rightarrow T$ be the quotient map. Then for each section $s:T\rightarrow G$ to $\pi$, the map $\TT(-;s)$ defines a homomorphism $\Aut G\rightarrow \Aut T$. Moreover, if $s$ and $s'$ are two choices of section then for all $\sigma \in \Aut G$ we have
$$
  \TT(\sigma;s')=\psi_{s,s'}\circ\TT(\sigma;s)\circ\psi_{s,s'}^{-1}.
$$
 In particular, the map $\TT(-;s)$ depends on $s$ only up to conjugation.
\end{proposition}

\begin{proof}
It is immediate from the construction that for any automorphism $\sigma$ of $G$, 
$\TT(\sigma;s)$ is indeed an automorphism of $T$ and it is clear 
that $\TT(-;s)$ is a homomorphism. A straightforward calculation verifies the statement concerning the dependence on the choice of section.
\end{proof}

\begin{construction}[$\GG(T)$]
\label{TtoG}~

\noindent {\em{Objects}}:
Let $T$ be a BY tree, viewed as a topological space. Let $G$ be the topological space given by glueing two disjoint copies $T^+$ and $T^-$ of $T$ along their common blue parts. Then $G$ comes with a natural map $\pi:G\rightarrow T$ making it into a two-to-one cover of $T$ ramified along $T_b$, as well as an involution $\iota$ (swapping the elements of the fibres of $\pi$ over points in $T_y$) such that $G/\langle \iota\rangle=T$. 
When $T$ has genus $\ge 2$, we make $G$ into a graph as follows:
\begin{itemize}
\item 
A blue vertex $\bar v\in V(T)$ which is not a genus 0 leaf gives an $\iota$-invariant vertex $v\in V(G)$
of genus $g(\bar v)$.
\item 
a yellow vertex $\bar v\in V(T)$ gives vertices $v^+$ and $v^-$ in $V(G)$ (where $v^+\in T^+$ and $v^-\in T^-$) swapped by~$\iota$,
of genus 0,
\item
a (necessarily blue) genus 0 leaf $\bar v$ of $T$ with an edge from $\bar{v}$ to a blue
vertex $\bar{w}$, of length $d$, gives an $\iota$-anti-invariant loop on $w$ of length~$d$,
\item
a (blue) genus 0 leaf $\bar v$ of $T$ joined by an edge of length $d$ to a yellow
vertex $\bar{w}$ gives an $\iota$-anti-invariant edge between $w^+$ and $w^-$ of length~$d$,
\item
a blue edge between (necessarily blue vertices) $\bar v$ and $\bar w$ of length $d$ gives an $\iota$-invariant edge between $v$ and $w$ 
 of length $d/2$,
\item
a yellow edge between $\bar{v}$ and $\bar{w}$ of length $d$ gives a pair of edges, one between $v^+$ and $w^+$ and one between $v^-$ and $w^-$, swapped by $\iota$, each of length $d/2$ (here, if $\bar{v}$ (resp. $\bar{w}$) is blue we set $v^+=v^-=v$ (resp. $w^+=w^-=w$)).
\end{itemize}
When $T$ is an open BY tree, we use exactly the same construction 
(say, by adding a vertex at $\infty$, following the steps above, and removing 
the vertices above $\infty$).

Finally, when $T$ has genus 0 or 1, we have to declare vertices of $G$ slightly 
differently, and we refer to Tables \ref{g012table}, \ref{tabg1a} for the 
correspondence.

\smallskip

\noindent {\em{Automorphisms}}: 
Write $G_b$ for  those points in $G$ fixed by $\iota$ and $G_y$ for $G\setminus G_b$. Further, write $G_y^+$ for the points in $G_y$ which come from $T^+$ and $G_y^-$ for the points coming from $T^-$. To ease notation, for $x\in G$ we also set $\bar{x}:=\pi(x)\in T$.

Now let $\sigma=(\sigma_0,\epsilon_\sigma)$ be an automorphism of $T$. We define the automorphism $\GG(\sigma)$ of $G$ as follows:
\begin{itemize}
\item
for $x\in G_b$ we set $\GG(\sigma)(x)$ to be the unique point of $G$ lying over $\sigma(\bar{x})$
\item
for $x\in G_y^+$ we set
\[\GG(\sigma)(x)=\begin{cases} \pi^{-1}(\{\bar{x}\})\cap G_y^+~~&~~\epsilon(\bar{x})=1\\\pi^{-1}(\{\bar{x}\})\cap G_y^-~~&~~\epsilon(\bar{x})=-1.\end{cases}\]
\item
similarly, for $x\in G_y^-$  we set
\[\GG(\sigma)(x)=\begin{cases} \pi^{-1}(\{\bar{x}\})\cap G_y^-~~&~~\epsilon(\bar{x})=1\\\pi^{-1}(\{\bar{x}\})\cap G_y^+~~&~~\epsilon(\bar{x})=-1.\end{cases}\]
\end{itemize}
(In the above, for $\bar{x}\in T$ yellow, we write $\epsilon_\sigma(\bar{x})$ for the value of $\epsilon_\sigma$ on the connected component of $T_y$ containing $\bar{x}$.) 

Note that $\sigma$ commutes with the hyperelliptic involution $\iota$.
\end{construction}

\begin{proposition}
If $T$ is a BY tree, then $\GG(T)$ is a hyperelliptic graph. The same 
holds in the open case. Further, the map $\TT(-)$ gives a homomorphism $\Aut T\rightarrow \Aut G$. 
\end{proposition}

\begin{proof}
That $\GG(T)$ is a hyperelliptic graph follows by reversing the argument of Proposition \ref{is a BY tree}. The claim about automorphisms is clear.
\end{proof}

\begin{remark} \label{choice of section in G construction}
As constructed, for a BY tree $T$, $G=\GG(T)$ comes equipped with a canonical section $s:T\rightarrow G$ sending $\bar{x}\in T$ to $x^+$ (in the notation of Construction \ref{TtoG}). In general, suppose that $s':T\rightarrow G$ is any section. Then given $\sigma=(\sigma_0,\epsilon_\sigma)\in \Aut T$ we may define an automorphism $\GG(\sigma;s')\in \Aut G$ by first defining it on the image of $s'$ as \[\GG(\sigma;s')(s'(\bar{x}))=\begin{cases}s'(\sigma_0(\bar{x}))~~&~~\bar{x}~~\text{blue or}~\bar{x}~\text{yellow and }\epsilon_\sigma(\bar{x})=1,\\\iota(s'(\sigma_0(\bar{x})))~~&~~\text{else,}\end{cases}\] 
and then extending to the whole of $G$ by insisting that $\GG(\sigma;s')$ commutes with the hyperelliptic involution. Then $\GG(\sigma;s)$ agrees with $\GG(\sigma)$ as defined in Construction \ref{TtoG}.  Moreover, writing  $\phi_{s,s'}$ for the automorphism of $G$ such that $s'=\phi_{s,s'}\circ s$, we have  \[\GG(\sigma;s')=\phi_{s,s'}\circ \GG(\sigma;s)\circ \phi_{s,s'}^{-1}\] for all $\sigma \in \Aut T$.  
\end{remark}

The following proposition establishes the correspondence between hyperelliptic graphs and BY trees. 

\begin{proposition} \label{by tree hyp graph proposition}
\label{GTgenus}
Let $G$ be a hyperelliptic graph, possibly open and/or metric and let $T=\TT(G)$ be the corresponding BY tree. Then 
\begin{enumerate}
\item[(1)~]
We have an equality of genera $g(G)=g(T)$,
\item[(2)~]
For any choice of section $s$, the map $\TT(-;s)$ gives an isomorphism $\Aut G\stackrel{\sim}{\longrightarrow} \Aut T$.
\item[(3)~] 
We have  $\GG(T)\iso G$ (as metric graphs if $G$ is metric).
\end{enumerate}
Conversely, let $T$ be a BY tree, possibly open and/or metric, and let $G=\GG(T)$ be the corresponding hyperelliptic graph. Then
  \begin{enumerate}
  \item[$(1)^\prime$] 
  We have an equality of genera $g(T)=g(G)$,
  \item[$(2)^\prime$]
  The map $\GG(-)$ gives an isomorphism $\Aut T\stackrel{\sim}{\longrightarrow} \Aut G$,
  \item[$(3)^\prime$]
  We have $\TT(G)\iso T$ (as metric BY trees if $T$ is metric).
  \end{enumerate}
 \end{proposition}
  

\begin{proof}
First let $G$ be a hyperelliptic graph and $T$ the associated BY tree. For concreteness we consider the closed non-metric case, the argument being identical in the other cases. It is clear from the constructions that we have $\GG(T)\iso G$ (non-canonically). 
We will show later in Proposition \ref{hyp graph homology} that there is an isomorphism  $H_1(G)  \iso H_1(T,T_b)$.
Since $\pi: G\to T$ is also a bijection on the vertices of positive genus, we have $g(T)=g(G)$ (cf. Definitions \ref{hggenus}, \ref{bygenus}). We have now established (1) and (3). 

To show (2), since changing the section $s$ serves to conjugate $\TT(-;s)$, it suffices to prove the result for a single section $s$, which we now fix. If $\sigma \in \Aut G$ is such that $\TT(\sigma;s)$ is the trivial automorphism of $T$, then  $\sigma$ acts trivally on the quotient $G/\left \langle \iota \right \rangle$ and also preserves the section $s$. Such an automorphism is easily seen to be the identity so $\TT(-;s)$ is injective. To show surjectivity  it now suffices to show that $|\Aut T|\leq |\Aut G|$. By (3), it suffices to show that $\GG(-):\Aut T\rightarrow \Aut \GG(T)$ is injective  which we do independently below. Modulo this remaining claim, this completes the proof of (1), (2) and (3).

Now let $T$ be a BY tree (say closed and non-metric) and $G=\GG(T)$. Again, it is clear from the construction that we have $\TT(G)\iso T$ (and now the isomorphism is canonical) so $(3)^\prime$ is proven. Part $(1)^\prime$ now follows from $(3)^\prime$ and (1).

We now show $(2)^\prime$. Let $\sigma=(\sigma_0,\epsilon_\sigma)$ be an automorphism of $T$. If $\GG(\sigma)$ is trivial then it is immediate from the definition that  $\GG(\sigma)$ is an automorphism of $G$ that commutes with the hyperelliptic involution and acts trivially on $s(T)$. Since such automorphisms are necessarily trivial, $\GG(\sigma)$ is injective (which also completes the proof of (2)). As above, the isomorphism of $(3)^\prime$ along with the injectivity of $\TT(-):G\rightarrow \TT(G)$ shown previously forces $\GG(\sigma)$ to be an isomorphism. 
\end{proof}

\begin{remark} \label{signed hyperelliptic graphs}
It is tempting to define `signed hyperelliptic graphs' as pairs $(G,s)$ where $G$ is a hyperelliptic graph (possibly open) and $s:G/\left \langle \iota \right \rangle \rightarrow G$ is a section to the quotient map, with isomorphisms between two such pairs $(G,s)$ and $(G',s')$ required to take $s$ onto $s'$. It is easy to see that two such pairs $(G,s)$ and $(G',s')$ are isomorphic if and only if $G$ and $G'$ are, and that $\GG$ and $\TT$ give a one-to-one correspondence between signed hyperelliptic graphs and BY trees in which there are no choices involved in identifying automorphism groups (automorphisms of signed hyperelliptic graphs which forget the section must be allowed though). Additionally, the isomorphism $G\iso\GG(\TT(G))$ becomes canonical. We have decided against setting up the correspondence this way, however, as for our application to hyperelliptic curves, the hyperelliptic graphs we obtain do not come with a natural section $s$ and so it is convenient to allow an arbitrary choice.
\end{remark}

\subsection{Cluster pictures $\leftrightarrow$ BY trees}
\label{sssto}

We now construct the maps between cluster pictures and open BY trees.

\begin{construction}[$\TT(\Sigma)$]
\label{StoT}~

\noindent {\em{Objects}}:
Let $(X,\Sigma)$ be a cluster picture.
Define $T=\TT(\Sigma)$ to be the open BY tree whose vertices are
\begin{itemize}
\item 
one vertex $v_\s$ for every proper cluster $\s$ which is not a twin, coloured yellow
if $\s$ is \"ubereven and blue otherwise.
\item
one blue vertex (a leaf) $v_\t$ for every twin $\t$.
\end{itemize}
For the edges 
\begin{itemize}
\item 
for every pair $\s'<\s$ (see Definition \ref{children}) with $\s'$ proper,  link $v_\s'$ and $v_\s$ by an edge, yellow if $\s'$ is even and blue otherwise.
\item
add one open edge $v_X\infty$, yellow if $X$ is even and blue otherwise.
\end{itemize}
In the metric version, set the length to be $\delta(\s,\s')$ for blue edges and $2\delta(\s,\s')$ for yellow edges.

Finally, define the genus of a vertex $v_\s$ to be the genus of the cluster $\s$ as in Definition \ref{genus of cluster}. 
\smallskip

\noindent {\em{Automorphisms}}: Let $\sigma=(\sigma_0,\epsilon_{\sigma})$ be an element of $\Aut\Sigma$, where $\sigma_0$ is viewed as a permutation of the proper clusters. Then we define an automorphism $\TT(\sigma)=(\TT(\sigma)_0,\epsilon_{\TT(\sigma)})$ of $T=\TT(\Sigma)$ as follows. For a vertex $v_\s$ of $T$, set $\TT(\sigma)_0(v_\s)=v_{\sigma(\s)}$. To define $\epsilon_{\TT(\sigma)}$, for a yellow component $T_y$ of $T$, pick a yellow edge from $v_\s$ to $v_{\s'}$ in $T_y$  where we take $v_{\s'}$ to be nearer to $\infty$ (since yellow vertices have only yellow edges, each component of $T_y$ has at least one yellow edge). Then $\s$ is an even cluster and we set $\epsilon_{\TT(\sigma)}(T_y)=\epsilon_{\sigma}(\s)$. The compatibility of signs on even clusters as in Definition \ref{defclpic} ensures that this is well defined. 
\end{construction}

\begin{proposition}
\label{TTStrue}
Let $(X,\Sigma)$ be a cluster picture. Then 
\begin{itemize}
\item[(1)] $\TT(\Sigma)$ is an open BY tree,
\item[(2)] The map $\TT(-)$ defines a homomorphism $\Aut\Sigma \rightarrow \Aut\TT(\Sigma)$.
\end{itemize}
\end{proposition}

\begin{proof}
(1) 
Let us check the conditions of a BY tree (Definition \ref{opendefbytree}).
When $|X|\le 2$, this is easy to check by hand. Otherwise:

\emph{Yellow vertices have genus 0, degree $\ge 3$, and only yellow edges:} 
\"ubereven clusters have genus 0 and at least 2 children, giving at least 2 edges
and an edge to the parent or to $\infty$. So there are at least 3 outgoing edges, all of which are yellow
since all children of an \"ubereven cluster are even.

\emph{Blue vertices of genus 0 have at least one yellow edge:} every even cluster has a yellow
edge to its parent or $\infty$, and every odd cluster of size $>2$ has either positive genus
or at least one even child that gives a yellow edge.

\emph{At every vertex, $2g(v)+2\ge$ \# blue edges at $v$: } 
every vertex $v$ comes either from a twin or from a cluster $\s$ of size $>2$. 
In the former case, $v$ has no blue edges. In the latter,

$\bullet$ if $\s$ is odd, it has an odd number (=$2g(\s)+1$) of odd children, and

\begin{tabular}{llllll}
$2g(\s)+2$ &=& 1 + \#\{odd children\} $\ge$ 1 + \#\{odd proper children\} \cr 
           &=& \# blue edges at $v$.
\end{tabular}

$\bullet$ if $\s$ is even, it has an even number
(=$2g(\s)+2$) of odd children, and

\begin{tabular}{llllll}
$2g(\s)+2$ &=& \#\{odd children\} $\ge$ \#\{odd proper children\} \cr 
           &=& \# blue edges at $v$.
\end{tabular}

(2) If $\sigma=(\sigma_0,\epsilon_\sigma)$ is an automorphism of $\Sigma$ then $\sigma_0$ preserves inclusion and the size of proper clusters by definition. Thus $\TT(\sigma)$ preserves adjacency, colour and genus and it is now clear that $\TT(\sigma)$ is an automorphism of $\TT(\Sigma)$. That the map $\TT(-)$ is an homomorphsim follows from the way we have defined composition of automorphisms for cluster pictures and BY trees (cf. Definitions \ref{autb}, \ref{autc}). 
\end{proof}


The construction in the opposite direction is as follows.

\begin{construction}[$\SS(T)$] \label{TtoS} ~

\noindent {\em{Objects}}:
Let $T$ be an open (possibly metric) BY tree. Define a partial order on the vertices of $T$ by setting $v\preceq v'$ if $v'$ lies on the unique shortest path from $v$ to $\infty$. 

For each blue vertex $v\in T_b$, let $\deg_{T_b}(v)$ be the number of blue edges at $v$ (i.e. the degree of $v$ in $T_b$) and set $m_v=2g(v)+2-\deg_{T_b}(v)$, which is non-negative by Definition \ref{opendefbytree} (3). Take $m_v$ singletons $x_{v,1},...,x_{v,m_v}$ and define $X_v:=\{x_{v,1},...,x_{v,m_v}\}$. Now take  $X=\bigcup_{v\in T_b}X_v$. 

Further, for every vertex $v$ of $T$ (of any colour), set 
$$
  \s_v=\bigcup_{v'\preceq v,~v'\textup{ blue}}X_{v'}
$$ 
and define the subset $\Sigma\subseteq \cP(X)$ as
$$
  \Sigma= \bigcup_{v\in T}\{\s_v\}\cup \bigcup_{x\in X}\{x\}.
$$
Set $\SS(T)=(X,\Sigma)$, the cluster picture associated to $T$. 

In the metric case, for $e\in E(T)$, write
\[l(e)=\begin{cases} \delta(e)~~&~~e \textup{ blue,}\\ \frac{1}{2}\delta(e)~~&~~\text {$e$ yellow,}\end{cases}\]
and extend to a distance function of $T$ in the obvious way. Now for vertices $v,w\in T$, define $\delta(\s_v,\s_w)=l(v,w)$. 
\smallskip

\noindent {\em{Automorphisms}}: Let $\sigma=(\sigma_0,\epsilon_\sigma)\in \textup{Aut}(T)$. Then we define an element $\SS(\sigma)=(\SS(\sigma)_0,\epsilon_{\SS(\sigma)})$ of  $\Aut\Sigma$ as follows. Noting that the map $v\mapsto \s_v$ is a bijection between the vertices of $T$ and the proper clusters of $\Sigma$, define a permutation $\SS(\sigma)_0$ of the proper clusters of $\Sigma$ by setting $\SS(\sigma)_0(\s_v)=\s_{\sigma_0(v)}$. Since $\s_v\subseteq \s_{v'}$ if and only if $v \preceq v'$, this preserves inclusion. To define $\epsilon_{\SS(\sigma)}$, let $\s\in \Sigma$ be an even cluster. Then $\s=\s_v$ for a vertex $v$ of $T$ and we'll see in Corollary \ref{genus of subtree} (2) below that the edge from $v$ towards $\infty$ is yellow. Writing $Z$ for the connected component of $T_y$ containing this edge, we define $\epsilon_{\SS(\sigma)}(\s)=\epsilon_\sigma(Z)$.
\end{construction}

\begin{proposition}
Let $T$ be an open BY tree. Then $\SS(T)$ is a cluster picture. Moreover, $\SS(-)$ defines a homomorphism $\Aut T\rightarrow \Aut \SS(T)$. 
\end{proposition}

\begin{proof}
In both cases this is clear from Construction  \ref{TtoS}. 
\end{proof}


Given a vertex $v$ of an open BY tree $T$, it is not obvious from Construction \ref{TtoS} how the size of the associated cluster $\s_v$ relates to invariants of $T$ and $v$. The following two results explain this, as well as showing that Construction \ref{TtoS} preserves genus.

\begin{proposition}  \label{TSgenus}
Let $T$ be an open BY tree, $e$ its unique open edge, and $\SS(T)=(X,\Sigma)$ the associated cluster picture. 
Then 
\[|X|=\begin{cases} 2g(T)+ 2~~&~~\text{if }e\text{ is yellow,}\\ 2g(T)+1~~&~~\text{if }e\text{ is blue.}\end{cases}\]
In particular, $g(T)=g(\Sigma)$, and $|X|$ is even if and only if $e$ is yellow.
\end{proposition}

\begin{proof}
By Construction \ref{TtoS} we see that
\[|X|=\sum_{v\in T_b}\left(2g(v)+2-\text{deg}_{T_b}(v)\right).\]
 Now since yellow vertices have genus $0$,  we may split this sum as
\[|X|=\sum_{v\in T} 2g(v)+\sum_{v\in T_b}\left( 2-\text{deg}_{T_b}(v)\right).\]
If $e$ is yellow then $T_b$ is a disjoint union of closed connected trees, hence
\[\sum_{v \in T_b}\left(2-\text{deg}_{T_b}(v)\right)=2|V(T_b)|-2|E(T_b)|=2\#\{\text{connected comps. of }T_b\}.\]
On the other hand, if $e$ is blue, it is counted one fewer times in the sum. We thus obtain
\[|X|=2\left(\sum_{v\in T}g(v)+\#\{\text{connected comps. of }T_b\}\right)-\begin{cases} 0~~&~~\text{if }e \text{ is yellow,}\\1~~&~~\text{if }e \text{ is blue.}\end{cases}\] Since $\text{rk}H_1(T,T_b)$ is equal to one less than the number of connected components of $T_b$ (see Remark \ref{reduced homology remark}) the result follows.
\end{proof}

\begin{cor}  \label{genus of subtree}
Let $T$ be an open BY tree and $\Sigma=\SS(T)$ the associated cluster picture. Fix a vertex $v\in T$ and (in the notation Construction \ref{TtoS}) let $\mathfrak{s}_v\in \Sigma$ be the associated cluster. Further, denote by $e_v$ the edge from $v$ towards $\infty$. Then:
\begin{itemize}
\item[(1)] We have  \[|\s_v|=\begin{cases} 2g(T_v)+2~~&~~\text{if }e_v \text{ is yellow,}\\2g(T_v)+1~~&~~\text{if }e_v \text{ is blue,}\end{cases}\]
where here $T_v$ denotes the open BY tree generated by the vertices $v'\preceq v$ of $T$ along with the open edge $e_v$,
\item[(2)] The cluster $\s_v$ is even if and only if $e_v$ is yellow, and \"ubereven if and only if $v$ itself is yellow,
\item[(3)] We have an equality of genera $g(\s_v)=g(v)$.
\end{itemize}
\end{cor}

\begin{proof}
The claims are easy to check when $T$ has genus $0$ or $1$, so assume the genus 
is at least 2.

(1). Apply Proposition \ref{TSgenus} to the open BY tree $T_v$. 

(2). That $\s_v$ is even if and only if $e_v$ is yellow is clear from (1).  Next, since yellow vertices have only yellow edges and no associated singletons, it is clear that if $v$ is yellow then $\s_v$ is \"ubereven. For the converse it is convenient to first observe that if $v$ is blue then the number of odd children of $\s_v$ is either $2g(v)+2$ or $2g(v)+1$, the former case occuring if and only if $e_v$ is yellow. Indeed, by Construction \ref{TtoS} the number of children of $\s_v$ of size $1$ is given by $m_v=2g(v)+2-\text{deg}_{T_b}(v)$, whilst part (1) applied to the vertices adjacent to $v$ shows that the number of odd proper children of $\s_v$ is given by the number of blue edges at $v$, excluding the edge $e_v$ towards infinity (should this be blue). 

Since $g(v)$ is non negative, it now follows that if $v$ is blue then $\s_v$ cannot be \"ubereven. 

(3). If $v\in T$ is yellow then $\s_v$ is \"ubereven by (2) and both $g(v)$ and $g(\s_v)$ are $0$. Suppose now that $v\in T$ is blue.  As above, the number of odd children of $\s_v$ is either $2g(v)+1$ or $2g(v)+2$. It is now  immediate from the definition of the genus of a cluster that $g(\s_v)=g(v)$ as claimed. 
\end{proof}

The following Proposition completes the proof of Theorem \ref{combmain1}.

\begin{proposition} \label{cluster picture by tree proposition}
Let $T$ be an open (possibly metric) BY tree and $(X,\Sigma)=\SS(T)$ be the associated cluster picture. Then
\begin{enumerate}
\item[(1)~]
we have an equality of genera $g(T)=g(\Sigma)$.
\item[(2)~]
The map $\sigma \mapsto \SS(\sigma)$ gives an isomorphism $\Aut T\stackrel{\sim}{\longrightarrow} \Aut\Sigma$.
\item[(3)~] 
We have  $\TT(\Sigma)\iso T$ (as metric BY trees if $T$ is metric).
\end{enumerate}
Conversely, let $(X,\Sigma)$ be a (possibly metric) cluster picture and let $T=\TT(\Sigma)$ be the associated open BY tree. Then 
\begin{enumerate}
\item[$(1)^\prime$]
we have an equality of genera $g(\Sigma)=g(T)$.
\item[$(2)^\prime$]
The map $\sigma \mapsto \TT(\sigma)$ gives an isomorphism $\Aut\Sigma \stackrel{\sim}{\longrightarrow} \Aut T$.
\item[$(3)^\prime$] 
We have  $\SS(T)\iso \Sigma$ (as metric cluster pictures if $\Sigma$ is metric).
\end{enumerate}
\end{proposition}

\begin{proof}
We consider the non-metric case throughout, the metric case being an easy extension. First let $T$ be an open BY tree and $(X,\Sigma)=\SS(T)$ the associated cluster picture. Part $(1)$ was shown previously in Proposition \ref{TSgenus}. 

Next we show part (3). In the notation of Constructions \ref{StoT} and \ref{TtoS}, consider the map $f:T\rightarrow \TT(\Sigma)$ sending $v\in T$ to $v_{\mathfrak{s}_v}$. It is clear from the constructions that this is a graph theoretic isomorphism. Moreover, by  Construction \ref{StoT} and Corollary \ref{genus of subtree} (2) we see that $f$ preserves colour (of both edges and vertices). Finally, to see that $f$ preserves genus, fix $v\in T$. Since for $\s\in \Sigma$ we defined the genus of $v_\s$ to be the genus of the cluster $\s$, it suffices to show that $g(v)=g(\s_v)$ for each $v\in T$, which is Corollary \ref{genus of subtree} (3). 

To show $(2)$, let $\sigma=(\sigma_0,\epsilon_{\sigma})\in \Aut T$ and suppose that $\SS(\sigma)$ is trivial in $\Aut\Sigma$. Then as $v\mapsto \s_v$ is a bijection between the vertices of $T$ and the proper clusters of $\Sigma$, $\sigma_0$ fixes every vertex of $T$. Moreover, if $\epsilon_{\SS(\sigma)}$ is trivial on each even cluster then $\epsilon_\sigma$ must be trivial on each yellow edge and is then itself trivial. This shows that $\SS(-)$ is injective and, in particular, that we have $|\Aut T|\leq |\Aut\Sigma|$. To show that $\SS(-)$ is an isomorphism, it suffices to show that we also have the reverse inequality. In light of (3), it suffices to show that the map $\TT(-):\Aut\Sigma \rightarrow \Aut \TT(\Sigma)$ is injective, which we do below. 

We now turn to $(1)^\prime$, $(2)^\prime$ and $(3)^\prime$, for which we fix a cluster picture $(X,\Sigma)$ and let $T=\TT(\Sigma)$ be the associated BY tree. We first show $(3)^\prime$.  In the notation of Constructions \ref{StoT} and \ref{TtoS}, we'll show that the map $h:\s\mapsto \s_{v_\s}$ is an isomorphism of cluster pictures. It is clear that it gives a bijection on proper clusters which preserves inclusion. To complete the argument, we prove by induction that $|\s|=|\s_{v_\s}|$ for all proper clusters $\s$. 

First suppose that $\s$ is a minimal proper cluster, i.e. all its children are singletons. Then by the definition of the genus of $\s$ we have
$$
|\s|=
   \begin{cases}
     2g(\s)+1~~&~~\s\text{ odd,}\\2g(\s)+2~~&~~\s\text{ even.}
   \end{cases}
$$
Now $v_\s$ is blue and has no children in $T$ by minimality of $\s$. Moreover, its parent edge is yellow if $\s$ is even and blue if $\s$ is odd. In particular we have 
$$
  \text{deg}_{T_b}(v_\s)=
    \begin{cases}
      1~~&~~\s\text{ odd,}\\0~~&~~\s\text{ even,}
    \end{cases}
$$
whence $m_{v_\s}=|\s_{v_\s}|=|\s|$ as desired (here, as in Construction \ref{TtoS}, for any vertex $v$ of $T$ we set $m_v=2g(v)+2-\text{deg}_{T_b}(v)$).  

Next, take a proper cluster $\s\in \Sigma$ and suppose that we have $|\s'|=|\s_{v_{\s'}}|$ for all proper clusters $\s'\subseteq \s$. In particular, to show that $|\s|=|\s_{v_\s}|$, it suffices to show that $\s$ and $\s_{v_\s}$ both have the same number of children of size $1$. Now by Corollary \ref{genus of subtree} and the discussion on genera in the proof of (3), it follows that $\s$ and $\s_{v_\s}$ have the same genus and parity. Combining the inductive hypothesis with the observation that the genus and parity of a cluster together determine how many odd children it has (cf. Definition \ref{genus of cluster}) completes the proof.

Part $(1)^\prime$ now follows upon combining $(1)$ and $(3)^\prime$. 

Finally, we show $(2)^\prime$. In light of $(3)^{\prime}$ and the injectivity of the map in (2) shown above, it suffices to show that $\sigma \mapsto \TT(\sigma)$ is injective (and this also completes the proof that the map in (2) is an isomorphism). This follows by noting that $\s\mapsto v_{\s}$ is a bijection between proper clusters of $\Sigma$ and vertices of $T$, and that even sized clusters give vertices whose parent edge is yellow, so triviality of $\epsilon_{ \TT(\sigma)}$ forces triviality of $\epsilon_{\sigma}$.  
 \end{proof}

\phantom{X}
\stepcounter{equation}
\begin{table}[!htb]
\caption{Dictionary for the correspondence (open case, genus $\ge 1$)}
\noindent\hskip-17mm
\begin{tabular}{|l|l|l|}
\hline
Cluster picture $\Sigma$ \phantom{${X^{X^X}}$}& Open BY tree  $T$              &  Open hyperelliptic graph $G$  \cr
\hline
\hline
&&\cr
cluster $\s$ with $|\s|>2$    & vertex $w_\s$ not a genus 0 leaf   & $\iota$-orbit of vertices $\{v_\s\}$ or $\{v_\s^+,v_\s^-\}$  \cr
of genus $g$ & of genus $g$& of genus $g$ \cr
&with parent edge $a_\s$&with $\iota$-orbit of parent edge(s) $p_\s$ or $\{p_\s^+,p_\s^-\}$\cr
...................................&...........................................&....................................................................\cr
$\bullet$ $\s$ odd & $\bullet$ $w_\s$ blue $a_\s$ blue  & $\bullet$ $v_\s$ with parent edge $p_\s$ \cr
...................................&...........................................&....................................................................\cr
$\bullet$ $\s$ even non-\ub & $\bullet$ $w_\s$ blue $a_\s$ yellow  & $\bullet$ $v_\s$ with parent edges $p_\s^+,p_\s^-$ \cr
...................................&...........................................&....................................................................\cr
$\bullet$ $\s$ \ub &$\bullet$  $w_\s$ yellow $a_\s$ yellow  &$\bullet$  $v_\s^+,v_\s^-$ with parent edges $p_\s^+,p_\s^-$ \cr
&&\cr
\hline
&&\cr
$\s<\s'$ & $a_\s =w_\s w_{\s'}$& $p_\s = v_\s v_{\s'}$ or \cr
&&$p_\s^{\pm} = v_\s v_{\s'}$ or $v_\s^{\pm} v_{\s'}$ or  $ v_\s v_{\s'}^{\pm} $ or $ v_\s^{\pm} v_{\s'}^{\pm} $\cr
$\delta(\s,\s')=\leftchoice{2d}{\s \mbox{ odd}}{d}{\s \mbox{ even}}$ & length $2d$& length $d$\cr
&&\cr
\hline
&&\cr
twin $\t$    & genus 0 leaf $u_\t$   &  an edge $\l_\t$\cr
$\t<\s$&with yellow edge $a_{\t}$ to $w_\s$&   $\l_\t =v_\s v_\s$ or $\l_t = v_\s^+ v_\s^-$\cr
$\delta(\t, \s)=d/2$& of length $d$ & of length $d$\cr
&&\cr
\hline
\hline
$(\sigma,\epsilon) \in\Aut\Sigma$ \phantom{${X^{X^X}}$}&$(\sigma,\epsilon) \in\Aut T$&$\sigma \in \Aut G$ \cr
\hline
&&\cr
$\sigma : \s \mapsto \s_2$& $\sigma: w_{\s} \mapsto w_{\s_2}$ & $\sigma: v_{\s} \mapsto v_{\s_2}$ or $\sigma:v_{\s}^{\pm} \mapsto v_{\s_2}^{\pm \epsilon(v_{\s}^*)}$ \cr
&&$\sigma:p_{\s}\mapsto p_{\s_2}$ or $\sigma:p_{\s}^{\pm} \mapsto p_{\s_2}^{\pm \epsilon(v_{\s}^*)}$\cr
&&\cr
$\epsilon(\s)$, \quad for $\s$ even&$\epsilon(a_\s)$,\quad for $a_\s$ yellow &  $\epsilon(v_\s^*)\in \{\pm 1\}$ \cr
&&\cr
\hline
&&\cr
$\sigma : \t \mapsto \t_2$, $\epsilon(\t)$& $\sigma: u_{\t} \mapsto u_{\t_2}$, $\epsilon(a_\t)$ & $\sigma: \l_{\t} \mapsto \epsilon(\l_\t)\l_{\t_2}$,  $\epsilon(\l_\t)\in \{\pm 1\}$\cr 
&&(where $-\l_{\t}$ is $\l_{\t}$ with reversed orientation)\cr
&&\cr
\hline
\end{tabular}
\label{tabdicto}
\end{table}

\subsection{Summary of constructions and examples}
The one-to-one correspondence given by Theorem \ref{combmain1} is easy to use in practice. Table \ref{tabdicto} summarizes the constructions (it follows from Constructions \ref{GtoT}, \ref{TtoG}, \ref{StoT} and Proposition \ref{cluster picture by tree proposition}). We illustrate how to use it in Examples \ref{WM} and \ref{GM}. 

As in Remark \ref{sections section}, the hyperelliptic graph described in the third column of Table \ref{combmain1} comes with the decomposition $G_y = G_y^+  \coprod G_y^-$, where $G_y^+$ consists of all edges and vertices denoted with a $+$. (Thus in order to construct automorphisms of a cluster picture or a BY tree from that of a hyperelliptic graph, it is first necessary to pick such a decomposition.) 

Recall that in a BY tree, from every vertex $v$ there is a shortest path towards $\infty$. The \emph{parent edge} of $v$ is the edge $a$ incident to $v$ on this path.
Similarly, for a vertex $v$ of a hyperelliptic graph, a \emph{parent edge} is an edge incident to $v$ on one of the shortest paths towards $\infty$. 

\begin{example}\label{WM}
Consider the open hyperelliptic graph $G$ and the open BY tree $T$ together with the automorphisms $\sigma_G$ and $(\alpha_T, \epsilon_T)$ of Examples \ref{Windmill} and  \ref{Miller}:

\WINDMILL
\qquad\qquad
\raise 1.7em\hbox{\MILLER}

 
%

Following Table \ref{tabdicto}, from $G$ we form the associated BY tree by creating:
\begin{itemize}
\item blue vertices $w_1, x, w_2$ of genera 2,0,0 corresponding to the $\iota$-invariant vertices $v_1, v_x, v_2$ of $G$,
\item a yellow  vertex $w_3$ corresponding to $\{v_3^+,v_3^-\}$, 
\item blue genus 0 leaves $u_1, u_2, u_3, u_4$ corresponding to the loops $\l_1, \l_2$ and the $\iota$-anti-invariant edges $\l_3, \l_4$,
\item a blue edge from $w_1$ to $x$ of length 2 corresponding to $\iota$-invariant edge $e_1$, 
\item yellow edges from $w_2$ to $x$ and $w_3$ to $w_2$ of lengths 2,2 corresponding to the $\iota$-orbits $\{e_2^+, e_2^-\}$  and $\{e_3^+, e_3^-\}$,
\item yellow edges from $u_1$ to $x$, $u_2$ to $x$, $u_3$ to $w_3$, $u_4$ to $w_3$ of lengths 5,5,6,6 corresponding to $\l_1, \l_2, \l_3, \l_4$, 
\item a yellow open edge from $x$ to $\infty$ corresponding to the two open edges from $v_x$ to $\infty$.
\end{itemize}
This construction precisely yields $T$.

To compute the automorphism $(\alpha,\epsilon)$ corresponding to $\sigma_G$, consider the hyperelliptic graph $G$ with the decomposition of its $\iota$-permuted part $G_y = G_y^+ \coprod G_y^-$, 
where $G_y^+$ consists of $e_2^+, e_3^+, v_3^+, e_{\infty}^+$ and the top halves of $\l_1$, $\l_2$, $\l_3$ and $\l_4$ (call these $\l_1^+, \l_2^+, \l_3^+, \l_4^+$).

Then $(\alpha, \epsilon)$ is given as follows: 
$\alpha$ acts on $T$ by swapping $u_1$ and $u_2$, since $\sigma_G$ swaps $\l_1$ and $\l_2$ and fixes all other $\iota$-orbits, and
\begin{align*}
\sigma_G(\l_1^+) = \l_2^+ &\implies \epsilon (u_1 x) = 1, \\
\sigma_G(\l_2^+) = \l_1^- &\implies \epsilon (u_2  x) = -1,  \\
\sigma_G(e_{\infty}^+) = e_{\infty}^+ &\implies \epsilon (x  \infty) = 1,\\
\sigma_G(e_2^+) = e_2^+ &\implies \epsilon (x  w_2) = 1,\\
\sigma_G(v_3^+) = v_3^- &\implies \epsilon(u_3  w_3) = \epsilon(u_4 w_3) = \epsilon(w_2  w_3) = \epsilon( w_3)=-1.
\end{align*}
In other words, we obtain $(\alpha,\epsilon)=(\alpha_T, \epsilon_T)$.

From $T$, to form the associated hyperelliptic graph, we create: 
\begin{itemize}
\item vertices $v_1, v_x, v_2$ fixed by $\iota$ of genera 2,0,0, corresponding to the blue vertices $w_1, x, w_2$, 
\item vertices $ v_3^+, v_3^-$ swapped by $\iota$, of genera 0,0 corresponding to the yellow vertex $w_3$, 
\item two edges $e_{\infty}^+, e_{\infty}^-$ swapped by $\iota$, corresponding the open yellow edge from $x$ to $\infty$, 
\item an edge from $v_1$ to $v_x$ of length 1, corresponding to the blue edge from $w_1$ to $x$, 
\item two edges $e_2^+$, $e_2^-$ swapped by $\iota$, of lengths 1, corresponding to the yellow edge from the blue vertex $w_2$ to the blue vertex $x$, 
\item two edges $e_3^+$, $e_3^-$ swapped by $\iota$, of lengths 1, corresponding to the yellow edge from the yellow vertex $w_3$ to the blue vertex $w_2$,  
\item two $\iota$-anti-invariant edges $\l_1, \l_2$ from $v_x$ to itself of lengths 5, corresponding to the yellow edges from the blue vertex $x$ to the genus 0 leaves $u_1$ and $u_2$, 
\item two $\iota$-anti-invariant edges $\l_3, \l_4$ from $v_3^+$ to $v_3^-$ of lengths 6, corresponding to the yellow edges from the yellow vertex $w_3$ to the genus 0 leaves $u_3$ and $u_4$. 
\end{itemize}
We obtain precisely $G$. 

The automorphism $\sigma$ corresponding to $(\alpha_T, \epsilon_T) \in \Aut(T)$ is given by
\begin{itemize}
\item $\sigma$ fixes $v_1, v_x, v_2$ and preserves $\{v_3^+, v_3^-\}$ since $\alpha_T$ fixes $w_1, x, w_2$ and $w_3$, 
\item $\sigma$ fixes $e_{\infty}^+, e_{\infty}^-, e_2^+, e_2^-$ since $\alpha_T$ fixes $w_2x$, $x \infty$ and $\epsilon_T$ is 1 on both edges, 
\item $\sigma$ swaps $e_3^+$ and $e_3^-$, $v_3^+$ and $v_3^-$ and changes the orientation of $\l_3$ and $\l_4$ since $\alpha_T$ fixes $w_3 w_2$, $w_3$, $u_3 w_3$, $u_4  w_3$ and $\epsilon_T$ is $-1$ on all of them,
\item $\sigma$ maps $\l_1$ to $\l_2$ since $\alpha_T$ maps $u_1$ to $u_2$ and $\epsilon_T (u_1  x)=1$,
\item $\sigma$ maps $\l_2$ to $-\l_1$ since   $\alpha_T$ maps $u_2$ to $u_1$ and $\epsilon_T (u_2  x)=-1$ (where $-\l$ denotes a loop $\l$ with opposite orientation).
\end{itemize}
We then obtain $\sigma=\sigma_G$.
\end{example}

\begin{example}\label{GM}
Consider the cluster picture $\Sigma$ and the open BY tree $T$ together with the automorphisms $(\alpha_{\Sigma}, \epsilon_{\Sigma})$ and $(\alpha_T, \epsilon_T)$ from Examples \ref{Grain} and \ref{Miller}: \\
 \raise 4em\hbox{\GRAINthree}
\quad
\hbox{\MILLER}

Following Table \ref{tabdicto}, from $\Sigma$ we construct the associated BY tree by creating:
\begin{itemize}
\item blue vertices $w_1, x, w_2$ of genera 2, 0, 0, corresponding to the non-\ub\ clusters $ \s_1, X, \s_2$ of size $>2$, 
\item a yellow vertex $w_3$ of genus 0 corresponding to the \ub\ cluster $\s_3$, 
\item  blue genus 0 leaves $u_1, u_2, u_3, u_4$ corresponding to the twins $\t_1, \t_2, \t_3, \t_4$, 
\item a yellow open edge from $x$ to $\infty$, since the top cluster $X$ is even, 
\item a blue edge from $w_1$ to $x$ of length 2, since $\s_1$ is odd, $\s_1<X$ and $\delta(\s_1,X)=2$, 
\item yellow edges from $u_1$ to $x$, $u_2$ to $x$, $w_2$ to $x$, $w_3$ to $w_2$, $u_3$ to $w_3$ and $u_4$ to $w_3$ of lengths 5,5,2,2,6,6, since $\t_1, \t_2,  \s_2, \s_3, \t_3, \t_4$ are even, $\t_1<X$, $\t_2<X$,  $\s_2<X$, $\s_3<\s_2$, $\t_3<\s_3$, $\t_4<\s_3$, and $\delta(\t_1, X) = \frac{5}{2}$, $\delta(\t_2, X) = \frac{5}{2}$, $\delta(\s_3, \s_2) =1$, $\delta(\s_2, X) =1$, $\delta(\t_3, \s_3) =3$ and $\delta(\t_4, \s_3) =3$.
\end{itemize}
This yields $T$.
The automorphism $(\alpha, \epsilon)$ corresponding to $(\alpha_{\Sigma}, \epsilon_{\Sigma})$ is given by 
\begin{itemize}
\item $\alpha$ swaps $u_1$ and $u_2$ and fixes all other vertices, since $\alpha_{\Sigma}$ swaps $\t_1$ and $\t_2$ and fixes all other clusters, 
\item $\epsilon$ is $1$ on the yellow edges $x \infty$, $u_1  x$, $w_2  x$, since $\epsilon_{\Sigma}(X) = \epsilon_{\Sigma}(\t_1) =\epsilon_{\Sigma}(\s_2)=1$, 
\item $\epsilon$ is $-1$ on the yellow edges $u_2  x$, $w_3  w_2$, $u_3  w_3$, $u_4  w_3$,  since $\epsilon_{\Sigma}(\t_2) = \epsilon_{\Sigma}(\s_3) =\epsilon_{\Sigma}(\t_3)= \epsilon_{\Sigma}(\t_4)=-1$, 
\item $\epsilon(w_3)=-1$ since it inherits the sign from its parent edge.
\end{itemize}
This shows precisely that $(\alpha,\epsilon)=(\alpha_T, \epsilon_T)$. 

From $T$ we construct the associated cluster picture by creating:
\begin{itemize}
\item clusters $\s_1, \s_2, \s_3, \s_x$ of size $>2$ and genera 2,0,0,0 corresponding to vertices that are not genus 0 leaves $w_1, w_2, w_3, w_x$, 
\item twins $\t_1, \t_2, \t_3, \t_4$ corresponding to the genus 0 leaves $u_1, u_2, u_3, u_4$,
\item $\s_x$ is the top cluster $X$ since $a_x = x \infty$,
\item $\s_1$ is an odd child of $s_x$ of relative depth 2 since $a_1 =w_1w_x$ is blue of length 2,
\item $\s_2$ is an even non \ub\ child of $s_x$ of relative depth 1 since $w_2$ is blue and $a_2 = w_2 w_x$ is yellow of length 2, 
\item $\s_3$ is an \ub\ child of $\s_2$ of relative depth 1 since $w_3$ is yellow and $a_3=w_3w_2$ is yellow of length 2,
\item $\t_1$ and $\t_2$ are children of  $\s_x$ of relative depth $\frac 52$ since $a_{\t_1} = u_1x$ and $a_{\t_2} = u_2x$, both of length 5,
\item $\t_3$ and $\t_4$ are children of  $\s_3$ of relative depth $3$ since $a_{\t_3} = u_3w_3$ and $a_{\t_4} = u_4w_3$, both of length 6,
\item $\s_1$ has 5 roots since it is odd and of genus 2,
\item $\s_3$ has no roots outside $\t_3, \t_4$ since it is \"ubereven,
\item $\s_2$ has 2 roots in addition to $\s_3$ since it is even non-\ub\ of genus 0,
\item $\s_x$ has 1 root in addition to $\s_1, \s_2, \t_1, \t_2$ since it is even of genus 0. 
\end{itemize}
This is precisely $\Sigma$.

The automorphism $(\alpha, \epsilon)$ corresponding to $(\alpha_T, \epsilon_T)$ is given by 
\begin{itemize}
\item $\alpha$ swaps $\t_1$ and $\t_2$ and fixes all other clusters since $\alpha_{T}$ swaps $u_1$ and $u_2$ and fixes all other vertices, 
\item $\epsilon(X) = \epsilon(\t_1) =\epsilon(\s_2) =1$, since $\epsilon_T$ is 1 on the parent edges of $x$, $u_1$ and $w_2$, 
\item $\epsilon(\t_2) = \epsilon(\t_3) = \epsilon(\t_4) =\epsilon(\s_3)=-1$, since $\epsilon_T$ is $-1$ on the parent edges of $u_2, u_3, u_4, w_3$.
\end{itemize}
This yields $(\alpha, \epsilon)=(\alpha_{\Sigma}, \epsilon_{\Sigma})$.\end{example}

\def\GraphScale{0.4}
\def\SnakeWiggle{3pt}
\def\clustersep{1.3pt}
\rootsize{0.5}

\section{One-to-one correspondence (closed case)}\label{s:121c}

In this section we study the notion of equivalence for open hyperelliptic graphs, open BY trees and cluster pictures (see Definitions \ref{equiv def hyp}, \ref{equiv def BY} and \ref{Sequivdef}). 
This enables us to prove a `closed version' of the correspondences of Section \ref{s:121o}.
We also 
\begin{itemize}
\item Explain how to explicitly obtain the core of an open hyperelliptic graph or open BY tree, and address the converse, namely which open BY trees have a specified core (Proposition \ref{Tclosure}, Corollary \ref{core to open}, Table \ref{opentable});
\item Identify a canonical representative in an equivalence class of BY trees, that corresponds to a balanced cluster picture 
(Remark \ref{canonical BY tree rep});
\item Describe `principal clusters' in a cluster picture, that correspond to vertices in the core of the associated hyperelliptic graph;
\item Interpret the moves for equivalence of cluster pictures (Definition \ref{Sequivdef}) in terms of the associated BY tree (proof of Lemma \ref{principal cluster lemma}).
\end{itemize}
The precise statement of the closed correspondence is as follows:

\begin{theorem}
\label{combmain2}
There is a genus-preserving one-to-one 
correspondence, both in the metric and non-metric case, between
\begin{itemize} 
\item[(i)]
Balanced cluster pictures up to isomorphism,
\item[(i$'$)]
Cluster pictures up to equivalence,
\item[(ii)]
Hyperelliptic graphs up to isomorphism,
\item[(ii$'$)]
Open hyperelliptic graphs up to equivalence,
\item[(iii)]
BY trees up to isomorphism,
\item[(iii$'$)]
Open BY trees up to equivalence.
\end{itemize}
Explicitly, the correspondence between hyperelliptic graphs and BY trees is given by the maps \underline{G} and \TT  of Section \ref{s:121o} and similarly, the correspondence between open BY trees, open hyperelliptic graphs and cluster pictures is given by the maps $\underline{G},\TT$ and $\underline{\Sigma}$. Maps $(ii')\rightarrow (ii)$ and $(iii')\rightarrow (iii)$ are given by taking the core, whilst $(i)\rightarrow (i)'$ takes a balanced cluster picture to its equivalence class. 

In genus $\ge 2$, the correspondences (i')$\leftrightarrow$(ii)$\leftrightarrow$(iii) set up bijections between various invariants\footnote
{The definition of cotwins and principal clusters is given in Definition \ref{def principal cotwin}.}
as shown in Table \ref{closed correspondence}.
\end{theorem}

\begin{proof}
The correspondence $(ii)\leftrightarrow (iii)$ was shown previously in Proposition \ref{by tree hyp graph proposition}. Lemma \ref{hyp tree equiv lemma} combined with Proposition \ref{Tclosure} (both shown below) and the open correspondence of Theorem \ref{combmain1} gives $(ii')\leftrightarrow (iii')$. The correspondences $(ii')\leftrightarrow (ii)$ and $(iii')\leftrightarrow (iii)$ follow from the definition of equivalence: Proposition \ref{Tclosure} shows that the core exists and is unique, and Corollary \ref{core to open} 
and Remark \ref{corropengraph}
show that each closed hyperelliptic graph (resp. closed BY tree) arises as the core of some open hyperelliptic graph (resp. open BY tree). Proposition \ref{Sclosure} below shows that two cluster pictures are equivalent if and only if the associated BY trees are, which combined with the open correspondence of Theorem  \ref{combmain1} gives $(i')\leftrightarrow (ii')$. Finally, we show in Lemma \ref{existence of balanced clusters} that each equivalence class of cluster pictures contains a unique balanced one, giving $(i)\leftrightarrow (i')$.

The bijections between the invariants shown in Table \ref{closed correspondence} follow from the explicit description of the correspondences.
\end{proof}

The above result makes no mention of automorphism groups. Since equivalent objects need not have the same automorphism groups, the situation is more delicate (see e.g. Examples \ref{ex1:0g}, \ref{ex1:0t} and \ref{ex1:0s2}). However, the correspondence for balanced cluster pictures
is fairly clean:

\begin{theorem}
\label{closedcorraut}
Let $(X,\Sigma)$ be a balanced (possibly metric) cluster picture and $\tilde G$ be the core of the associated hyperelliptic graph $\GG(\Sigma)$.
Then the natural map $\Aut(\Sigma)\to \Aut(\tilde G)$, given by applying $\GG\circ \TT$ and restricting to the core, is surjective. The kernel is trivial if $X$ is \ub, and $C_2$  if $X$ is non-\ub\  (generated by the trivial permutation with $\epsilon(X)=-1$ and all other signs $+1$).
\end{theorem}
\begin{proof}
This follows from Corollary \ref{AutEquiv} along with the comparison of equivalence between hyperelliptic graphs and BY trees (Lemma \ref{hyp tree equiv lemma} and Proposition \ref{by tree hyp graph proposition}).
\end{proof}

\stepcounter{equation}
\begin{table}[h] 
\caption{Dictionary for the one-to-one correspondence (closed case, genus $\ge 2$)} 
\label{closed correspondence} 
\begin{tabular}{|l|l|l|}
\hline
Hyperelliptic graph & BY tree  & Cluster picture \cr
\hline
\hline
&&\cr
$\iota$-invariant vertices       & blue vertices of genus $g$   & principal non-\"ubereven\cr
of genus $g$                    & that are not genus 0 leaves     & clusters of genus g   \cr
&&\cr
\hline
&&\cr
pairs of vertices                & yellow vertices of genus 0 & principal \"ubereven    \cr
$\{v,\iota(v)\}$                       &                                           & clusters          \cr
&&\cr
\hline
                                          &                                            & twins of distance $\delta$ to \cr 
                                          & genus 0 leaves with the     & a non-\"ubereven parent,   \cr
loops of length $2\delta$   & edge to a blue vertex of    & \qquad\qquad and \cr
                                          &  length $2\delta$                 & cotwins of distance $\delta$\cr 
                                          &                                             & to a non-\"ubereven  child\cr
\hline
                                           &                                           & twins of distance $\delta$ to \cr
$\iota$-anti-invariant edges   & genus 0 leaves with the      & an \"ubereven parent, \cr 
of length $2\delta$, that are        & edge to a yellow vertex      &  \qquad\qquad and \cr
 not loops            & of length $2\delta$               &  cotwins of distance $\delta$ \cr
                                         &                                              &   to an \"ubereven  child\cr
\hline
                                         &                                             & odd principal clusters    \cr
                                         &                                            & of distance $\delta$ to a  \cr
$\iota$-invariant edges of         &                                           & principal parent, \cr
length $\delta/2$        & blue edges of length $\delta$     & \qquad \qquad  and \cr
  & & $X$ if $X=\s_1\coprod \s_2$ with \cr
  & & $\s_1, \s_2$ odd principal of   \cr
  & & distance $\delta$ to each other    \cr
\hline
                                             &                                                      & even principal clusters   \cr
                                             &                                                         & of distance $\delta$ to a  \cr
pairs of edges $\{e,\iota(e)\}$    & yellow edges of length $2\delta$     &  principal parent \cr
of length $\delta$               &not incident to a genus 0                    & \qquad \qquad  and \cr
                                            & leaf                                                   & $X$ if $X=\s_1\coprod \s_2$ with \cr
  & & $\s_1, \s_2$ even principal of   \cr
  & & distance $\delta$ to each other    \cr
\hline
\end{tabular}
\end{table}

\subsection{Equivalence: BY trees, hyperelliptic graphs}


Recall that the core of an open hyperelliptic graph (resp. open BY tree) is its maximal closed hyperelliptic subgraph (resp. closed BY subtree).  Proposition \ref{Tclosure} below shows that the core exists and is unique.
Granted this, we single out vertices that come from the core:
%

\smallskip

\begin{definition}[Principal vertex]
~
\begin{itemize}
\item[(1)]
Let $G$ be an open hyperelliptic graph of genus $\ge 2$, with core $\tilde{G}$. 
A vertex $v$ of $G$ is \emph{principal} if it corresponds to a vertex in the core (i.e. lies in the core and does not become a point on an edge upon removing $G\setminus \tilde{G}$). 
\item[(2)] Let $T$ be an open BY tree of genus $\ge 2$, with core $\tilde{T}$. A vertex $v$ of $T$ is \emph{principal} if it corresponds to a vertex in the core which is not a genus 0 leaf of $\tilde{T}$. 
\end{itemize}
\end{definition}

Still assuming Proposition \ref{Tclosure}, we now show that cores and principal vertices are
preserved under the correspondence:

\begin{lemma} \label{hyp tree equiv lemma}
The correspondences $\underline{G}$ and $\TT$ between open hyperelliptic graphs and open BY trees preserve equivalence. Explicitly, if $T$ is an open BY tree with core $\tilde{T}$, then the core of $\underline{G}(T)$ is isomorphic to $\underline{G}(\tilde{T})$. Similarly, for an open hyperelliptic graph $G$ with core $\tilde{G}$, the core of $\TT(G)$ is isomorphic to $\TT(\tilde{G})$.

Moreover,
let $G$ be an open hyperelliptic graph of genus $\ge 2$, $T=\TT(G)$ the associated BY tree and $\pi:G\rightarrow T$ the quotient map. Then $\pi$ induces a bijection between $\iota$-orbits of principal vertices of $G$ and principal vertices of $T$ (here $\iota\in \Aut G$ is the hyperelliptic involution). 
\end{lemma}

\begin{proof}
Recall that the maps $\underline{G}$ and $\TT$ of Constructions \ref{GtoT} and \ref{TtoG} were defined for both open and closed objects. It follows from the explicit constructions that if $G$ is an open hyperelliptic graph and $G'$ a (closed) hyperelliptic subgraph, then $\TT(G')$ is  canonically a closed BY subtree of $\TT(G)$, and that inclusion of subtrees is preserved by $\TT$. Since the same is also true if we start with an open BY tree and consider the map $\underline{G}$ applied to subtrees of~$T$, we obtain the result. 

For the second part,
it only remains to recall from Construction \ref{GtoT} that vertices of $T$ either arise from $\iota$-orbits of vertices of $G$, or from $\iota$-anti-invariant edges, and that the latter are precisely the genus zero leaves of~$T$.
\end{proof}

\clearpage

\comment

\begin{table}[!htb]        
\begin{center}
\ \ \includegraphics{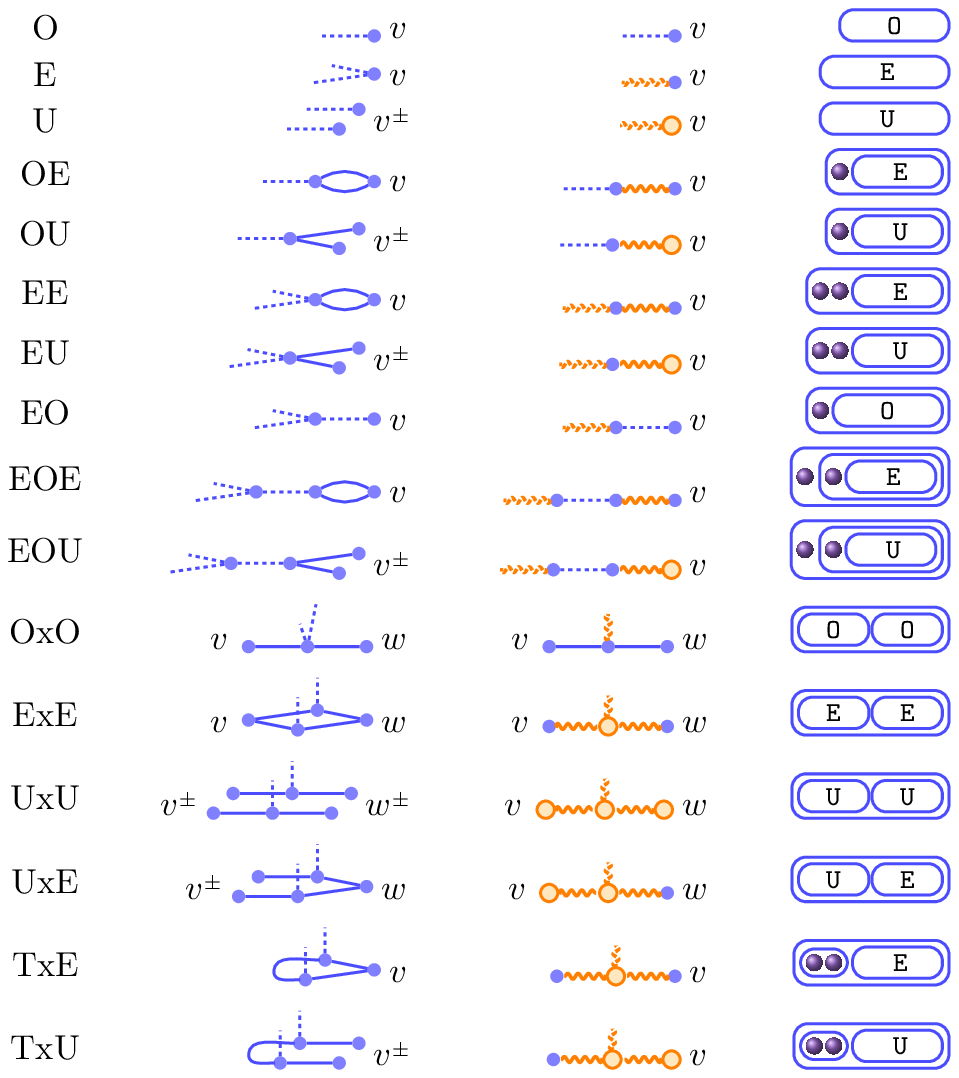}\\[5pt]
\end{center}
\begingroup\begin{flushleft}\footnotesize 
Notation (cf. Propositions \ref{Tclosure}, \ref{Sclosure}) 
Column 2/3: dashed edges, and the leftmost vertex in 
EO/EOE/EOU, are not in the core; vertices in the picture other than $v^\bullet,w^\bullet$ have 
genus 0 and no other outgoing edges. 
Column 2: $v^\bullet,w^\bullet$ have positive genus or $\ge 3$ outgoing non-dashed edges.
Column 3: $v,w$ have positive genus or an outgoing non-dashed yellow edge.
Column 4: O/E/U denotes an odd/even/\"ubereven maximal principal cluster. 
\\[8pt]\end{flushleft}
\endgroup
\caption{16 possible `neighbourhoods of $\infty$' for cluster pictures, open 
hyperelliptic graphs and open BY trees $(g\ge 2)$}
\label{opentable}
\label{nbhds of infty table}  
\end{table}

\endcomment

\def\v{\raise 1.3pt\hbox{$\>\scriptsize v\>\,$}}
\def\vpm{\raise 1.3pt\hbox{$\>\scriptsize v\rlap{$^\pm$}\>\>\>\,$}}
\def\w{\raise 1.3pt\hbox{$\>\scriptsize w\>\,$}}
\def\wpm{\raise 1.3pt\hbox{$\>\scriptsize w\rlap{$^\pm$}\>\>\>\,$}}

\stepcounter{equation}
\begin{table}[!htb]                                         
\caption{Neighbourhoods of $\infty$ (genus $\ge 2$)}
\begin{center}
\begin{tabular}{|c@{\quad}|r@{\quad}r@{\quad}r|}
\hline     
\scalebox{0.75}{Case} & \scalebox{0.75}{Connected component of $\infty$}   & \scalebox{0.75}{Connected component of $\infty$ } & \scalebox{0.75}{Configuration of maximal}\cr
     & \scalebox{0.75}{of $G\setminus\text{\{principal vertices\}}$} & \scalebox{0.75}{of $T\setminus\text{\{principal vertices\}}$} & \scalebox{0.75}{principal clusters of $\Sigma$}\\[3pt]
\hline     
\raise1pt\hbox{O} & \hgO \v & \byO \v & \clpO\vphantom{$\int^{\int^a}$} \cr
\raise1pt\hbox{E} & \hgE \v & \byE \v & \clpE \cr
\raise1pt\hbox{U} & \hgU \vpm & \byU \v & \clpU \cr
\raise3pt\hbox{OE} & \hgOE \v & \byOE \v & \clpOE \cr
\raise3pt\hbox{OU} & \hgOU \vpm & \byOU \v & \clpOU \cr
\raise3pt\hbox{EE} & \hgEE \v & \byEE \v & \clpEE \cr
\raise3pt\hbox{EU} & \hgEU \vpm & \byEU \v & \clpEU \cr
\raise3pt\hbox{EO} & \hgEO \v & \byEO \v & \clpEO \cr
\raise5pt\hbox{EOE} & \hgEOE \v & \byEOE \v & \clpEOE \cr
\raise5pt\hbox{EOU} & \hgEOU \vpm & \byEOU \v & \clpEOU \cr
\raise3pt\hbox{OxO} & \v \hgOxO \w & \v \byOxO \w & \clpOxO \cr
\raise3pt\hbox{ExE} & \v \hgExE \w & \v \byExE \w & \clpExE \cr
\raise3pt\hbox{UxU} & \vpm \hgUxU \wpm & \v \byUxU \w & \clpUxU \cr
\raise3pt\hbox{UxE} & \vpm \hgUxE \w & \v \byUxE \w & \clpUxE \cr
\raise3pt\hbox{TxE} & \hgTxE \v & \byTxE \v & \clpTxE \cr
\raise3pt\hbox{TxU} & \hgTxU \vpm & \byTxU \v & \clpTxU \cr
\hline     
\end{tabular}
\end{center}
\begingroup\begin{flushleft}\footnotesize 
In every row the three entries correspond to one another under $\GG, \TT, \SS$.\\
Column 2/3: labelled vertices are principal, others are not.\\
Column 2/3: dashed edges, and the leftmost vertex in 
EO/EOE/EOU, are not in the \rlap{core.}\\
Column 4: O/E/U denotes an odd/even/\"ubereven maximal principal cluster. 
\\[8pt]\end{flushleft}
\endgroup
\label{opentable}
\label{nbhds of infty table} 
\end{table}


%

\begin{proposition}
\label{Tclosure}
Let $T$ be an open BY tree and $G$ an open hyperelliptic graph.

\noindent
(1) The core $\tilde{T}$ of $T$ (resp. $\tilde{G}$ of $G$) exists and is unique. 

\noindent
(2) $g(\tilde T)=g(T)$ and $g(G)=g(\tilde G)$. 

\smallskip
\noindent
Now assume $g(T)\ge 2$ and $g(G)\ge 2$.
\smallskip

\noindent
(3) There are 16 possibilities for the connected component
of $\infty$ of $T\setminus\text{\{principal vertices\}}$, and 16 
corresponding ones for hyperelliptic graphs.
%
%
They are given in 
Table~\ref{opentable}. 

\noindent
(4) 
In the notation of Table~\ref{opentable},
$\tilde T$ is obtained from $T$ by 
starting at $\infty$ and removing 
\begin{itemize}
\item
(Cases O--EU) edge
\item
(Cases EO--EOU) edge, vertex, edge
\item
(Cases OxO--TxU) edge, vertex
\end{itemize}
and $\tilde G$ from $G$ by removing
\begin{itemize}
\item
(Cases O--U) edge(s) 
\item
(Case EO) edges, vertex, edge
\item
(Case EOE, EOU) edges, vertex, edge, vertex
\item
(Cases OE--EU, OxO--TxU) edge(s), vertex/vertices.
\end{itemize}
Here `edge/edges' means removing the unique edge / $\iota$-orbit of two edges from $\infty$ 
or the latest removed vertex; and `vertex/vertices' means taking the vertex / $\iota$-orbit 
of two vertices incident to the latest removed edge(s), and 
(a) when they have degree 1, removing them or (b) when they have degree~2, declaring
them to not be vertices anymore (but interior points on the resulting merged edges 
instead).
%
%
%
\end{proposition}
\begin{proof}
That the core of $T$ (resp. $G$) exists and is unique follows 
by inspection in genus 0 and 1 (see Tables \ref{g012table}, \ref{tabg1a})
and, otherwise, from the explicit construction of the core detailed below. 
For that, we just do the $T\to\tilde T$ case, the $G\to\tilde G$ case being its translation 
via the correspondence between open hyperelliptic graphs and open BY trees. 

Clearly, to get from $T$ to $\tilde T$, the unique open edge (say, from $v_0$)
needs to be removed. 
If $v_0$ becomes a valid BY tree vertex (see Definition \ref{defbytree}),
we are done. There are 7 such configurations depending on whether $v_0$ is a genus~0 
leaf or not, and on the edge colours (Cases O--EU).

Otherwise, after the open edge is removed, 
$v_0$ must violate \ref{defbytree} (1) or (2). 
If it violates (1), $v_0$ is yellow (of genus 0) with exactly two other edges.
There are 5 such configurations depending on the two adjacent vertices, and 
on the edge colours (Cases ExE--TxU). Declaring $v_0$ to not be a vertex gives a BY tree.

If $v_0$ violates (2), it must have become a blue genus 0 vertex with no yellow edges.
Then it has one or two blue edges (by \ref{defbytree} (3)), and the removed edge
was yellow (by (2)).
When there are two blue edges, this is case OxO. 
Declaring $v_0$ to not be a vertex gives a BY tree.
If there is one blue edge, we remove $v_0$ and this edge --- these are 
Cases EO, EOE, EOU depending on the vertex adjacent to $v_0$.

The statement about the genus is clear, as no positive genus vertices are removed and the 
(relative) homology is unchanged. (See also Proposition~\ref{le:H1iso}.)
%
\end{proof}

An immediate corollary is the following. 

\begin{corollary} \label{res of auts} 
Let $T$ be an open BY tree with core $\tilde{T}$. Then the map $\Aut T \rightarrow \Aut \tilde T$ given by restriction of automorphisms has kernel $C_2$ in Cases EE-OxO and is injective in all other cases. Its image consists of all those automorphisms of $\tilde T$ fixing the point closest to $\infty$.\\
\end{corollary}

\begin{remark}
The proof of Proposition \ref{Tclosure} 
gives a straightforward way to compute the core in practice.
\end{remark}

We now 
consider the opposite direction. 
That is, given a closed BY tree~$\tilde{T}$, which open BY trees have $\tilde{T}$ as their core:

%
%



%

%

\begin{corollary} \label{core to open}
Let $\tilde{T}$ be a closed BY tree. Then an open BY tree $T$ has core $\tilde{T}$ if and only if it is obtained from $\tilde{T}$ in one of the following ways:
\begin{itemize}
\item declaring a point on an edge of $\tilde{T}$ to be a vertex of genus 0 (and the same colour  as the edge) and adding a yellow open edge at this vertex,
\item adding a yellow open edge to a vertex of $\tilde{T}$,

\item adding a blue open edge to a blue vertex $v$ of $\tilde{T}$ which has $2g(v)+2>\# \text{blue edges at }v$,

\item  adding `closed  blue edge $\rightarrow$ genus $0$ blue vertex $\rightarrow$ open yellow edge' to a blue vertex $v$ of $\tilde{T}$ which has $2g(v)+2>\# \text{blue edges at }v$.
\end{itemize}
\end{corollary}

\begin{proof}
This is true by inspection in genus 0 and 1. Otherwise,
it follows from Proposition \ref{Tclosure} that the core of any open BY tree is obtained by removing either a single open edge, blue or yellow, or removing the configuration consisting of a closed blue edge followed by an open yellow edge. To prove the result, one just checks which conditions need to be satisfied at a point $x\in\tilde{T}$ in order for the graph given by glueing on one of these three configurations to be a valid open BY tree with core $\tilde{T}$.  
\end{proof}

\begin{remark}
\label{corropengraph}
Corollaries  \ref {res of auts} and \ref{core to open} have obvious analogues for hyperelliptic graphs via the correspondence. Since the statements are neater for BY trees and these are the ones we will use when comparing the notion of equivalence for  BY trees/hyperelliptic graphs to that for cluster pictures, we have omitted them.
\end{remark}

\subsubsection{Centres of BY trees}

Given a closed BY tree $\tilde{T}$, Corollary \ref{core to open} can be viewed as describing the equivalence class of open BY trees with core (isomorphic to) $\tilde{T}$. In this subsection we single out a canonical representative in each equivalence class of open BY trees. To do this, we first single out a canonical `centre' (either a vertex or edge) on a given closed BY tree. Glueing on an open yellow edge there gives the sought representative of the associated equivalence class of open BY trees. 

The following purely graph theoretic lemma shows the existence of a `centre' with respect to a weighting on the vertices of a tree. We omit the proof.

\begin{lemma} \label{centres of trees}
Let $T$ be a finite connected tree and $w:V(T)\rightarrow \mathbb{R}_{\geq 0}$ be a `weight' function on the vertices of $T$ such that each vertex of degree one or two has positive weight. For a subtree $T'\leq T$, set $w(T')=\sum_{v\in T'}w(v)$ and for each $v\in T$, define 
\[\phi(v)=\textup{max}\left\{w(T')\bigm|T'~~\text{is a connected component of }T\setminus \{v\}\right\}.\] 
Then either
\begin{itemize}
\item[(1)] $\textup{min}_{v\in T}\phi(v)<\frac{1}{2}w(T),$ in which case the minimum is attained at a unique vertex of $T$, and all other vertices have $\phi(v)>\frac{1}{2}\phi(T)$,
\end{itemize}
or
\begin{itemize}
\item[(2)] $\textup{min}_{v\in T}\phi(v)=\frac{1}{2}w(T),$ in which case the minimum is attained at precisely two vertices of $T$, and these vertices are adjacent. 
\end{itemize}
In case $(1)$ we call the minimising vertex the centre of $T$ with respect to the weighting $\phi$. In case $(2)$, we define the centre to be the edge joining the two minimising vertices.
\end{lemma}

\begin{definition} \label{BY tree centre}
Let $T$ be a closed BY tree. We define its \emph{centre} to be the vertex or edge afforded by Lemma \ref{centres of trees} applied to the weight function $w:V(T)\rightarrow \mathbb{Z}_{\geq 0}$ given by
\[w(v)=\begin{cases}0~~&~~v~\textup{yellow,}\\ 2(g)+2-\textup{deg}_{T_b}(v)~~&~~v~\textup{blue,}\end{cases}\] 
where here for a blue vertex $v$, $\textup{deg}_{T_b}(v)$ denotes the number of blue edges at $v$.
Note that as $w$ is invariant under all automorphisms of $T$, the centre of $T$ is also. 
\end{definition}

\begin{remark} \label{genus weighting computation}
Let $T$ be a closed BY tree. 
With $w$ as in Definition \ref{BY tree centre} above, the same argument as in the proof of Proposition \ref{TSgenus} gives 
$$
  w(T)=2g(T)+2.
$$
Similarly, for each $v\in T$ and each connected component $T'$ 
of $T\setminus\{v\}$ 
we have
\[w(T')=\begin{cases}2g(T')+2~~&~~\textup{ if the open edge of $T'$ is yellow,}\\2g(T')+1~~&~~\textup{ if the open edge of $T'$ is blue}.\end{cases}\]
\end{remark}

\begin{remark} \label{canonical BY tree rep}
Glueing an open yellow edge to the centre of a closed BY tree $\tilde{T}$ gives (up to isomorphism) a canonical representative in the equivalence class of open BY trees having $\tilde{T}$ as their core. Letting $T$ denote this representative, the natural map $\Aut T\rightarrow \Aut\tilde{T}$ given by restriction of automorphisms is surjective
(this follows from Corollary \ref{res of auts}  since all automorphisms of $\tilde{T}$ fix the centre).
The kernel of the restriction homomorphism is trivial if the centre of $\tilde{T}$ is yellow, whilst if the centre of $\tilde{T}$ is blue then the kernel is isomorphic to $C_2$, generated by the automorphism of $T$ which fixes all vertices, has sign $-1$ on (the component of $T_y$ containing) the yellow open edge, and trivial sign on all other components of $T_y$ (again see Corollary \ref{res of auts}).
\end{remark}

\begin{remark}
Since equivalence is preserved by the correspondence between open BY trees and open hyperelliptic graphs, the construction above gives a canonical representative in each equivalence class of open hyperelliptic graphs.
\end{remark}

\subsection{Equivalence: cluster pictures}


%
%
%
%


We now turn to cluster pictures. We begin by describing the clusters that will correspond to principal vertices and to genus 0 leaves on the associated open BY tree,


\subsubsection{Principal clusters, twins and cotwins}

\begin{definition}\label{def principal cotwin}
Let $(X,\Sigma)$ be a cluster picture of genus $g\ge 2$.
Recall that a cluster of size 2 is a \emph{twin}.

A cluster $\s$ is a \emph{cotwin} if it has a child of size $2g$ whose complement is not a twin.

A proper cluster $\s$ is \emph{principal} if it is neither a twin nor a cotwin, 
and 
if $|\s|=2g+2$ then $\s$ has at least 3 children.
\end{definition}

\begin{example}
\label{ex1:0s3}
Each of the 7 pictures from Example \ref{ex1:0s2}
has exactly one principal cluster 
(smallest one of size $\ge 4$), and either one twin or one cotwin. 
\end{example}

\begin{remark} 
When $g(\Sigma)\ge 2$,
either $\Sigma$ has a unique maximal principal cluster $\s$, of size $\ge 2g$, 
or $X=\s\coprod\s'$ is a union of two principal clusters. 
Marking a maximal principal cluster by 
{\rm O/E/U} according to whether it is odd/even non-\"ubereven/\"ubereven
gives 16 possible configurations, as listed in Table~\ref{opentable} (right column). 
Note that $\Sigma$ has a cotwin if and only if it is in  cases OE-EU or EOE-EOU.
\end{remark}

\begin{lemma} \label{principal cluster lemma}
Let $\Sigma$ be a cluster picture of genus $\ge 2$, and $T=\TT(\Sigma)$ the associated open BY tree, with core $\tilde{T}$. Then a cluster $\s\in \Sigma$ is principal if and only if the associated vertex  $v_\mathfrak{s}$ of $T$ (see Construction \ref{StoT}) is a principal vertex. Moreover, $\s$ corresponds to a genus $0$ leaf in $\tilde{T}$ if and only if it is either a twin or a cotwin.
\end{lemma}

\begin{proof}
This follows from Construction \ref{StoT} (describing the association 
$\s\mapsto v_\mathfrak{s}$) and Table \ref{opentable},
which shows what is removed to obtain the core.
Note that twins correspond to 
genus 0 leaves of $T$, each of which remains a genus 0 leaf in the core, whilst cotwins correspond 
to vertices of $T$ which are not genus 0 leaves but become so when passing to the core.
\end{proof}

\subsubsection{Comparison with equivalence for open BY trees} 

We now show that the maps $\TT$ and $\underline{\Sigma}$ between cluster pictures and open BY trees preserve equivalence. Since (up to isomorphism) these maps are inverse to each other, it suffices to show the result for $\TT$.

\begin{proposition}
\label{Sclosure}
Two cluster pictures $(X,\Sigma)$, $(X',\Sigma')$ are equivalent if and only if 
the corresponding BY trees $\TT(\Sigma)$, $\TT(\Sigma')$ are. The same holds in the metric case.
\end{proposition}

\begin{proof}
Again, this is true by inspection in genus 0 and 1 (see Tables \ref{g012table}, \ref{tabg1a}). 
Now suppose $(X,\Sigma)$ is a cluster picture of genus $\ge 2$, and $T=\TT(\Sigma)$.

In the notation of Table \ref{opentable},
cluster pictures without clusters of size $2g$ or $2g+1$ fall into cases 
E, U, OxO, ExE, UxU and UxE.
Possible moves (see Definitions \ref{Sequivdef}, \ref{de:distance}) between 
such cluster pictures are

\def\arrno#1{$\!\!\overarrow{\hbox{\tiny (#1)\ }}\!\!$}
\def\v{\raise -7pt\llap{$\>\scriptsize v$}}
\def\onetwoarrow{\mathrel{\mathop{\scalebox{2}[1]{$\rightleftarrows$}}\limits^{\mathrm{\tiny (i)}}_{\mathrm{\tiny (ii)}}}}
\def\smonetwoarrow{\mathrel{\mathop{\scalebox{1.5}[1]{$\rightleftarrows$}}\limits^{\mathrm{\tiny (i)}}_{\mathrm{\tiny (ii)}}}}
\def\smtwoonearrow{\mathrel{\mathop{\scalebox{1.5}[1]{$\rightleftarrows$}}\limits^{\mathrm{\tiny (ii)}}_{\mathrm{\tiny (i)}}}}
\def\smthreefourarrow{\mathrel{\mathop{\scalebox{1.5}[1]{$\rightleftarrows$}}\limits^{\mathrm{\tiny (iii)}}_{\mathrm{\tiny (iv)}}}}

\begin{center}
E/U $\>\>\onetwoarrow\>\>$ OxO/ExE/UxE/UxU.
\end{center}

\noindent
The corresponding BY tree in cases E and U has an open yellow edge attached to a principal
vertex $v$. A principal child $\s$ of $X$ corresponds to an adjacent principal vertex $v_\s$
and the moves (i) above are obtained by adding a `cocluster' to $\s$.
The effect on $T$ is to move the yellow edge to a point on the edge between $v$ and $v_\s$
(without changing the metric on the core in the metric case).
The moves (ii) above are the inverse of this.

The only moves to and from cluster pictures that have a cluster of size $2g$ are

\begin{center}
  \parbox[c]{3.6em}{\clpOE}   $\smthreefourarrow$
  \parbox[c]{5.0em}{\clpEOE}  $\smtwoonearrow$
  \parbox[c]{4.5em}{\clpEE}   $\smonetwoarrow$
  \parbox[c]{4.5em}{\clpTxE}  $\smtwoonearrow$
  \parbox[c]{4.5em}{\clpE}    
\end{center}

\vspace{-13pt}

\begin{center}
  \parbox[c]{3.6em}{\byOE\v}   \qquad
  \parbox[c]{5.0em}{\scalebox{0.8}{\byEOE\v}}  \qquad
  \parbox[c]{4.5em}{\byEE\v}   \qquad
  \parbox[c]{4.5em}{\byTxE\v}  \qquad
  \parbox[c]{4.5em}{\byE\v}    
\end{center}

\noindent
and

\begin{center}
  \parbox[c]{3.6em}{\clpOU}   $\smthreefourarrow$
  \parbox[c]{5.0em}{\clpEOU}  $\smtwoonearrow$
  \parbox[c]{4.5em}{\clpEU}   $\smonetwoarrow$
  \parbox[c]{4.5em}{\clpTxU}  $\smtwoonearrow$
  \parbox[c]{4.5em}{\clpU}    
\end{center}

\vspace{-13pt}

\begin{center}
  \parbox[c]{3.6em}{\byOU\v}   \qquad 
  \parbox[c]{5.0em}{\scalebox{0.85}{\byEOU\v}}  \qquad 
  \parbox[c]{4.5em}{\byEU\v}   \qquad 
  \parbox[c]{4.5em}{\byTxU\v}  \qquad 
  \parbox[c]{4.5em}{\byU\v}    
\end{center}

\noindent
depending on whether the cluster of size $2g$ is non-\"ubereven or \"ubereven. 
By construction of $T$, 
moving along the chains 
transforms the `tail at~$\infty$' without altering the core, as shown above.
%

Finally, cluster pictures that have a cluster of size $2g+1$ but not of size $2g$
are cases O and EO. The only moves between them are 

\begin{center}
  \parbox[c]{3.6em}{\clpO}   $\smthreefourarrow$
  \parbox[c]{5.0em}{$\>\>$\clpEO}  $\smtwoonearrow$
  \parbox[c]{4.5em}{\clpE}   
\end{center}

\vspace{-13pt}

\begin{center}
  \parbox[c]{3.6em}{\byO\v}   \qquad 
  \parbox[c]{5.0em}{$\>\>$\byEO\v}  \qquad 
  \parbox[c]{4.5em}{\byE\v}   
\end{center}

\noindent
As before, these transform the tail at $\infty$ without altering the core, as shown.

%



This covers all possible moves between cluster pictures. Thus, 
equivalent cluster pictures yield BY trees with isomorphic cores (in other 
words, equivalent). 

Conversely, if $T$ are $T'$ are open BY trees with the same core, the moves described above,
the fact that BY trees are connected and Corollary \ref{core to open}
show that the associated cluster pictures are equivalent.
\end{proof}

\begin{remark}
\label{sequni}
Incidentally, the proof of the proposition shows that every equivalence of cluster pictures
can be broken up into moves (i)-(iv) \emph{uniquely} (without going back). 
More precisely, fix an equivalence class of cluster pictures. Consider 
the graph whose vertices are cluster pictures in this class (up to isomorphism)
and edges are given by moves (i)-(iv). Then this graph is a tree.
See Table \ref{tabcorr} for an example; here directions of arrows indicate moves (ii) and (iv).
\end{remark}

\stepcounter{equation}
\begin{table}[h]          
\begin{center}
\includegraphics[scale=0.7]{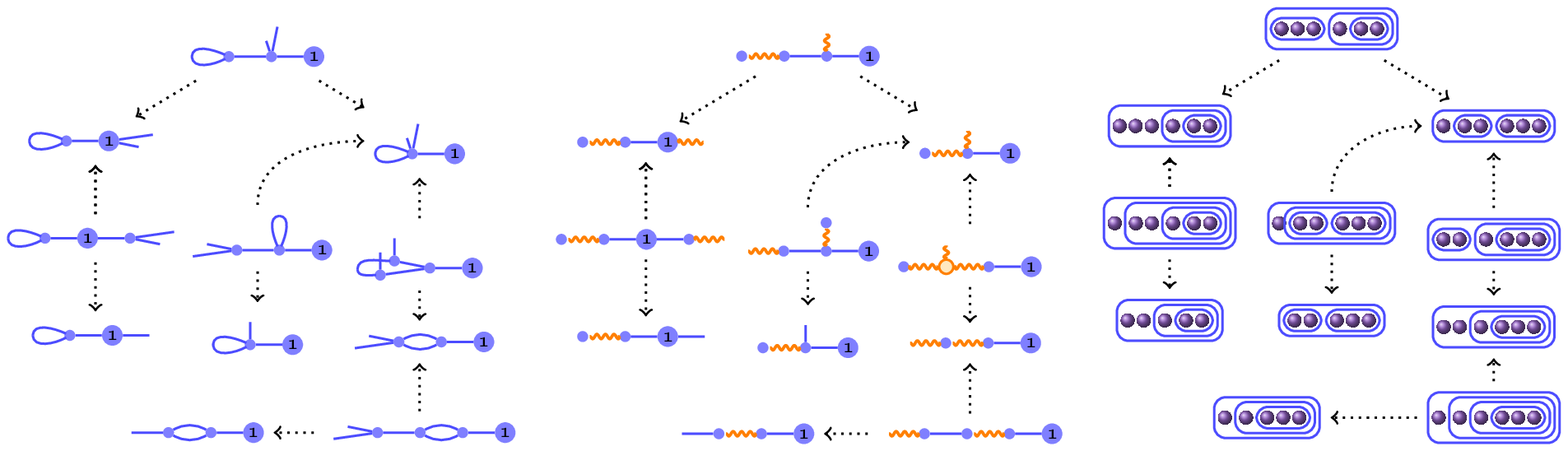}\\[12pt]
\end{center}
\caption{Example: an equivalence class of hyperelliptic graphs, open BY trees, 
and cluster pictures (Type $\hetype{1xI_n}$)}
\label{tabcorr}
\end{table}



\begin{proposition}
\label{adjacencyprop}
Say that two proper clusters $\s_1,\s_2$ in a cluster picture $(X,\Sigma)$ 
are adjacent if $\s_1<\s_2$, $\s_2<\s_1$, or $X=\s_1\coprod\s_2$ with $\s_1,\s_2<X$ and $X$ is even.
If $(X,\Sigma)$ and $(X',\Sigma')$ are equivalent cluster pictures of genus $\ge 2$,
then there is an adjacency preserving bijection
$$
  \text{\{principal $\s\in\Sigma$\}} \leftrightarrow \text{\{principal $\s'\in\Sigma'$\}},
$$
$$
  \text{\{twins and cotwins $\s\in\Sigma$\}} \leftrightarrow \text{\{twins and cotwins $\s'\in\Sigma'$\}}.
$$
In the metric case, we can also insist that the bijection preserves distances between clusters.
\end{proposition}

\begin{proof}
Let $T=\TT(\Sigma)$ and let $\tilde{T}$ be its core. 
By Lemma \ref{principal cluster lemma}, under the map $\s\mapsto v_\mathfrak{s}$, 
principal clusters of $\Sigma$ correspond to vertices of $\tilde{T}$ which are not genus $0$ 
leaves, whilst twins and cotwins of $\Sigma$ correspond to genus 0 leaves of $\tilde{T}$.  
First note 
that vertices $v_\mathfrak{s}$ and $v_\mathfrak{s'}$ are adjacent in the open BY tree $T$ 
if and only if $\s<\s'$ or $\s'<\s$. 
That two vertices in the core $\tilde{T}$ are adjacent if and only if the
corresponding clusters are now follows by consulting Table~\ref{opentable}.
In particular, any isomorphism
from $\tilde{T}$ to the core of $\TT(\Sigma')$ (one such 
necessarily exists by Proposition \ref{Sclosure}) induces a bijection as in the statement.  
\end{proof}

%

\subsubsection{Centres and balanced cluster pictures}

Recall that a cluster picture $(X,\Sigma)$ is balanced if $|X|=2g+2$ is even, 
there are either 0 or 2 clusters of size $g+1$, and $X$ is the only cluster of size $>g+1$.
(For instance, in Table \ref{tabg1a}, the second row in each of the three groups
is balanced.)


\begin{lemma} \label{existence of balanced clusters}
Every equivalence class of cluster pictures has (up to isomorphism) a unique balanced one. Under $\TT$, it corresponds to the canonical representative of the associated equivalence class of open BY trees as defined in Remark \ref{canonical BY tree rep}.
\end{lemma}

\begin{proof}
Let $\Sigma$ be a cluster picture of genus $g$, $T=\TT(\Sigma)$ the associated BY tree, and $\tilde{T}$ its core. Then by Proposition \ref{Sclosure}, the cluster pictures equivalent to $\Sigma$ are precisely those associated to the open BY trees obtained from $\tilde{T}$ by one of the operations of Corollary \ref{core to open}. Note that glueing an open blue edge to a vertex of $\tilde{T}$ results in a cluster picture of odd size (see Table \ref{opentable}) and such cluster pictures are not balanced. Similarly, glueing a closed blue edge (whose endpoint is blue of genus 0) followed by an open yellow edge onto a vertex of $\tilde{T}$ results in a cluster picture having a cotwin which again is not balanced. 

Next, fix a vertex $v$ of $\tilde{T}$ and consider the cluster picture $(X',\Sigma')$ associated to the open BY tree obtained by glueing a yellow open edge to $v$. We have $|X'|=2g+2$. Moreover, the children of $X'$ are all of the form $\mathfrak{s}_{v'}$ for $v'$ adjacent to $v$. For each such vertex, let $T_{v'}$ denote the connected component of $\tilde{T}\setminus{v}$ containing $v'$. By Remark \ref{genus of subtree}, it follows that the size of $\mathfrak{s}_{v'}$ is equal to $w(T_{v'})$ where $w$ is the weight function of Definition \ref{BY tree centre}. It now follows from Lemma \ref{centres of trees} that $(X',\Sigma')$ is balanced if and only if $v$ is the centre of $\tilde{T}$ (see also Remark \ref{genus weighting computation}).

Similarly, one sees that the cluster picture associated to the open BY tree obtained by glueing an open yellow edge to an existing edge of $\tilde{T}$ is balanced if and only if this edge is the centre of $\tilde{T}$. 
%
%
\end{proof}

\begin{corollary}\label{AutEquiv}
Let $(X,\Sigma)$ be a balanced cluster picture, let $T=\TT(\Sigma)$ be the associated open BY tree and let $\tilde{T}$ denote the core of $T$. Then the natural map $\Aut\Sigma\rightarrow \Aut\tilde{T}$, sending $\sigma \in \Aut\Sigma$ to the restriction of $\TT(\sigma)$ to $\tilde{T}$, is surjective. Its kernel is trivial if $X$ is \ub, and $C_2$  if $X$ is non-\ub\  (generated by the trivial permutation with $\epsilon(X)=-1$ and all other signs $+1$).
\end{corollary}

\begin{proof}
By Theorem \ref{combmain1}, the map $\Aut \Sigma \rightarrow \Aut T$ sending $\sigma$ to $\TT(\sigma)$ is an isomorphism. The result now follows from Lemma \ref{existence of balanced clusters} and the corresponding statement for BY trees discussed in Remark \ref{canonical BY tree rep}.
\end{proof}

\def\GraphScale{0.4}
\def\SnakeWiggle{3pt}
\def\clustersep{1.3pt}
\rootsize{0.5}

\section{The homology lattice $\Lambda$}\label{s:homology}



In this section we study the lattices $\Lambda$ attached to hyperelliptic graphs, BY trees and cluster pictures, along with their natural automorphism actions, and show that the correspondences identify them. 

\subsubsection*{Action of automorphisms}

Recall how to identify automorphism groups across the correspondence.
Suppose $G$ is a hyperelliptic graph, $T=\TT(G)$ the associated BY tree,
and $\Sigma=\SS(T)$ the associated cluster picture.
A choice of a section
$s:G/\left\langle \iota\right \rangle\rightarrow G$ as in Construction \ref{GtoT} 
gives an isomorphism $\Aut G\to\Aut T$, which makes $\Lambda_T$ an $\Aut G$-module.
There is also a canonical isomorphism $\Aut T\to\Aut \Sigma$, independent of any choices,
which makes $\Lambda_\Sigma$ into an $\Aut G$-module as well.
Finally, 
automorphisms of an open BY tree $T$ act on the core $\tilde T$, and 
similarly for hyperelliptic graphs.

\begin{theorem}[Lattice correspondence] \label{lattice thm 1}
If $\Sigma$ is a cluster picture, then there are canonical $\Aut\Sigma$-equivariant isomorphisms
$$
  \Lambda_{\widetilde{\GG(\TT(\Sigma))}}\iso \Lambda_{\GG(\TT(\Sigma))}\iso \Lambda_\Sigma\iso \Lambda_{\TT(\Sigma)}\iso 
  \Lambda_{\widetilde{\TT(\Sigma)}}.
$$
If $T$ is a BY tree, then there are canonical $\Aut T$-equivariant isomorphisms
$$
  \Lambda_{\GG(T)}\iso \Lambda_T\iso \Lambda_{\SS(T)}.
$$
If $G$ is an open hyperelliptic graph, choose a section $s: G/\langle \iota\rangle\to G$. 
Then there are canonical $\Aut G$-equivariant isomorphisms
$$
  \Lambda_G\iso \Lambda_{\TT(G)}\iso \Lambda_{\SS(G)}.
$$
For another section $s'$,
the two isomorphisms $\Lambda_G\iso \Lambda_{\TT(G)}$ differ by $\psi_{s,s'}$
of Proposition \ref{dependence of s}. The isomorphism $\Lambda_{\TT(G)}\iso \Lambda_{\SS(G)}$ 
does not depend on the choice of $s$.

In the metric case, all isomorphisms preserve the pairings. 
%
%
\end{theorem}

\begin{proof}
This follows upon combining Lemmas \ref{le:H1iso} with Propositions \ref{hyp graph homology} and \ref{lattices BY cluster} and Remark \ref{dependence of lattice iso on section}.
\end{proof}


%

\begin{corollary} \label{lattice equiv  cor}
Let $(X,\Sigma)$ and $(X',\Sigma')$ be equivalent cluster pictures. Then there is an isomorphism $\Lambda_\Sigma\iso \Lambda_{\Sigma'}$ which, in the metric case, preserves the respective pairings. 
\end{corollary}

\begin{proof}
In both the metric and non-metric cases, the core $\tilde{T}$ of $\underline{T}(\Sigma)$ is an invariant of its equivalence class. Hence so is the associated lattice $\Lambda_{\tilde{T}}$. The result now follows from Theorem \ref{lattice thm 1}.
\end{proof}

\begin{remark}
A proof of Corollary \ref{lattice equiv  cor} without passing through the correspondence can be given by using Proposition \ref{adjacencyprop}. 
\end{remark}

\subsection{Reduction to the closed case}

We begin by showing that the homology groups of open hyperelliptic graphs (resp. open BY trees) are isomorphic to those of their core. As a consequence we will only consider the closed case after this subsection.

\begin{proposition}\label{le:H1iso}
Let $G$ be an open hyperelliptic graph with core $\tilde{G}$. Then there is a canonical isomorphism \[H_1(G)\iso H_1(\tilde{G}),\]
equivariant for the action of $\Aut G$. Similarly, if $T$ is an open BY tree with  core $\tilde{T}$ then there is a canonical isomorphism
\[ H_1(T,T_b)\iso H_1(\tilde{T},\tilde{T}_b),\]
equivariant for the action of $\Aut T$.
(In the above, automorphisms of $G$ (resp. $T$) act on $H_1(\tilde{G})$ (resp. $H_1(\tilde{T},\tilde{T}_b)$) via their restriction to the core.)
\end{proposition}

\begin{proof}
It is easy to check the claim in genus 0 and 1 (see Tables \ref{g012table}, \ref{tabg1a}),
so assume $g\ge 2$.
From Table \ref{nbhds of infty table} we see that any open hyperelliptic graph admits a deformation retract onto its core. This induces the sought isomorphism on homology groups. 

In the case of BY trees, Table \ref{nbhds of infty table} shows that $T$ admits a deformation retract onto its core $\tilde{T}$ which induces a deformation retract from $T_b$ to $\tilde{T}_b$. This induces maps $H_i(T,T_b)\rightarrow H_i(\tilde{T},\tilde{T}_b)$ for each $i$. It also induces maps $H_i(T)\rightarrow H_i(\tilde{T})$ and $H_i(T_b)\rightarrow H_i(\tilde{T}_b)$ for each $i$ which, being induced by deformation retracts, are isomorphsims. That the maps on relative homology groups are also isomorphisms now follows from the relative homology exact sequence and the $5$-lemma.

The claim about the action of automorphisms is immediate since the deformation retracts 
act as identity on the core by definition.
\end{proof}

\begin{remark}
It follows from Lemma \ref{le:H1iso} that the action of an automorphism of an open hyperelliptic graph $G$ (resp. open BY tree $T$) on $H_1(G)$ (resp. $H_1(T,T_b)$)  depends only on its restriction to the core.
\end{remark}

\subsection{Hyperelliptic Graphs $\leftrightarrow$ BY trees}


\begin{proposition} \label{hyp graph homology} 
Let $G$ be a hyperelliptic graph, $T=\underline{T}(G)$ the associated  BY tree and $s:T\rightarrow G$ a section to the quotient map $\pi:G\rightarrow T$. Then there is a canonical isomorphism
\[H_1(G)\iso H_1(T,T_b),\] equivariant for the action of $\Aut G$ and, in the metric case, preserving the respective pairings. (In the above, $\Aut G$ acts on $H_1(T,T_b)$ via the isomorphism  $\Aut G\rightarrow \Aut T$ determined by $s$ (see Construction \ref{GtoT})).
\end{proposition}

\begin{proof}
We take the usual $\Delta$-complex structure on $T$, so that the $0$-simplices are the vertices and the $1$-simplices are the edges. For the $\Delta$-complex structure on $G$, we take the usual one, and then subdivide each $\iota$-anti-invariant edge at the preimage of the associated vertex of $T$ (which in each case is a genus 0 leaf). 
Define a map of complexes $C_{\bullet}(T)\rightarrow C_{\bullet}(G)$ given by $x\mapsto s(x)-\iota(s(x))$. Since the section $s$ is continuous, this map is compatible with the boundary operators on each side (strictly speaking, we need to choose an orientation on the edges of $T$ and $G$ respectively to define the boundary operators; we do this in such a way that 
both $\iota$ and $\pi$ are orientation-preserving). The kernel of this map of complexes is $C_\bullet(T_b)$ and, along with the quotient map $\pi:C_\bullet(G)\rightarrow C_\bullet (T)$ we obtain a short exact sequence of complexes
\[0\rightarrow C_\bullet(T)/C_\bullet(T_b)\longrightarrow C_\bullet(G)\longrightarrow C_\bullet(T)\rightarrow 0.\]
Since $H_2(T)=0=H_1(T)$ ($T$ is contractible)
and $H_\bullet(T,T_b)$ is the homology of the leftmost complex,
this sequence gives an isomorphism
\[H_1(T,T_b)\stackrel{\sim}{\longrightarrow}H_1(G).\]
In the metric case, the pairings on $H_1(G)$ and $H_1(T,T_b)$ are induced by ones on $C_1(G)$ and $C_1(T)$ respectively. The scaling factors in Construction \ref{GtoT} 
are defined in such a way that the map in degree 1 in the short exact sequence above preserves these. Similarly, the compatibility with automorphism actions 
can be checked on the level of the map $C_1(T)\to C_1(G)$ (again, see Construction \ref{GtoT}).
\end{proof}

\begin{remark} \label{dependence of lattice iso on section}
For each section $s: T\to G$, write $f_s:H_1(G)\rightarrow H_1(T,T_b)$ for the (inverse of the) isomorphism constructed in Proposition \ref{hyp graph homology}. 
Then given two sections $s$ and $s'$ one has
\[f_{s'}=\psi_{s,s'}\circ f_s\]
where $\psi_{s,s'}\in \Aut T$ is as in Definition \ref{defi of psi}.
\end{remark}

\begin{remark}
Another approach to proving the existence of the isomorphism of Proposition \ref{hyp graph homology} is as follows. Writing $G_b$ for the subgraph of $G$ fixed by the hyperelliptic involution, one has $H_1(G)\iso \tilde{H}_0(G_b)$. This follows by enlarging the closed sets $T_b\subset T, G_b\subset G$ to their small open neighbourhoods $\tilde T_b, \tilde G_b$ and applying the Mayer--Vietoris sequence to the open sets $U=\tilde G_b\cup s(T_y)$ and $V=\tilde G_b\cup \iota(s(T_y))$ which cover $G$ (note that $U\cap V$ is homotopic to $G_b$, whilst $U$ and $V$ individually are homotopic to the tree $T$). Since $G_b$ and $T_b$ are homeomorphic, we have $\tilde{H}_0(G_b)\iso \tilde{H}_0(T_b)$. The latter group is isomorphic to $H_1(T,T_b)$ via the relative homology sequence.
\end{remark}

\subsection{BY trees}

In this subsection we give an explicit description of the first relative homology group of a (closed, possibly metric) BY tree with respect to its blue part. This will be necessary for establising the second isomorphism of Theorem \ref{lattice thm 1} but may also be of independent interest.  It will be convenient to work with rooted BY trees, i.e. BY trees with a distinguished point (which may be a vertex but could also be a point on an edge). Our description of the relative homology group will naturally be compatible with automorphisms of the BY tree which fix the root (but not general automorphisms).

\subsubsection{Rooted BY trees}

\begin{definition}
By a \emph{rooted BY tree} we mean a pair $(T,R)$ where $T$ is a (closed, possibly metric) BY tree and $R$, the `root', is a point on $T$ (i.e. a vertex or a point on an edge). By an automorphism of a rooted BY tree we mean an automorphism of the underlying BY tree (complete with signs on yellow components) that preserves $R$. We write $\Aut_RT$ for the group of automorphisms of a rooted BY tree $(T,R)$. Given a vertex $v\neq R$ of $T$, we refer to the unique edge of $v$ in the direction of $R$ as the \textit{parent edge} of $v$.
\end{definition}

\begin{remark}  
Every BY tree has a centre (in the sense of Definition \ref{BY tree centre}),
which is fixed by all automorphisms.  Thus any BY tree can be made into a rooted BY tree in such a way that there is no difference between `rooted' and `non-rooted' automorphisms. 
\end{remark}


\begin{definition}
Let $T$ be an open BY tree and $\tilde{T}$ its core. Then we give $\tilde{T}$ the structure of a rooted BY tree by defining the root $R$ to be the point on $\tilde{T}$ which is closest to $\infty$ in $T$. 
\end{definition}

\begin{remark}
Let $\tilde{T}$ be the core of an open BY tree $T$ of genus $\ge 2$.  
Then whether or not the root $R$ is a vertex depends on the type of the neighbourhood of infinity in $T$ (cf. Table \ref{opentable}).
Specifically, it is a vertex of $\tilde{T}$ in cases 
O--EOU, and lies on an edge otherwise. 
In cases O--E and OE--EOU, $R$ is a blue vertex and in case $U$, $R$ is a yellow vertex. In case OxO, $R$ is a point on a blue edge and in cases ExE--TxU, $R$ is a point on a yellow edge. Note that the automorphism group of an open BY tree $\tilde{T}$ maps surjectively to the automorphism group of its core, viewed as a rooted BY tree (the map being restriction of automorphisms). This is an isomorphism apart from cases EE-OxO, where the kernel is isomorphic to $C_2$ and acts trivially on $H_1(T,T_b)$ (see Corollary \ref{res of auts}). 
\end{remark}

\subsubsection{Homology of a rooted BY tree}

Let $(T,R)$ be a 
rooted (closed) BY tree with blue part $T_b$ and yellow part $T_y$. For the purposes of computation, we make $T$ into a $\Delta$-complex in the usual way, save that, in the case that the root $R$ lies on an edge, we subdivide this edge at the root $R$ so that $R$ becomes a $0$-simplex. Moreover, we orient all edges so that they point towards $R$. Note that every automorphism of $T$ is orientation preserving since it fixes the root.

For the rest of the section we adopt the following convention.

\begin{convention} \label{root of subtree}
If $T'$ is a closed subtree of a rooted BY tree $(T,R)$ then we take as the root of $T'$ the point on $T'$ closest to $R$ in $T$ (which is either $R$ itself or a vertex of both $T$ and $T'$ (or both)). 
\end{convention}




\begin{definition} \label{BY tree lattice}
Let $(T,R)$ be a rooted BY tree. For a connected component $Y$ of $T_y$, let $\overline{Y}$ denote its closure in $T$, viewed as a rooted tree with root $R_Y$ as in Convention \ref{root of subtree}. Write $\cL_Y$ for the set of (non-root) leaves of $\overline{Y}$ and define $\cL_T=\cup_{Y}\cL_Y$ (note that this union is disjoint);
equivalently, 
$$
  \cL_T=\{\text{blue vertices}~v \neq R~\text{whose parent edge is yellow}\}.
$$
For $v\in \cL_T$, take $Y$ for which $v\in \cL_Y$ and define $\hat v=R_Y$.
We then define the free $\Z$-module $\Pi_T$ by
$$
  \Pi_T = 
  \begin{cases}
  \>\>\>
  \vphantom{\int^{\int^X}_{\int_X}}
  \Z[\cL_T] & \text{if $R$ is blue,}\cr
  \>\Bigl\{\sum_v \lambda_v v\in \Z[\cL_T]\bigm|\sum_{\hat v=R}\lambda_v =0\Bigr\} &
    \text{if $R$ is yellow.}\\[8pt]
  \end{cases}
$$
%
In the metric case, we define a pairing on $\Z[\cL_T]$ and $\Pi_T$ by setting
\[\left \langle v_1,v_2\right \rangle =\begin{cases} \delta(v_1\wedge v_2,\hat v_1)~~&~~\text{if }\hat v_1=\hat v_2,\\0~~&~~\text{otherwise,} \end{cases}\]
where $v_1\wedge v_2$ denotes the point (vertex or $R$ (or both)) at which the unique paths in $T$ from $v_1$ to $R$ and $v_2$ to $R$ meet.
Automorphisms $(\sigma,\epsilon_\sigma)\in\text{Aut}_RT$ act on $\Z[\cL_T]$ and $\Pi_T$ 
by setting, for a leaf $v_1\in \overline{Y}$,
\[(\sigma,\epsilon_\sigma)\cdot v_1=\epsilon_\sigma(Y)\sigma(v_1),\]
and extending linearly.
\end{definition}


\begin{remark}
For $v\in \cL_T$, $\hat v$ is blue unless $R$ is yellow and $\hat v=R$ 
(as all yellow vertices of $T$ have only yellow edges).
\end{remark}

\begin{remark} \label{extension of pairing}
Note that the action of $\text{Aut}_RT$ on $\Z[\cL_T]$
is particularly simple, being given by signed permutations.   
\end{remark}

\begin{proposition} \label{BY tree homology proposition}
Let $(T,R)$ be a rooted BY tree.
 Then there is a canonical isomorphism 
\[ \Pi_T\iso H_1(T,T_b),\]
equivariant for the action of $\Aut_RT$ and, in the metric case, 
respecting the pairings.
\end{proposition}

\begin{proof}
First note that the map sending a connected component 
of $T_b$ to its root gives a bijection between 
the set of connected components of $T_b$ not containing $R$ and the set $\cL_T$. 
It now follows that
$$
  \textup{rk}~\Pi_T=\textup{rk}~H_1(T,T_b),
$$
as Remark \ref{reduced homology remark} shows that the rank of $H_1(T,T_b)$ is one less that the number of connected components of $T_b$.

Now consider the homomorphism $p:\Z[\cL_T]\rightarrow C_1(T)/C_1(T_b)$ where for $v\in \cL$, we define $p(v)$ as the shortest path in $T$ from $v$ to $\hat v$. Since this path is yellow by construction, $p$ is injective. Note also that $C_1(T)/C_1(T_b)$ is a free $\Z$-module since $C_1(T)$ is the direct sum of $C_1(T_b)$ and $C_1(T_y)$. We claim that the image of $p$ is a direct summand of $C_1(T)/C_1(T_b)=C_1(T_y)$. Indeed, for each $v\in \cL_T$, its parent edge appears in $p(v)$ with  multiplicity one, and does not appear in  $p(v')$ for any $v\neq v'\in  \cL_T$. Thus the set $\{p(v)\>|\>v\in \cL_T\}$ may be completed to a basis for $C_1(T_y)$ by adding in all yellow edges of $T$ except the parent edges of vertices in $\cL_T$. 

Now denote by $\tilde{p}$ the restriction of $p$ to $\Pi_T$, which is injective since $p$ is. We claim that its image is contained in $H_1(T,T_b)$. Indeed, writing $d:C_1(T)\rightarrow C_0(T)$ for the boundary map as usual, for $v\in \cL_T$  we have $d(p(v))=\hat v- v$. Provided that $\hat v$ is blue (i.e. unless $R$ is yellow and $\hat v=R$), we see that $d(p(v))$ lies in $C_0(T_b)$ in which case $p(v)$ is an element of $H_1(T,T_b)$. In particular, the claim holds for $R$ blue. On the other hand, if $R$ is yellow and $v,v'\in \cL_T$ with $\hat v =\hat v'=R$, then $d(p(v-v'))=v'-v \in C_0(T_b)$. By the way we have defined $\Pi_T$ when $R$ is yellow, this proves the claim in this instance also.

Since $\Pi_T$ is a direct summand of $\Z[\cL_T]$ (it is the  kernel of the homomorphism into $\Z$ sending $v\in \cL_T$ with $\hat v$ yellow to $1$, and all other elements of $\cL_T$ to $0$, which shows that $\Z[\cL_T]/\Pi_T$ is torsion free), $p$ is injective, and $p(\Z[\cL_T])$ is a direct summand on $C_1(T_y)$, it follows that $\tilde{p}(\Pi_T)$ is a direct summand of $C_1(T_y)$ also. We are now in the following situation: we have inclusions of free, finite rank $\Z$-modules $\tilde{p}(\Pi_T)\subseteq H_1(T,T_b)\subseteq C_1(T_y)$, with $\tilde{p}(\Pi_T)$ a direct summand of $C_1(T_y)$ and $\textup{rk}~\tilde{p}(\Pi_T)=\textup{rk}~H_1(T,T_b)$.
It now follows formally that $\tilde{p}(\Pi_T)=H_1(T,T_b)$ whence $\tilde{p}$ is an isomorphism.
%
%
%
\end{proof}





\newpage

\subsection{Cluster Pictures $\leftrightarrow$ BY trees}






\begin{proposition} \label{lattices BY cluster}
Let $(X,\Sigma)$ be a cluster picture, $T=\TT(\Sigma)$ the associated open $BY$ tree and $\tilde{T}$ its core. Then there is a canonical isomorphism
\[\Lambda_\Sigma \iso \Pi_{\tilde{T}},\]
equivariant for the action of $\Aut\Sigma=\Aut T$ and, in the metric case, preserving the respective pairings. 
(Here, $\Pi_T$ is as defined in Definition \ref{BY tree lattice}.)
\end{proposition}

\begin{proof}
We first claim that, under the correspondence between cluster pictures and open BY trees, the set  $\cE_\Sigma$ of Definition \ref{the set cal E} (consisting of even, non- \"ubereven clusters $\s\neq X$) is identified with the set $\cL_{\tilde{T}}$ (see Definition \ref{BY tree lattice}). Indeed, under the correspondence, even non-\"ubereven clusters correspond to blue vertices of $T$ whose parent edge is yellow. Now if $\s$ is an even non-\"ubereven cluster with associated vertex $v_\s$, one easily checks from Table \ref{nbhds of infty table} that the parent edge is in the core if and only if $\s\neq X$. Thus, as desired, $\cE_\Sigma$ corresponds to the set of blue vertices in the core which are not the root, and whose parent edge is yellow. 

It now follows that the map $\s\leftrightarrow v_\s$ induces an isomorphism between $\Lambda_\Sigma$ and  $\Pi_{\tilde{T}}$. 
Moreover, given $\s\in \cE_\Sigma$, corresponding to a leaf $v_\s$ of a yellow component $Y$ of $\tilde{T}$, one sees that $\hat\s$ (cf. Definition \ref{the set cal E}) corresponds to the root $R_Y$ of $\bar{Y}$. Indeed, this is immediate if $\s$ is contained in some non-\"ubereven cluster, and the case where no such cluster exists follows upon consulting Table \ref{nbhds of infty table}. The claimed result now follows immediately from the definitions of $\Lambda_\Sigma$ and $\Pi_{\tilde{T}}$, complete with pairing and action of automorphisms (using the identification of automorphism groups as given in Proposition \ref{cluster picture by tree proposition}).
\end{proof}

\subsection{An example}

Consider Examples \ref{Grain}, \ref{Miller}, \ref{Windmill}, \ref{WM} and \ref{GM} with $T$ and $G$ open. We construct bases for $\Lambda_{\Sigma}, \Lambda_T, \Lambda_G$ that illustrate why the lattices are isomorphic.

\hskip 1cm
\GRAINthree

\noindent\hskip 1cm\raise 1em\hbox{\MILLER}
\WINDMILL

\noindent
Recall that by definition
$$
\Lambda_{\Sigma} = \Z\t_1 + \Z\t_2 + \Z\t_3 + \Z\t_4 + \Z\s_2,
$$ the basis vectors corresponding to the even, non-\ub\ clusters in $\Sigma$. 

We produce a basis for $\Lambda_T=H_1(T,T_b)$ by looking at one yellow component at a time.
For a given yellow component $Y$ (with closure $\bar{Y}$), let $z$ be the leaf of $\bar{Y}$ closest to $\infty$ and as basis vectors take the shortest paths from the other leaves to $z$. In our example, we have 4 yellow components and the basis vectors are $[u_1, x], [u_2, x], [u_3,w_2], [u_4, w_2], [w_2,x]$, which is precisely the basis in Example \ref{Miller}. In other words, the basis we've chosen is indexed by blue vertices with a yellow edge towards $\infty$, which exactly correspond to even, non \ub\ clusters of $\Sigma$ (see Example \ref{GM}). 

To see why the two pairings coincide, consider for example the two paths $[u_3,w_2], [u_4, w_2]$. Their intersection is of length 2 which agrees with the distance $\delta(w_3, w_2)$. Since $w_2$ corresponds to $\s_2=\hat{\t}_3 = \hat{\t}_4$ and $w_3$ corresponds to $\s_3 = \t_3 \wedge \t_4$, 
$$
2=\langle[u_3,w_2], [u_4, w_2]\rangle =  \delta(w_3, w_2) = \delta(\t_3 \wedge \t_4, \hat{\t}_3) = \langle\t_3, \t_4\rangle=2.
$$

Now consider the action of the automorphisms $\sigma_{\Sigma}=(\alpha_{\Sigma}, \epsilon_{\Sigma})$ and $\sigma_{T}=(\alpha_{T}, \epsilon_{T})$ on the lattices. Recall that $\sigma_{\Sigma} (\s) = \epsilon_{\Sigma}(\s)\alpha_{\Sigma}(\s)$ by definition, e.g. $\sigma_{\Sigma} (\t_2) = -\t_1$ and that $\sigma_T([u,v]) = \epsilon_T(z)[\alpha_T(u),\alpha_T(v)]$, where $[u,v]$ is a path within one yellow component and $z$ any point on that component, e.g. $\sigma_T([u_2,x]) = -[u_1,x]$.

A path in our basis of the form $[u,*]$ is sent to $[\alpha_T(u), *]$ which is another basis vector, e.g. $\alpha_{T}([u_2, x]) = [u_1,x]$. By construction, if $u$ corresponds to the cluster $\s$ then $\alpha_T(u)$ corresponds to $\alpha_{\Sigma}(\s)$ and $\epsilon_T(u) = \epsilon_{\Sigma}(\s)$, e.g. $\epsilon_T([u_2,x]) = -1 = \epsilon_{\Sigma}(\t_2)$. It follows that the action of $\sigma_T$ on $\Lambda_T$ is the same as the action of $\sigma_{\Sigma}$ on $\Lambda_{\Sigma}$. 

Now consider the hyperelliptic graph $G$ with the decomposition of its $\iota$-permuted part $G_y = G_y^+ \coprod G_y^-$, 
where $G_y^+$ consists of  $e_2^+, e_3^+, v_3^+, e_{\infty}^+$ and the top halves of $\l_1$, $\l_2$, $\l_3$ and $\l_4$ (call these $\l_1^+, \l_2^+, \l_3^+, \l_4^+$).
To construct a basis for $\Lambda_G$ in a systematic way 
\begin{itemize}
\item construct $G'$ from $G_y$ by removing its edges towards $\infty$ and taking the closure. In our example $G' = G \setminus\{e_1, v_1, e_{\infty}^+, e_{\infty}^-\}$,
\item the $\iota$-invariant points remaining are $v_x, v_2$ and the mid-points of $\l_1, \l_2, \l_3, \l_4$,
\item for $\iota$-invariant points with an edge in $G'$ towards $\infty$, create a loop by following $G_y^+$ towards $\infty$ to the next $\iota$-invariant point and back via $G_y^-$; for our example we obtain the loops $\l_1$ (oriented clockwise) and $\l_2$ (anti-clockwise), and the loops $e_2^+-e_2^-$, $\l_3^++e_3^+-e_3^--\l_3^-$, $\l_4^+ + e_3^+ - e_3^--\l_4^-$, where we have oriented each edge and half-edge towards $\infty$. It is exactly the basis given in Example \ref{Windmill}.
\end{itemize}
Under the $2:1$ map $G \to T$, these loops correspond to yellow paths from one blue vertex to another blue vertex which is closer to $\infty$. By construction, this gives the basis of $\Lambda_T$.  

Since both pairings measure the length of the intersection of loops/paths, we get the same pairing on both spaces, e.g. 
$$
2=\langle[u_3,w_2], [u_4, w_2]\rangle =  \delta(w_3, w_2) =2\delta(e_3^+)= \delta(e_3^+)+\delta(e_3^-) = 
$$
$$
 = \langle\l_3^++e_3^+-e_3^--\l_3^-, \l_4^+ + e_3^+ - e_3^--\l_4^-\rangle=2.
$$

As in Example \ref{WM}, the action of $\sigma_G$ on $G/\langle \iota \rangle = T$ is that of $\a_T$, in particular $\a_T([u_2,x]) =[u_1,x]$ corresponds to $\sigma_G(\l_2) = \pm\l_1$. Moreover $\epsilon_T([u_2,x])=-1$ corresponds to $\sigma_G(\l_2^+) = \l_1^-$ so that $\sigma_T([u_2,x]) = -[u_1,x]$ corresponds precisely to $\sigma_G(\l_2)= -\l_1$.

\def\GraphScale{0.4}
\def\SnakeWiggle{3pt}
\def\clustersep{1.3pt}
\rootsize{0.5}

\section{Tamagawa groups of hyperelliptic graphs}\label{s:components}

In this section we study the Tamagawa group $\Phi(G)$ (see Definition \ref{componentgroup}) of 
a hyperelliptic graph $G$ whose edge lengths are integers, and the corresponding group for 
BY trees and cluster pictures. In Proposition \ref{component group vs jacobian} 
we identify it with the graph-theoretic Jacobian of $G$ along with automorphism action.
We then give an explicit description of the 2-torsion in this group (Corollary \ref{2torscor}).

\subsection{Integral hyperelliptic graphs and Tamagawa groups}

\begin{definition}
A (closed) metric hyperelliptic graph $G$  is \emph{integral} 
if all edge lengths are integers, unless $G$ is the genus 1 circle graph from \eqref{circlegraph}.
In that exceptional case, we say that $G$ is integral if the sum of lengths of its two edges 
is an integer.

A closed metric BY tree $T$ is \emph{integral} if $\GG(T)$ is. 

A metric cluster picture $(X,\Sigma)$ is \textit{integral} if the core of $\GG(\TT(\Sigma))$ is.
\end{definition}

\begin{lemma} \label{integral BY trees lemma}
A (closed) metric BY tree is integral if and only if all edges have integral length and all edges not incident to a genus 0 leaf have even length. 

A metric cluster picture $(X,\Sigma)$ of genus $\ge 2$ is integral if and only if 
\begin{itemize}
\item $\delta(\s,\s')\in 2\Z$ for $\s'<\s$ with $\s'$, $\s$ principal, $\s'$ odd,
\item $\delta(\s,\s')\in 2\Z$ for $\s,\s'$ odd principal, $X=\s\bigsqcup \s'$,
\item $\delta(\s,\s')\in \Z$ for $\s'<\s$ with $\s'$, $\s$ principal, $\s'$ even,
\item $\delta(\s,\s')\in \Z$ for $\s,\s'$ even principal, $X=\s\bigsqcup \s'$,
\item $\delta(\s,\t)\in \frac{1}{2}\Z$ for a twin $\t<\s$,
\item $\delta(\s,\cc)\in \frac{1}{2}\Z$ for $\cc$ a cotwin, $\s<\cc$ of size $2g$.
\end{itemize}
\end{lemma}

\begin{proof}
This follows immediately from the (closed) case of the correspondence between cluster pictures, BY trees and hyperelliptic graphs as detailed in Section \ref{s:121c} (see in particular Table \ref{closed correspondence}).
\end{proof}

\begin{remark}
Given an integral hyperelliptic graph $G$, the paring on $H_1(G,\Z)$ takes integer values. In particular, $H_1(G,\Z)$ embeds in its abstract dual $H_1(G,\Z)^\vee$ via $x\mapsto \left \langle x,-\right \rangle$. By using the correspondences of previous sections, it follows that if $X$ is either an integral BY tree or an integral cluster picture, then $\Lambda_X$ embeds into its abstract dual $\Lambda_X^\vee$ similarly.
\end{remark}

\begin{definition}
\label{componentgroup}
Let $X$ be a hyperelliptic graph/BY tree/cluster picture and suppose that $X$ is integral. Then we define the \emph{Tamagawa group} $\Phi(X)$ as
\[\Phi(X)=\Lambda_X^\vee/\Lambda_X.\]
In each case, the action of $\Aut X$ on $\Lambda_X$ induces an action on $\Phi(X)$.
\end{definition}


\begin{theorem}[Tamagawa group correspondence] \label{component theorem}
Let $(X,\Sigma)$ be an integral cluster picture, $T$ (resp. $G$) the associated open BY tree (resp. open hyperelliptic graph) and $\tilde{T}$ (resp. $\tilde{G}$) its core. Then we have  isomorphisms 
\[\Phi(\Sigma)\iso \Phi(\tilde{T})\iso \Phi(\tilde{G}),\]
 equivariant for the action of $\Aut\Sigma$.
\end{theorem}

\begin{proof}
This follows immediately from Theorem \ref{lattice thm 1}.
\end{proof}

\subsection{Jacobians of graphs}
\label{ss:jacgraphs}

In this section, we show that the Tamagawa group of an integral hyperelliptic graph $G$ 
of coincides with the Jacobian of a (combinatorial) graph $G_{\Z}$ canonically associated to $G$.
Throughout this subsection, $G$ has genus $\ge 2$.

\begin{notation} \label{integer points}
For an integral hyperelliptic graph $G$, we denote by $G_\Z$ the graph having the same underlying topological space as $G$, but whose set of vertices consists of those points on $G$ which are an integer distance from the vertices of $G$. Equivalently, $G_\Z$ is the graph obtained by subdividing each edge $e$ of $G$, say of length $l$, by adding $l-1$ vertices at intervals of unit distance along the edge, so as to obtain a new graph all of whose edge lengths are $1$.  
\end{notation}

\begin{remark}
We  have $\Aut G_\Z=\Aut G$ and the discussion in Section \ref{subdivision} shows that $H_1(G_\Z)$ is canonically isomorphic to $H_1(G)$, with the isomorphism preserving the respective pairings and automorphism actions. 
\end{remark}

In what follows we shall think of $G_\Z$ as being a finite combinatorial graph with unweighted edges (though possibly with loops and multiple edges) and disregard the genus marking. We now recall the definition of the Jacobian of such a graph.

\begin{definition}
Let $\cG$ be a finite combinatorial graph, possibly with loops and multiple edges (we reserve the letter `$G$' for hyperelliptic graphs). Write $\Div(\cG)$ for the free $\Z$-module on the vertices $V(\cG)$ of $\cG$ and $\Div^0(\cG)$ for the subgroup of $\Div(\cG)$ consisting of elements whose coefficients sum to zero. Contained in $\Div^0(\cG)$ is a certain full rank submodule $\Prin(\cG)$ consisting of `principal divisors' which may be defined as follows:

 For $v,v'\in V(\cG)$, set
\[v\cdot v'=\begin{cases}\deg(v)-2~\#~\text{loops at }v~~&~~v=v',\\- \#~ \text{edges between }v~\text{and}~v'~~&~~v\neq v',\end{cases}\]
and define a map $\alpha:\Div(\cG)\rightarrow \Div(\cG)$ by, for $v\in V(\cG)$, setting
\[\alpha(v)=\sum_{v'\in V(\cG)}(v\cdot v')v',\] 
and extending linearly. We then have $\Prin(\cG)=\im(\alpha)$. 

The \textit{Jacobian} of $\cG$ is then defined as
\[\Jac(\cG)=\Div^0(\cG)/\Prin(\cG).\]
It is a finite abelian group and the action of $\Aut\cG$ on $\Div(\cG)$ induces an action on $\Jac(\cG)$.
\end{definition}

\begin{remark} \label{loop edges}
The notion of the Jacobian of a graph appears in multiple places in the literature and is referred to by several different names, the most notable other ones being the sandpile group and the Picard group (see \cite[Section 1.1]{W} and the references therein for an overview of its occurence). Various equivalent definitions of the Jacobian also appear in the literature. The definition above is a slight variant of the one given in \cite{BN2}. There the Jacobian is only defined for graphs without loops (but possibly with multiple edges). Our definition of  $v\cdot v'$ above ensures that our definition of $\Jac(\cG)$ (along with automorphism action) agrees with that of the Jacobian of the graph obtained by removing all loop-edges from $\cG$. 

We also remark that in \cite[Section 3.1]{BR} a generalisation of the Jacobian is defined for arbitrary metric graphs. In the case that $G$ is an integral hyperelliptic graph the group $\Jac_\Z(G)$ in the notation of  loc. cit. agrees with $\textup{Jac}(G_\Z)$ as defined above. However, since the definition of the Jacobian of a metric graph is less elementary than that of a finite combinatorial graph, we have elected to work with $G_\Z$ rather than introduce $\Jac_\Z(G)$.
\end{remark}

\begin{proposition} \label{component group vs jacobian}
Let $G$ be an integral hyperelliptic graph of genus $\ge 2$. 
Then there is a canonical isomorphism 
\[\Jac(G_\Z)\iso \Phi(G),\]
equivariant for the action of $\Aut G_\Z=\Aut G$.
\end{proposition}

\begin{remark}
Several versions of this proposition, in various levels of generality (in particular, it is not specific to hyperelliptic graphs), appear in the literature though to the best of our knowledge the action of automorphism groups is not considered. We begin by reducing to the situation covered by \cite[Theorem B.4]{BN2} and deduce the compatibility of automorphisms from the explicit map defined there.
\end{remark}

\begin{proof}[Proof of Proposition \ref{component group vs jacobian}.]
When defining $\Phi(G)$ as $H_1(G)^\vee/H_1(G)$ we are at liberty to choose the $\Delta$-complex structure on $G$ and we do so by taking the $0$-simplices to consist of the vertices of $G$ and the $1$-simplices as the edges of $G_\Z$ (along with their endpoints). Note that if $e$ is a loop-edge of $G_\Z$ then it generates an orthogonal direct summand of $H_1(G)$ and, having length $1$, we have $\left \langle e,e\right \rangle=1$. In particular we see that $e$ does not contribute to the quotient $H_1(G)^\vee/H_1(G)$. Combining this observation with Remark \ref{loop edges}, it suffices to prove the result under the assumption that $G_\Z$ contains no loops.

We are now in the situation covered by \cite[Theorem B.4]{BN2} and our choice of $\Delta$-complex structure on $G$ ensures that $H_1(G)$, as computed with this choice, coincides with their $\Lambda^1(G_\Z)$. Following loc. cit., we now fix a base vertex $v\in V(G_\Z)$ and define a map $$
  f_v:\Div(G_\Z)\rightarrow H_1(G)^\vee/H_1(G)=\Phi(G)
$$ 
as follows. Given $v'\in V(G_\Z)$, pick a path $p_{v,v'}$ in $G_\Z$ from $v$ to $v'$ and set $f_v(v')=\left \langle p_{v,v'},-\right \rangle$. Since $p_{v,v'}$ has integral length, its pairing will all elements of $H_1(G)$ is integral and so it defines a valid element of $H_1(G)^\vee$. Moreover, given two different choices of path from $v$ to $v'$, their difference is an element of $H_1(G)$ so $f_v$ is independent of the choice of path $p_{v,v'}$. Restricting $f_v$ to $\Div^0(G_\Z)$ we obtain a map $f:\Div^0(G_\Z)\rightarrow \Phi(G)$ which does not depend on the choice of base vertex $v$. Then as asserted in loc. cit. (see \cite[Proposition 7.2]{BLHN} for the proof), the map $f$ induces the sought isomorphism $\Jac(G_\Z)\iso \Phi(G)$. 

With the explicit map in hand, it is easy to check compatibility with automorphisms. Let $\sigma\in \Aut G$ and view it as an automorphism of $G_\Z$. Let $v,v'\in V(G_\Z)$. Then $f$ sends $v-v'\in \Div^0(G_\Z)$ to $\left \langle p_{v,v'},-\right \rangle$ where $p_{v,v'}$ is any path from $v$ to $v'$. Now $\sigma(p_{v,v'})$ is a path from $\sigma(v)$ to $\sigma(v')$ and $\sigma(v)-\sigma(v')\in \Div^0(G_\Z)$ is mapped by $f$ to $\left \langle \sigma(p_{v,v'}),-\right \rangle$, which is the same as  we obtain by acting by $\sigma$ on $f(v-v')$. The result now follows since $\Div^0(G_\Z)$ is generated by the elements $v-v'$ as $v$ and $v'$ range over the vertices of $G_\Z$. 
\end{proof}

\subsection{2-torsion in the Tamagawa group} 
As an application of the correspondence between hyperelliptic graphs and BY trees, and the description of the group $H_1(T,T_b)$ for a BY tree $T$ afforded by Proposition \ref{BY tree homology proposition}, we end this section by computing the 2-torsion in the Tamagawa group of a hyperelliptic graph. 

The result for BY trees is the following. 

\begin{theorem} \label{2-tors in BY tree} 
Let $T$ be an integral BY tree of genus $\ge 2$. 
Write $S$ for the set of connected components of $T_b$, excluding the genus 0 leaves of $T$ whose unique (necessarily yellow) edge has odd length; $\Aut T$ acts naturally on $S$. Then, as an $\Aut T$-module,
$$
  \Phi(T)[2]\iso \begin{cases}0~~&~~S=\emptyset~~\text{and}~~\rk H_1(T,T_b)~~\text{even,}\\\Z/2\Z~~&~~S=\emptyset~~\text{and}~~\rk H_1(T,T_b)~~\text{odd,} \\ \ker\left((\Z/2\Z)[S]\stackrel{\textup{sum}}{\longrightarrow}\Z/2\Z\right)~~&~~\text{else,}\end{cases}
$$
where
%
`\textup{sum}' denotes the sum of the coefficients map.
\end{theorem}

\begin{proof} 
For the time being we will ignore the action of $\Aut T$, adding it back in at the end. 
Let $R\in T$ be a vertex and make $T$ into a rooted BY tree by taking $R$ to be the root. Now let $\Pi_T$ be as in Definition \ref{BY tree lattice}, so that by Proposition \ref{BY tree homology proposition}
we have an isomorphism of $\Z$-lattices $\Pi_T\iso H_1(T,T_b)$ and, in particular, we have \[\Phi(T)\iso \Pi_T^\vee/\Pi_T.\]
Noting that $\Pi_T$ is torsion free as a $\Z$-module and applying the snake  lemma to the commutative diagram with exact rows
\[
\xymatrix{0\ar[r] & \Pi_T\ar[r]\ar[d]^{2} & \Pi_T^\vee\ar[r]\ar[d]^{2} & \Phi(T)\ar[r]\ar[d]^{2} & 0\\
0\ar[r] & \Pi_T\ar[r] & \Pi_T^\vee\ar[r] & \Phi(T)\ar[r] & 0,
}
\] 
it follows that we have
\[\Phi(T)[2]\iso \ker\left(\Pi_T/2\Pi_T\longrightarrow \Pi_T^\vee/2\Pi_T^\vee\right).\]

Suppose first that $S\neq\emptyset$, so that $T$ has either a blue vertex which is not a genus 0 leaf, or that $T$ has a genus 0 leaf whose unique edge has even length, and take $R$ to be one such. Then since $R$ is blue, we have $\Pi_T=\Z[\cL_T]$ where the set $\cL_T$ of Definition \ref{BY tree lattice} (which depends on $R$) consists of the blue vertices different from $R$ whose parent edge is yellow. The pairing on $\Pi_T$ is given by 
\[\left \langle v_1,v_2\right \rangle=\begin{cases}0~~&~~\widehat{v_1}\neq \widehat{v_2},\\\delta(v_1\wedge v_2,\widehat{v_1})~~&~~\text{else,} \end{cases} \]
for $v_1,v_2\in \cL_T$ (see Definition \ref{BY tree lattice} for the definitions of $\hat{v}$ and $v_1\wedge v_2$).

We claim that for any $v_1,v_2\in \cL_T$ we have
\[\left \langle v_1,v_2\right \rangle \equiv \begin{cases}0\text{ (mod 2)}~~&~~v_1\neq v_2,\\ l_p(v)\text{ (mod 2)}~~&~~v_1=v=v_2,\end{cases}\]
where for $v\in T$ (not equal to $R$) $l_p(v)$ is the length of its parent edge. To prove the claim, first take $v_1\neq v_2\in \cL_T$ and assume that $\widehat{v_1}=\widehat{v_2}$ (otherwise $v_1$ pairs trivially with $v_2$ by definition and we are done). Then $v_1\wedge v_2$ cannot be a leaf and so each (necessarily yellow) edge on the shortest path from $v_1\wedge v_2$ to $\widehat{v_1}$ has even length. It now follows from Lemma \ref{integral BY trees lemma} that $\delta(v_1\wedge v_2,\widehat{v_1})$ is an even integer as desired. The case where $v_1=v_2$ is similar: every edge in the path from $v_1$ to $\widehat{v_1}$ is a yellow edge not incident to a genus 0 leaf, save possibly for the parent edge of $v_1$. 

Now for $v\in \cL_T$, let $\phi_v$ denote the homomorphism in $\Pi_T^\vee$ dual to $v$ (i.e. sending $v$ to $1$ and all other  elements of $\cL_T$ to $0$). Then by the claim, we see that the map $\Pi_T/2\Pi_T\rightarrow \Pi_T^\vee/2\Pi_T^\vee$ is given by $v\mapsto l_p(v) \phi_v$. Since the set $\{\phi_v|v\in \cL_T\}$ is a basis for $\Pi_T^\vee$ (the dual basis to the standard basis for $\Pi_T=\Z[\cL_T]$), the kernel of this map is the $\F_2$-vector space having as basis the elements $v\in \cL_T$ for which $l_p(v)$ is even. Now by Lemma \ref{integral BY trees lemma}, $v\in\cL_T$ can only have $l_p(v)$ odd if it is a genus 0 leaf. In particular, writing $O$ for the set of genus 0 leaves in $T$ whose unique edge has odd length, an $\F_2$-basis for $\ker(\Pi_T/2\Pi_T\rightarrow \Pi_T^\vee/2\Pi_T^\vee)$ is given by the set $\cL_T\setminus O$. The map sending $v\in \cL_T$ to its connected component in $T_b$ is a bijection onto the set of connected components of $T_b$ not containing $R$. It follows that the isomorphism of the statement in the case $S\neq\emptyset$ holds abstractly. 

To additionally obtain the isomorphism as $\Aut T$-modules, recall from the proof of Proposition  \ref{BY tree homology proposition} that the canonical isomorphism of $\Z$-lattices $\Pi_T\iso H_1(T,T_b)$ is given by sending $v\in \cL_T$ to the unique shortest path $p(v)$ between $v$ and $\hat{v}$. It follows from the argument above that, as $\Aut T$-modules, $\Phi(T)[2]$ is isomorphic to the subgroup of $H_1(T,T_b)/2H_1(T,T_b)$ generated by the set $\{p(v)\>|\>v\in \cL_T\setminus O\}$. One checks that the isomorphism $H_1(T,T_b)\rightarrow \tilde{H}_0(T_b)$ coming from the relative homology sequence (Remark \ref{reduced homology remark}) identifies this subgroup with $\ker\left((\Z/2\Z)[S]\stackrel{\textup{sum}}{\longrightarrow}\Z/2\Z\right)$. Since the map  $H_1(T,T_b)\rightarrow \tilde{H}_0(T_b)$ is $\Aut T$-equivariant upon passing to quotients by multiplication by $2$ on each side (since then we no longer need to consider orientation or signs) we are done. 

Suppose now that $S=\emptyset$, so that all blue vertices of $T$ are genus 0 leaves and each of their edges has odd length. Then $T$ necessarily has a yellow vertex and now we take the root $R$ to be one such. Note that now $\cL_T$ is precisely the set of genus 0 leaves of $T$. Now $T$ necessarily has precisely one yellow component, whence $\Pi_T$ sits in a short exact sequence
\[0\rightarrow \Pi_T \longrightarrow \Z[\cL_T]\stackrel{\text{sum}}{\longrightarrow} \Z\rightarrow 0,\]
the map `sum' sending $\sum_{v\in \cL_T}\lambda_v v$ to the sum of the $\lambda_v$. Since $\Pi_T$ is a free $\Z$-module, the sequence remains exact after applying the functor $\textup{Hom}(-,\Z)$ (which we denote $(-)^\vee$ for simplicity) and we obtain a commutative diagram with exact rows \[
\xymatrix{0\ar[r] & \Pi_T \ar[r]\ar[d]  & \Z[\cL_T]  \ar[r]\ar[d] & \Z\ar[r] & 0\\
0 & \Pi_T^\vee \ar[l] & (\Z[\cL_T])^\vee  \ar[l] & \Z ^\vee\ar[l] & 0\ar[l]
}
 \] 
 where here the two  vertical maps are induced by the pairing (see Definition \ref{BY tree lattice}). Since each object in the diagram is torsion free, tensoring by $\Z/2\Z$ we obtain a commutative diagram with exact rows
\[
\xymatrix{0\ar[r] & \Pi_T/2\Pi_T\ar[r]\ar[d]  & \Z[\cL_T]/2\Z[\cL_T] \ar[r]\ar[d] & \Z/2\Z\ar[r] & 0\\
0 & \Pi_T^\vee/2\Pi_T^\vee\ar[l] & (\Z[\cL_T])^\vee/2(\Z[\cL_T])^\vee \ar[l] & \Z^\vee/2\Z^\vee\ar[l] & 0.\ar[l]
}
 \]
The same argument as in the case $S\neq\emptyset$ shows that the rightmost of the two vertical maps sends $v\in \cL_T$ to its dual vector $\phi_v$ (each $l_p(v)$ being odd) and as such is injective.  Moreover, the map $\Z^\vee/2\Z^\vee\rightarrow (\Z[\cL_T])^\vee/2(\Z[\cL_T])^\vee$ sends the unique non-trivial element of $\Z^\vee/2\Z^\vee$ to $\sum_{v\in \cL_T}\phi_v$. Combining exactness in the middle of the bottom row with the injectivity of the rightmost vertical map shows that $\sum_{v\in \cL_T}v$ is the unique non-trivial element of the kernel of the map $\Z[\cL_T]/2\Z[\cL_T]\rightarrow \Pi_T^\vee/2\Pi_T^\vee$ given by composing the rightmost vertical map with the restriction map $(\Z[\cL_T])^\vee/2(\Z[\cL_T])^\vee\rightarrow \Pi_T^\vee/2\Pi_T^\vee$. Further, the top sequence shows that $\sum_{v\in \cL_T}v$ lies in $\Pi_T/2\Pi_T$ if and only if $|\cL_T|=\rk H_1(T,T_b)+1$ is even. Thus 
 \[\ker\left(\Pi_T/2\Pi_T\rightarrow \Pi_T^\vee/2\Pi_T^\vee\right)\iso \begin{cases} 0~~&~~\rk H_1(T,T_b)~~\text{ even,}\\\Z/2\Z~~&~~\rk H_1(T,T_b)\text{ odd,}\end{cases}\] 
 which  completes the proof of the theorem (note that we do not need to consider the action of $\Aut T$ in this case since the only possible action of any group on $\Z/2\Z$ is trivial).
\end{proof}

\begin{corollary}
\label{2torscor}
Let $G$ be a hyperelliptic graph of genus $\ge 2$. Write $G_b$ for the subgraph of $G$ fixed by the hyperelliptic involution and write $G(\Z)$ for the set of points on $G$ which are an integer distance from a vertex. 
Let $\cW$ denote the set of connected components of $G_b$ which contain a point of $G(\Z)$.
Then we have isomorphisms of $\Aut G$-modules
\[\Phi(G)[2]\iso \begin{cases}0~~&~~\cW=\emptyset~~\text{and}~~\rk H_1(G)~~\text{even,}\\\Z/2\Z~~&~~\cW=\emptyset~~\text{and}~~\rk H_1(G)~~\text{odd,} \\ \ker\left((\Z/2\Z)[\cW]\stackrel{\textup{sum}}{\longrightarrow}\Z/2\Z\right)~~&~~\text{else,}\end{cases}\]
where `\textup{sum}' denotes the sum of the coefficients map.
\end{corollary}

\begin{proof}
Let $T=\TT(G)$ be the BY tree associated to $G$. Then the quotient map gives a homeomorphism from $G_b$ to $T_b$. Let $Z$ be a connected component of $T_b$. Then as yellow vertices of $T$ have only yellow edges, $Z$ necessarily contains a vertex of $T$. In fact, either $Z$ contains a vertex which is not a genus 0 leaf, or $Z=\{v\}$ for a single genus 0 leaf $v$. In the first case, the preimage under $\pi$ of this vertex is a vertex of $\pi^{-1}(Z)$. On the other hand, if $Z=\{v\}$ for a genus 0 leaf $v$, then $\pi^{-1}(v)$ is not a vertex of $G$ but the midpoint of an $\iota$-anti-invariant edge. 
In particular, $\pi^{-1}(v)\in G(\Z)$ if and only if the parent edge of $v$ has even length. 
Thus, the set $\cW$ corresponds under $\pi$ to the set of connected components of $T_b$ excluding the genus 0 leaves whose parent edge has odd length. Since the number of components of $T_b$ is equal to $\rk H_1(T,T_b)-1$ (see Remark \ref{reduced homology remark}), the result now follows from Theorem \ref{2-tors in BY tree}, along with usual identification of $\Aut G$ with $\Aut T$ (choices of section here are irrelevant since all signs act trivially on the objects involved).
\end{proof}

\def\GraphScale{0.4}
\def\SnakeWiggle{3pt}
\def\clustersep{2pt}
\rootsize{0.5}

\newpage

\section{Classification of semistable types and naming convention}\label{s:naming}

\subsection{Types of BY trees (and hyperelliptic graphs/cluster pictures)}

We propose a naming scheme for cluster pictures, (open) BY trees and (open) 
hyperelliptic graphs. We define it for BY trees and transport to the other two
categories via the one-to-one correspondence.

\begin{notation}
\label{namingnot}
Let $T$ be an \emph{open BY tree}. For the edges, we use
\smallskip

\noindent
\begin{tabular}{llllll}
\quad 
& $\cdot$           & blue edge\cr
& $:$               & yellow edge \cr
& $\cdot_{\scriptscriptstyle d/2}$, $:_{\scriptscriptstyle d}$  & edge of length $d$\cr
\noalign{\noindent and for the vertices\vphantom{$\int^X_X$}}
%
%
%
& U                    & yellow vertex\cr
& I or 0               & blue vertex of genus 0\cr
& 1,2,3,...            & blue vertex of genus 1,2,3,...\cr
\end{tabular}

\noindent
To define the notation (`Type') of $T$ itself, 
%
%
%
%
let $e$ be its open edge, say incident to a vertex $v$.
As a topological space, $T$
decomposes as a disjoint union
$$
  T=\{v\}\cup\{e\}\cup t_1\cup\ldots\cup t_k\cup T_1\cup...\cup T_n.
$$
where the $t_i$ are open trees `\BYtwin' (a blue vertex of genus 0 with one open 
yellow edge, say of length $n_i$), 
and the $T_i$ are the other connected components of $T \setminus \{e,v\}$; 
they are open BY trees. 
Then we define, inductively,
\begin{center}
  $\Type(T) = [e][v]_{n_1,...,n_k}\Type(T_1)\cdots\Type(T_n)$,
\end{center}
where $[e]$ is the notation for the edge $e$, as above, $[v]$ is the notation for the 
vertex $v$, as above.
To avoid ambiguity, when $n>0$, unless $T$ is the full tree that we are interested in,
we bracket everything after $[e]$ and write
\begin{center}
  $\Type(T) = [e]\bigl([v]_{n_1,...,n_k}\Type(T_1)\cdots\Type(T_n)\bigr)$.
\end{center}
In the non-metric case, the subscripts $n_i$ are placeholders instead of lengths whose 
purpose is only to record the number of genus 0 leaves%
\footnote{this agrees with Kodaira and Namikawa-Ueno types 
for semistable curves of genus 1 and 2}.
See Example \ref{exnot1} below.
\end{notation}

\begin{notation}
\label{typecl}
For a \emph{closed BY tree} $\tilde T$, recall from Remark \ref{canonical BY tree rep} 
that there is a canonical open BY tree $T$ with core $\tilde T$ obtained 
by glueing a yellow open edge $e_0$ to the centre of $\tilde T$.
%
%
%
We let $\Type(\tilde T)$ to be the name of $T$ with [$e_0$] omitted.
To emphasize the configurations for which the centre is an edge, 
we also use an alternative notation
$$
\begin{array}{llllll}
  0{\tt .}_{\scriptscriptstyle m}\Type(T_1){\tt .}_{\scriptscriptstyle n}\Type(T_2) & \longmapsto & \Type(T_1){\times}_{\scriptscriptstyle m+n}\Type(T_2)\cr 
  \text{U}{\tt :}_{\scriptscriptstyle m}\Type(T_1){\tt :}_{\scriptscriptstyle n}\Type(T_2) & \longmapsto & \Type(T_1)\,{\circ}_{\scriptscriptstyle m+n}\Type(T_2).\cr
\end{array}
$$
The symbol $\times$ or $\circ$ can only appear once in the name, so the names of $T_1$ and 
$T_2$ do not have to be bracketed.
\end{notation}

\begin{example}
\label{exnot1}
Here are a few examples:

\bigskip\noindent
\begin{center}
\begin{tabular}{|c@{\quad}c|@{\quad}c@{\quad}c@{\quad}|c|}
\hline
\vphantom{$\int^X$}
$T$ (open) & $\Type(T)$ & $\tilde T$ (closed) & $\Type(\tilde T)$ & genus \\
\hline
\vphantom{$\int^X$}
\namexA & 1 \\
\namexD & 2 \\
\namexB & 2 \\
\namexC & 2 \\[2pt]
\hline
\end{tabular}
\end{center}
\end{example}

\begin{example}
In genus 2 the 7 balanced configurations (Table \ref{tabg2a1}) are
\begin{center}
\begin{tabular}{cccccccc}
\TwBn
\end{tabular}
\end{center}
\end{example}

\begin{example}
\label{exnotbig}
The BY tree in Example \ref{Miller} has Type \hetype{:0_{5,5}._1 2:_2(0:_2 U_{6,6})} and its core 
Type \hetype{0_{5,5}._1 2:_2(0:_2 U_{6,6})}.
\end{example}

\begin{remark}
\label{remnames}
All the main invariants of a BY tree $T$ can be seen from the type name:
vertices which are not genus 0 leaves are the capitals I, U, 0, 1, 2, ... in the name. 
The genus 0 leaves are their subscripts; say there are $k$ of them.
Edges not incident to genus 0 leaves are the symbols
$\cdot$, :, $\times$ and $\circ$.
The genus of $T$ is the sum of 
all (non-subscript) numbers in the type plus
\begin{center}
$\rk \Lambda_T$ = \#colons + $k$ - \#`U's - \#`$\circ$'s
\end{center}
and similarly for open BY trees, by ignoring the first symbol if it is `:'.
\end{remark}

\def\exonetyd{U$_f$:(U$_d$:0$_a$:0$_b$):(0$_e$.1$_c$)}

\def\exonepid{\clusterpicture[0.65]
  \Root(1.30,2)(r1);
  \Root(1.80,2)(r2);
  \ClusterL(c1){(r1)(r2)}{$a$};
  \Root(2.50,2)(r3);
  \Root(3.00,2)(r4);
  \ClusterL(c2){(r3)(r4)}{$b$};
  \ClusterL(c3){(c1)(c1n)(c2)(c2n)}{$d$};
  \Root(3.80,2)(r5);
  \Root(4.40,2)(r6);
  \Root(4.90,2)(r7);
  \Root(5.40,2)(r8);
  \ClusterL(c4){(r6)(r7)(r8)}{$c$};
  \ClusterL(c5){(r5)(c4)(c4n)}{$e$};
  \ClusterL(c6){(c3)(c3n)(c5)(c5n)}{$f$};
\endclusterpicture}

\subsection{Automorphisms}

\begin{notation}[Automorphism]
\label{notaut}
Let $T$ be a (possibly open) BY tree, 
and $E=\{e_1,...,e_n\}\subset E(T)$ an ordered subset of its edges,
preserved by $\Aut T$ (or some subgroup that we care about).
By default, we take 
$$
  E=\{\text{all closed edges}\},
$$ 
ordered as follows: 
if $e=\{v_1,v_2\}$, $e'=\{v'_1,v'_2\}$, we check which 
$v\in\{v_1,v_2,v'_1,v'_2\}$ comes last in the name $\Type(T)$ 
(as a capital letter or its subscript); if it is $v'_i$ then $e'$ comes after $e$, and 
vice versa.
%

We then write $\sigma$ as a permutation on the indices of the blue edges and $\pm$indices of
the yellow edges, where the sign of $\pm \sigma(e)$ is determined by $\epsilon(e)$.

%
%
\end{notation}

\begin{example}
\label{exInxIn}
\def\clustersep{1.6pt}
Take a BY tree of genus 2 that corresponds to two nodal genus 0 curves meeting at a point,
with two nodes of the same depth:

\begin{center}
\begin{tabular}{c@{\qquad}ccc@{\qquad}c}
Type & $T$ & $\SS(T)$ & $\GG(T)$ & $\Aut(T)$ \\[2pt]
\pb{\hetype{I_nxI_n}} & \pb{\InxImT} & \pb{\InxImS} & \pb{\InxImG} & $D_4$ \\[2pt]
\end{tabular}
\end{center}

\noindent
Its automorphism group is $D_4$ (order 8). To write its elements we order the edges 
as above:
%
%
\begin{center}
\InxImT\\
$1\quad 2 \quad 3$
\end{center}

%
\noindent
As signed permutations, the elements of $\Aut T$ are 

$$
  \id,\> ({-1}\>{1}),\> ({-3}\>{3}),\> ({-1}\>{1})({-3}\>{3}),
$$
$$
     (13)({-1}\>{-3}),\> ({1}\>{-3})({-1}\>{3}), \> ({1}\>{-3}\>{-1}\>{3}),\>
        ({1}\>{3}\>{-1}\>{-3})
$$
On the corresponding hyperelliptic graph 
$$
  \InxImG 
$$
the element $({-1}\>{1})$ reflects the left loop in the $x$-axis, $({-3}\>{3})$ reflects the right loop, 
and $({1}\>{3}\>{-1}\>{-3})$ sends the left loop to the right one keeping the orientation 
and the right one to the right one reversing the orientation.
(In~this example, it is also reasonable to take $E=\{\text{yellow edges}\}$ instead of all
edges.)
\end{example}

\subsection{BY trees with an automorphism}

For semistable hyperelliptic curves over local fields, it is important to keep
track of the action of Frobenius. Therefore we need a naming convention for 
BY trees with a distinguished automorphism.

\begin{notation}[Type with an automorphism]
We incorporate the action of an automorphism on the edges and the signs into the type 
name, as follows. 

Suppose $T$ is an open BY tree, and $\phi\in\Aut T$.
As in Notation \ref{namingnot}, write
$$
  T=\{v\}\cup\{e\}\cup t_1\cup\ldots\cup t_k\cup T_1\cup...\cup T_n.
$$
We extend the notation 
\begin{center}
  $\Type(T) = [e][v]_{n_1,...,n_k}\Type(T_1)\cdots\Type(T_n)$.\\
\end{center}
to a notation for $\Type(T,\phi)$ as follows:

The automorphism $\phi$ permutes the $t_i$ and the $T_i$, and we rearrange them, 
if necessary, by $\phi$-orbits. In other words, each $\phi$-orbit, say of length $m$, 
is a block $t_i,...,t_{i+m-1}$ or $T_i,...,T_{i+m-1}$.
\begin{enumerate}
\item 
For each $\phi$-orbit $t_i,...,t_{i+m-1}$ replace commas in $n_i,...,n_{i+m-1}$ by 
$\sim$, a symbol for `are in the same $\phi$-orbit'. 
\item
Similarly for each $\phi$-orbit $\Type(T_i)\cdots\Type(T_{i+m-1})$, let $\phi^k$ be the smallest
power of $\phi$ that stabilises $T_i$. Instead of $\Type(T_j)$ write $\Type(T_j,\phi^k)$,
defined inductively, with the first edge symbol `.' or `:' replaced by $\sim$ 
for $j>i$. 
\item
For a closed BY tree (see Notation \ref{typecl}), we similarly replace 
$\hetype{x}, \hetype{o}$ by $\FrobX, \FrobO$ when the endpoints 
of the central edge are swapped by $\phi$.
\end{enumerate}

We decorate the type with signs as follows. For a vertex $v$, let $\phi^{k_v}$ be the smallest power of $\phi$ that stabilizes $v$ and $\epsilon_{v}$ be the sign of $\phi^{k_v}$ on its parent edge should it be yellow.  For each vertex $v$ that is first in its $\phi$-orbit (ordered by appearance in the type):
\begin{itemize}
\item if $v$ is yellow such that its parent edge does not lead to a yellow vertex, decorate the symbol for $v$ in the type name with the superscript $\epsilon_v$,
\item if $v$ is blue, let $w_1,..,w_s$ be the blue vertices joined to $v$  by a yellow edge leading away from $\infty$ (ordered by appearance in the type).  For each $w_i$ that is first in its $\phi^{k_v}$-orbit, decorate the symbol for $v$ in the type name with the superscript $\epsilon_{w_i}$. By convention, these superscripts appear in the same order as the $w_i$'s and are separated by commas.
\end{itemize}
Finally if the open edge is yellow and  incident to a blue vertex, decorate the initial colon with the sign of $\phi$ on the open edge.

In the case of a closed BY tree $\tilde{T}$ with an automorphism $\phi$, define the type $(\tilde{T},\phi)$ to be the type $(T,\phi')$ with the first dot or colon (and their sign) deleted, where $T$ is obtained from $\tilde{T}$ by glueing a yellow open edge to  its center, and $\phi'$ extends $\phi$. 
We use an analogous convention as in Notation~\ref{namingnot} for the cases where the center is an edge. In these cases, we decorate $\circ$ with the sign of the initial $U$ and we write $\tilde{\circ}, \tilde{\times}$ if $T_1$ and $T_2$ are swapped by~$\phi$.

One can check that, in the open or closed case, $(T,\phi)$ and $(T',\phi')$ 
get the same notation if and only if they are isomorphic as pairs, that is
there is an isomorphism $\psi: T\to T'$ such that $\psi\circ\phi=\phi'\circ\psi$.
\end{notation}

\begin{example}[Elliptic curves]
\label{exellc}
Let $T$ be one of the BY trees
$$
\pbox[c]{\textwidth}{\begin{tikzpicture}[scale=\GraphScale]
  \GraphVertices
  \Vertex[x=1.50,y=0.000,L=1]{1};
  \InfVertices
  \Vertex[x=0.000,y=0.000,L=\InfLabel]{2}
  \BlueEdges
  \Edge(1)(2)
\end{tikzpicture}}
\qquad\text{or}\qquad
\pbox[c]{\textwidth}{\begin{tikzpicture}[scale=\GraphScale]
  \BlueVertices
  \Vertex[x=1.50,y=0.000,L=\relax]{1};
  \Vertex[x=3.00,y=0.000,L=\relax]{2};
  \InfVertices
  \Vertex[x=0.000,y=0.000,L=\InfLabel]{3}
  \BlueEdges
  \Edge(1)(3)
  \YellowEdges
  \Edge(1)(2)
\end{tikzpicture}}
$$
The associated cluster pictures are all possible ones of size 3 
(see Table \ref{tabg1a})
$$
\def\GraphScale{0.4}
\def\SnakeWiggle{3pt}
\def\clustersep{1.3pt}
\rootsize{0.5}
\raise-1.6pt\hbox{
\clusterpicture[\clpicscale]
  \Root(2.20,2)(r1);
  \Root(2.42,2)(r2);
  \Root(2.64,2)(r3);
  \Cluster(c1){(r1)(r2)(r3)};
\endclusterpicture}
\qquad\qquad\qquad
\raise-3.2pt\hbox{\clusterpicture[\clpicscale]
  \Root(2.20,2)(r1);
  \Root(2.49,2)(r2);
  \Root(2.71,2)(r3);
  \Cluster(c1){(r2)(r3)};
  \Cluster(c2){(r1)(c1)};
\endclusterpicture}
$$
and they correspond to elliptic curves with good and multiplicative reduction.
If $\phi\in\Aut T$, then $\Type(T,\phi)$ is  
$\hetype{.1}$ in the first case, and $\hetype{.I_n^+}, \hetype{.I_n^-}$ 
in the second case, depending on the $\phi$-action on 
the yellow edge. 
When $\phi$ is Frobenius, $\hetype{.I_n^+}$ is 
`split multiplicative' and $\hetype{.I_n^-}$ `non-split multiplicative' reduction.
If one is only interested in elliptic curves and not general curves of genus 1,
one could omit the first dot and write the types as 
\hetype{1}, \hetype{I_n^+} and \hetype{I_n^-}.
\end{example}

\begin{example}[$\hetype{I_nxI_n}$]
In Example \ref{exInxIn}, for the 5 conjugacy classes of automorphisms 
$\phi\in D_4=\Aut T$ the label $\Type(T,\phi)$ is
$$
  \hetype{I_n^+xI_n^+}, \qquad 
  \hetype{I_n^+xI_n^-}, \qquad 
  \hetype{I_n^-xI_n^-}, \qquad 
  \hetype{I_n^+\FrobX I_n}, \qquad 
  \hetype{I_n^-\FrobX I_n}.
$$
See Table \ref{tabg2bible} for all possible types with an automorphism in genus 2.
\end{example}

\begin{example}
The BY tree with automorphism from Example \ref{Miller} has type  
\hetype{:^+0^{-,+}_{5~5}._1 2:_2(0:_2 U^-_{6,6})}
\end{example}
\section{Tables}\label{s:tables}

Table \ref{tabg3} illustrates the `closed' one-to-one correspondence 
in genus 3. 
In genus 0,1,2 and 3 there are, respectively, 1, 2, 7 and 32 `semistable types', that is
equivalence classes of hyperelliptic graphs/BY trees/cluster pictures 
(cf. Theorem \ref{combmain2}). In genus 0,1 and 2 they are listed in Table \ref{g012table} 
(p. \pageref{g012table}). 
The easiest way to generate them in any genus $g$ is to produce all balanced cluster 
pictures in $X=\{1,...,n\}$ with $n\in\{2g+1,2g+2\}$, up to $S_n$-conjugacy. 

\smallskip

Table \ref{tabg2a1} illustrates the `open' one-to-one correspondence 
(Theorem \ref{combmain1}) in genus 2. In genus 0 and 1, see Table \ref{tabg1a} 
(p. \pageref{tabg1a}). To obtain these, we can list all cluster pictures, 
balanced or not.

\smallskip

Table \ref{tabg2bible}
lists all genus 2 types with an automorphism $\phi$. (In the context of curves over local fields of odd residue characteristic, these 
correspond to all possible Frobenius actions on the dual graph of the special fiber of the minimal regular model of semistable genus 2 curves. 
For elliptic curves, the corresponding types are 1, I$_n^+$ and I$_n^-$ --- good, split multiplicative
and non-split multiplicative reduction; see Example \ref{exellc}.) 
Note that 
by Theorems \ref{combmain2} and \ref{closedcorraut}, there is a bijection between
\begin{itemize}
\item Isomorphism classes of pairs $(\Sigma,\phi)$, where $\Sigma$ is a balanced cluster picture and $\phi\in\Aut\Sigma$ has sign $+1$ on $X$ if $X$ is non-\ub. Here two pairs  $(\Sigma,\phi)$  and  $(\Sigma',\phi')$ are isomorphic if there is an isomorphism $\psi:\Sigma\to\Sigma'$ such that $\psi\phi\psi^{-1}=\phi'$.
\item Isomorphism classes of pairs $(G,\phi)$ of hyperelliptic graphs with an automorphism, where two pairs  $(G,\phi)$  and  $(G',\phi')$ are isomorphic if there is an isomorphism $\psi:G\to G'$ such that $\psi\phi\psi^{-1}=\phi'$.
\end{itemize}
Explicitly, the bijection is given by mapping $\Sigma$ 
to the core $G$ of $\GG(\Sigma)$ and $\phi$ to the restriction of $\GG(\phi)$ to $G$.
This makes it easy to list the types on the level of cluster pictures.
The lattice $\Lambda$, the $\phi$-action on
it and the Tamagawa group can also be computed from it as well
(Definition \ref{lambdaCP}).

\stepcounter{equation}
\begin{table}[t]                             
\hbox{\kern-1.2cm\hbox{\includegraphics[scale=0.9]{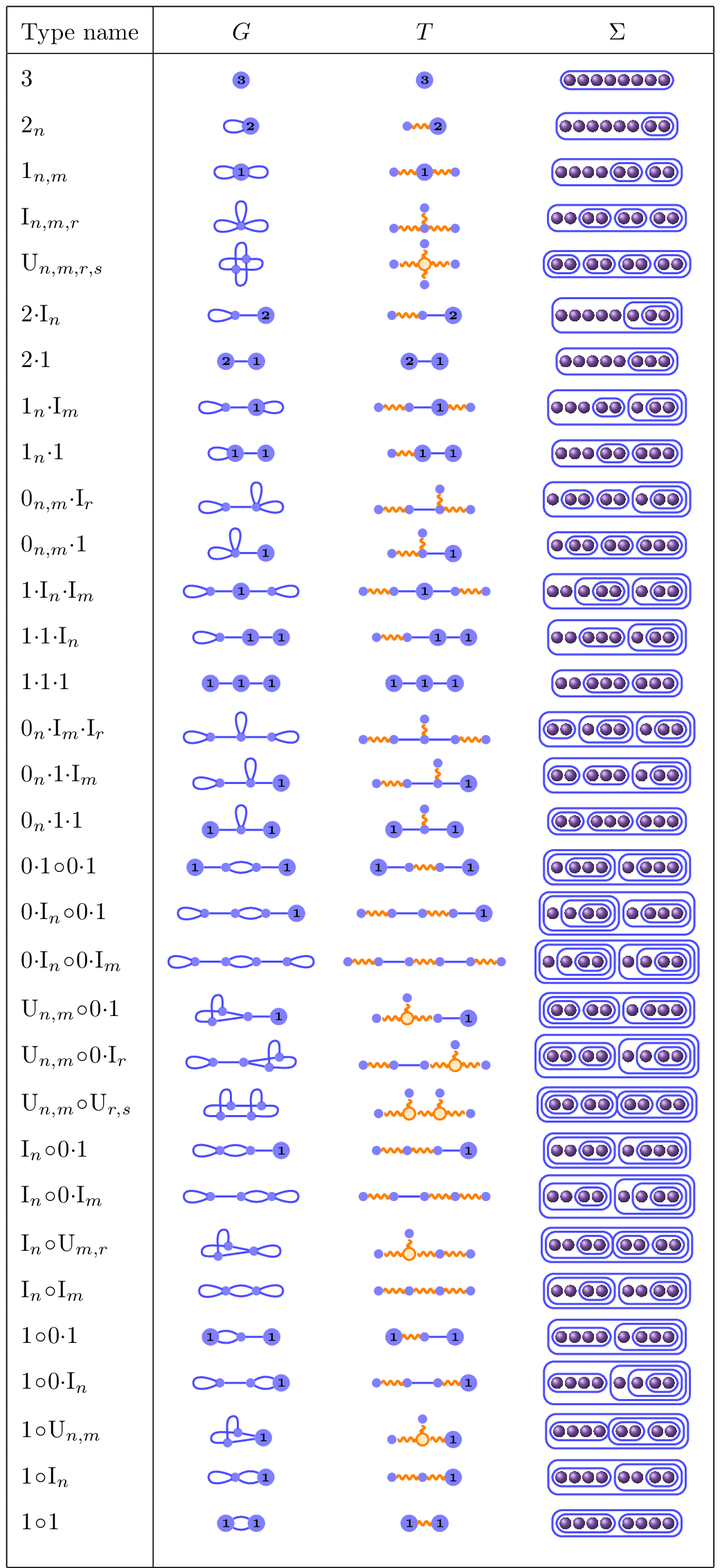}}}
\medskip
\caption{Balanced cluster pictures, hyperelliptic graphs and BY trees in genus 3}
\label{tabg3}
\end{table}

\stepcounter{equation}
\begin{table}[b]                                
\hbox{\kern-0.5cm\hbox{\includegraphics[scale=1.1]{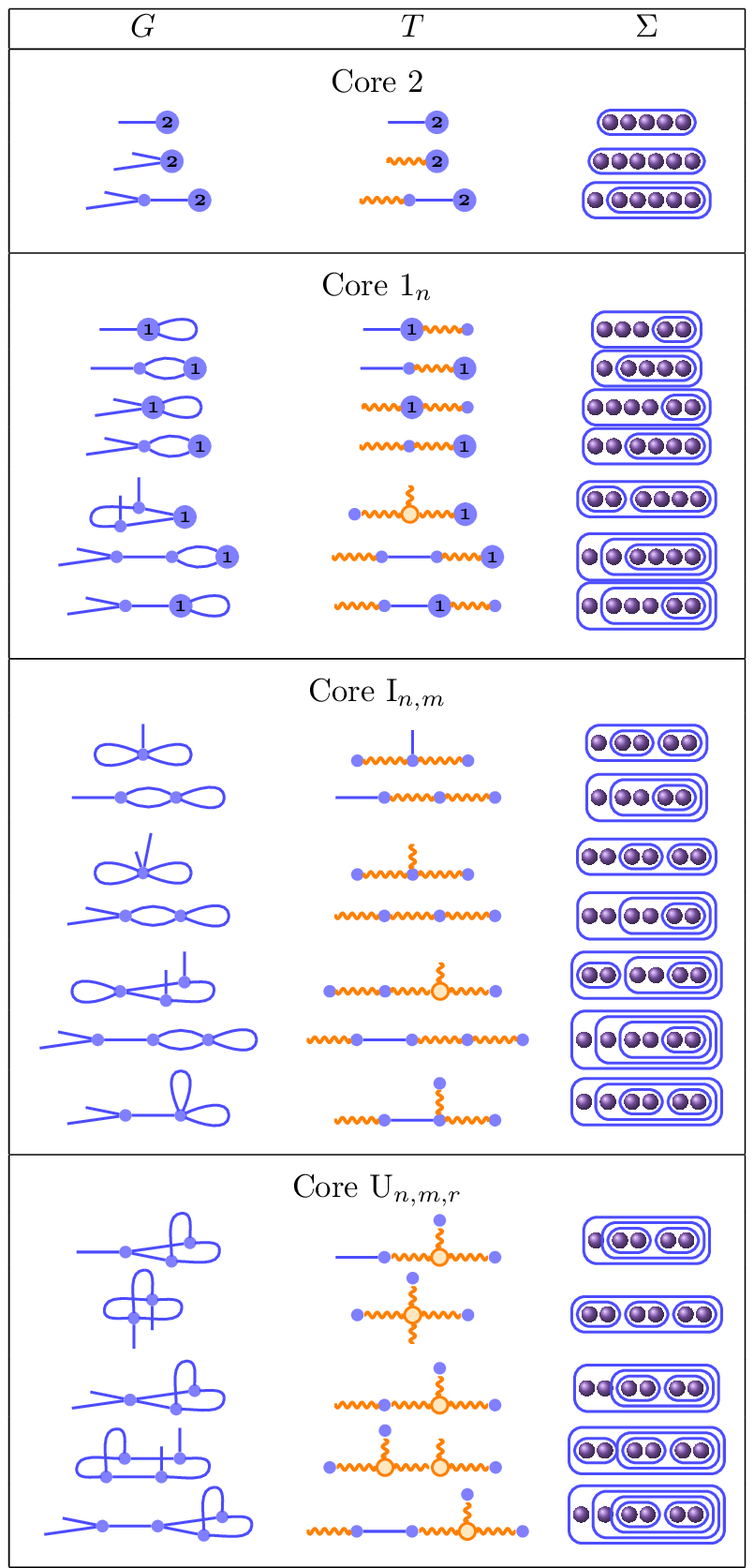}}}
\bigskip
\caption{Cluster pictures, open hyperelliptic graphs and open BY trees up to 
isomorphism in genus 2}
\label{tabg2a1}
\end{table}


\begin{table}[t]                             
\hbox{\kern-0.7cm\hbox{\includegraphics[scale=1.1]{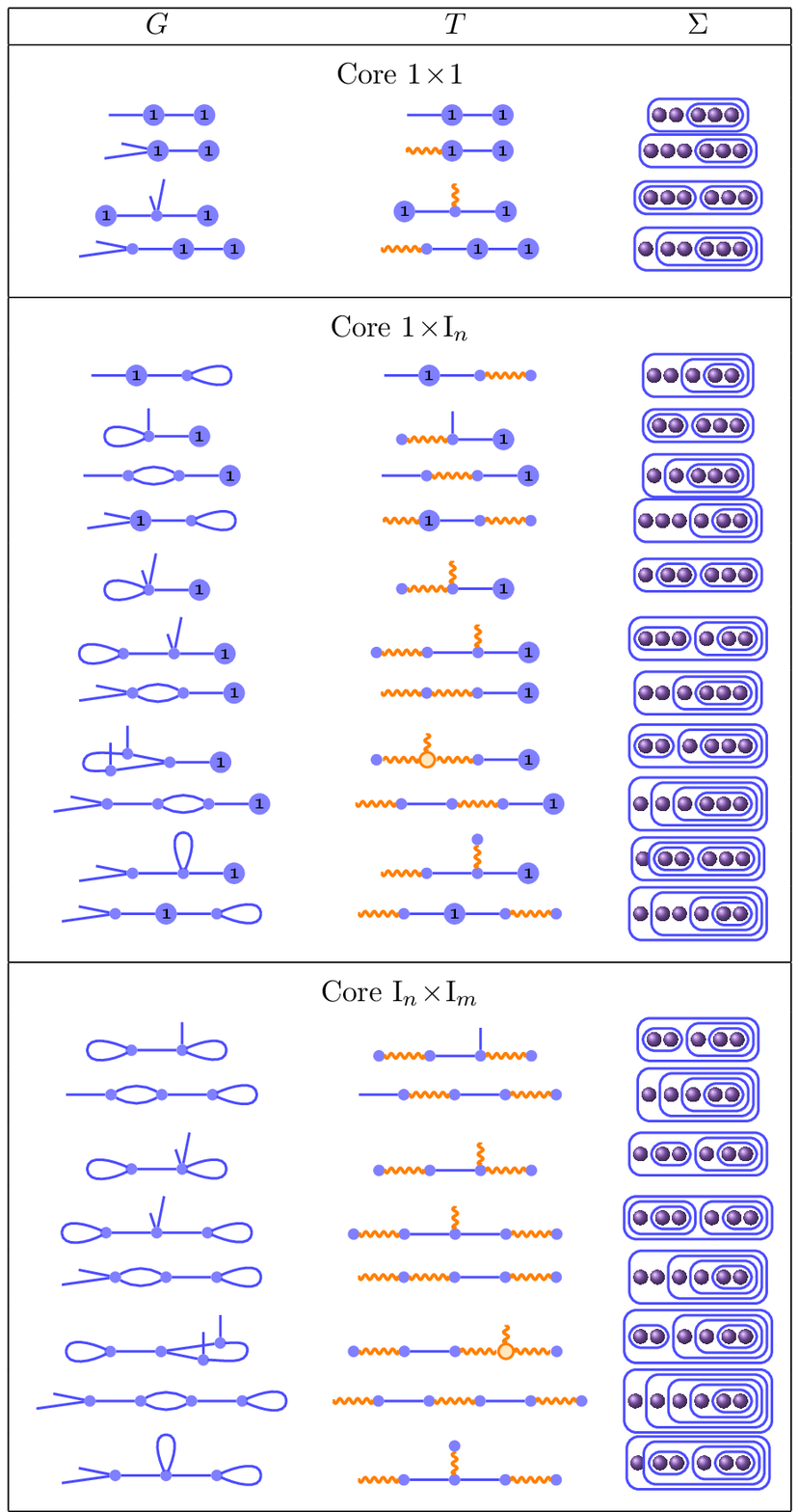}}}
\bigskip
\caption{(continued)}
\label{tabg2a2}
\end{table}

\stepcounter{equation}
\begin{table}[t]                             
\hbox{\kern-1.7cm\hbox{\includegraphics[scale=1]{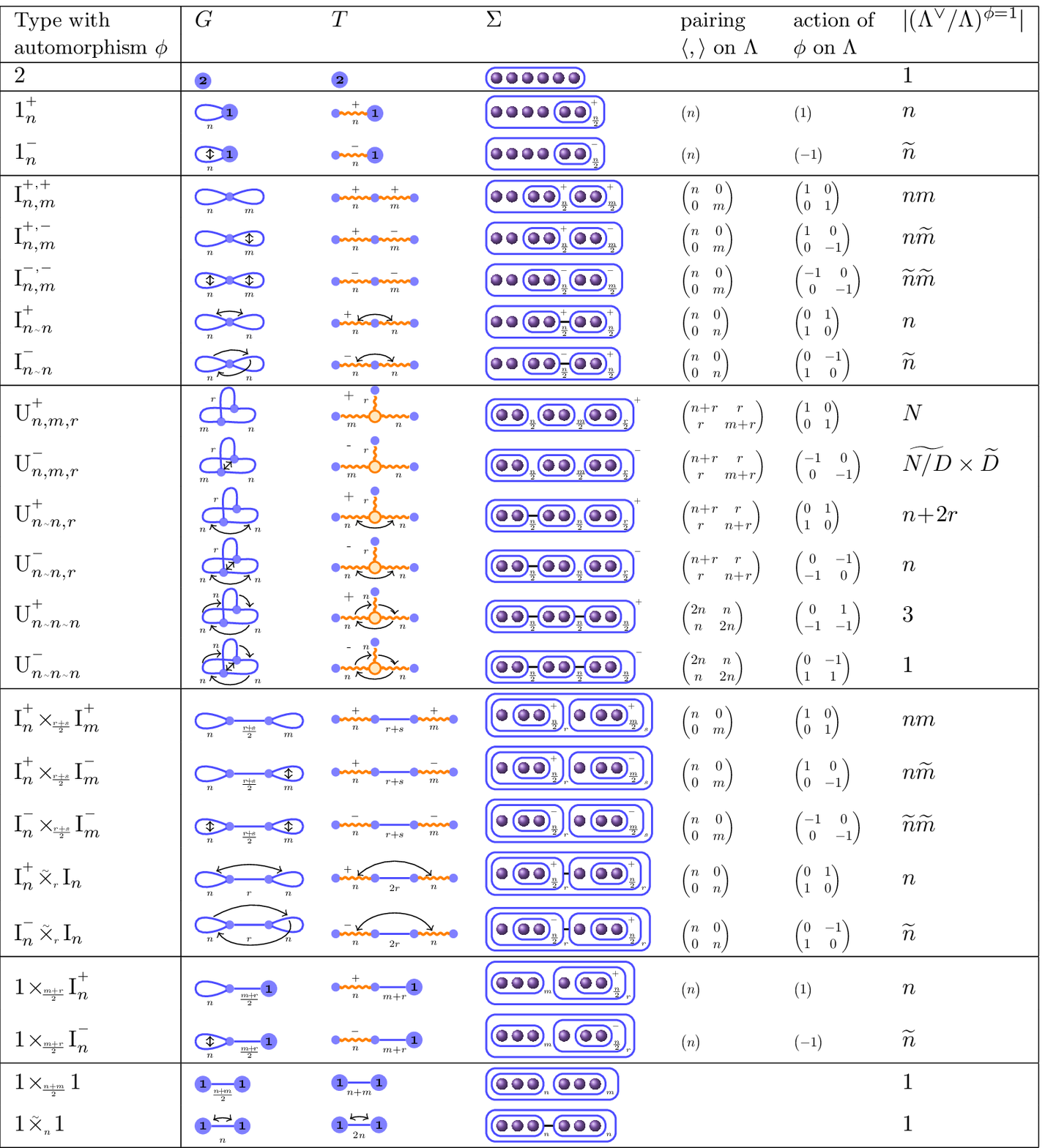}}}
\hbox{\footnotesize \kern-1.7cm\hbox to 10em{Notation in the last column: 
$\tilde n=2$ if $2|n$ and $\tilde n=1$ if $2\nmid n$; $D=\gcd(m,n,r)$; $N=nm+nr+mr$.}}
\hbox{\footnotesize \kern-1.7cm\hbox to 10em{Black arrows in $G$ and $T$ and black lines in $\Sigma$ indicate the automorphism; $+/-$ in $T$ and $\Sigma$ indicate the value of $\epsilon_{\phi}$.}}
\hbox{\footnotesize \kern-1.7cm\hbox to 10em{Numbers indicate lengths of edges in $G$ and $T$, and distances to the parent clusters in $\Sigma$.}}
\bigskip
\caption{Types with an automorphism in genus 2}
\label{tabg2bible}
\end{table}



\end{document}